%% file: main_new.tex
\documentclass[12pt,reqno,a4paper,oneside]{amsart} 
\usepackage{graphicx} 
\usepackage{mathtools}
\usepackage{tikz}
\usepackage{amssymb}
\usepackage{amsmath}
\usepackage{amsthm}
\usepackage{xcolor}
\usepackage{hyperref}
\usepackage{cleveref}
\usepackage{float}
\usepackage{subcaption}
\usepackage[a4paper, margin=2.5cm]{geometry} 
\usepackage{comment}
\usepackage{enumitem}

\usepackage{marginnote}
\usepackage[title]{appendix}

\let\etoolboxforlistloop\forlistloop 
\usepackage{autonum}
\let\forlistloop\etoolboxforlistloop 

\newcommand{\red}[1]{\textcolor{red}{#1}}

\definecolor{canard}{RGB}{0,116,128}
\definecolor{leman}{RGB}{0,167,159}
\newcommand{\norm}[1]{\left\lVert#1\right\rVert}

\renewcommand{\tilde}{\widetilde}

\def\intR{\int_{\mathbb{R}^2}}
\def\Uper{U^{\rm per}}
\def\Ulin{U^{\rm lin}}

\def\RR{\mathbb{R}^2}
\def\R{\mathbb{R}}
\def\N{\mathbb{N}}
\def\PP{\mathbb{P}}
\def\Omegaper{\Omega^{\rm per}}
\def\Omegalin{\Omega^{\rm lin}}

\DeclareMathOperator{\curl}{curl}
\DeclareMathOperator{\im}{im}
\DeclareMathOperator{\Div}{div}
\let\div\relax
\DeclareMathOperator{\div}{div}

\let\Re\relax
\DeclareMathOperator{\Re}{Re}
\let\Im\relax
\DeclareMathOperator{\Im}{Im}

\theoremstyle{plain}

\newtheorem{theorem}{Theorem}[section]

\newtheorem{lemma}[theorem]{Lemma}

\newtheorem{corollary}[theorem]{Corollary}
\newtheorem{proposition}[theorem]{Proposition}

\theoremstyle{definition}
\newtheorem{remark}[theorem]{Remark}

\newtheorem*{theorem*}{Theorem}
\newtheorem*{lemma*}{Lemma}

\newcommand{\jap}{\langle \nabla \rangle}
\newcommand{\p}{\partial}

\title[Non-unique vanishing viscosity solutions]{Non-unique vanishing viscosity solutions to the forced 2D Euler equations}

\numberwithin{equation}{section}

\author[Albritton]{Dallas Albritton} 
\address[Dallas Albritton]{University of Wisconsin-Madison, Department of Mathematics, 480 Lincoln Dr, Madison, WI 53706, USA}
\email{dalbritton@wisc.edu}

\author[Colombo]{Maria Colombo}
\address[Maria Colombo]{Institute of Mathematics, EPFL SB, Station 8, CH 1015 Lausanne, Switzerland}
\email{maria.colombo@epfl.ch}

\author[Mescolini]{Giulia Mescolini}
\address[Giulia Mescolini]{Institute of Mathematics, EPFL SB, Station 8, CH 1015 Lausanne, Switzerland}
\email{giulia.mescolini@epfl.ch}

\date{\today}

\begin{document}

\begin{abstract}
The forced 2D Euler equations exhibit non-unique solutions with vorticity in $L^p$, $p > 1$, whereas the corresponding Navier-Stokes solutions are unique. We investigate whether the inviscid limit $\nu \to 0^+$ from the forced 2D Navier-Stokes to Euler equations is a selection principle capable of ``resolving" the non-uniqueness. We focus on solutions in a neighborhood of the non-uniqueness scenario discovered by~\cite{vishik2018instability1,vishik2018instability2}; specifically, we incorporate viscosity $\nu$ and consider $O(\varepsilon)$ size perturbations of the initial datum. We discover a \emph{uniqueness threshold} $\varepsilon \sim \nu^{\kappa_{\rm c}}$, below which the vanishing viscosity solution is unique and radial, and at which there are viscous solutions converging to non-unique, non-radial solutions.
\end{abstract}

\maketitle
\setcounter{tocdepth}{1}
\tableofcontents

\input{introduction}

\input{spectral}

\input{modifiedbackground}

\input{initiallayer}
\input{nonlinear}
\input{thm_proof}

\appendix
\numberwithin{theorem}{section}

\input{spectralappendix}
\input{forceremarks}

\subsubsection*{Acknowledgments} DA was supported by NSF Grant No. 2406947 and the Office of the Vice Chancellor
for Research and Graduate Education at the University of Wisconsin–Madison with funding from the Wisconsin Alumni Research Foundation. The research of MC and GM was supported by the Swiss State Secretariat for Education, Research and Innovation (SERI) under contract number MB22.00034 through the project TENSE. 


\bibliographystyle{alpha}
\bibliography{bibliography}

\end{document}

%% file: introduction.tex
\section{Introduction}
\label{sec:introduction}

The incompressible Euler equations exhibit non-unique solutions. This is known on the basis of rigorous mathematics~\cite{Scheffer1993,shnirelman1997nonuniqueness,delellis2009euler,isett2018proof,OnsagerAdmissible,Szkelyhidi2011,vishik2018instability1} and supported by numerical evidence~\cite{MR4182316,thalabard2020butterfly}. One therefore seeks to understand whether the non-unique solutions constructed mathematically are somehow ``physical" and in what sense. A particularly high bar is that physically reasonable Euler solutions should arise as vanishing viscosity limits of Navier-Stokes solutions to a well-posed initial value problem. 
One may even hope that vanishing viscosity is a \emph{selection principle} capable of ``resolving" the non-uniqueness, as in scalar conservation laws~\cite{kruvzkov1970first} and the transport equation in the DiPerna-Lions framework~\cite{DiPerna1989}.

In this paper, we investigate this problem for the 2D incompressible Euler equations, driven by a well-chosen radial body force: 
\begin{equation}
\label{eq:euler_1}
\left\lbrace
\begin{aligned}
   \partial_t u + (u\cdot \nabla) u + \nabla p &= \bar{f} \\
\div u &= 0 \, .
\end{aligned}
\right.
\end{equation}
The non-uniqueness scenario we have in mind was introduced by Vishik in~\cite{vishik2018instability1,vishik2018instability2}
and elaborated on in~\cite{albritton2023instability}. We consider the 2D incompressible Navier-Stokes equations
\begin{equation}
\label{eq:NS_forced}
\left\lbrace
\begin{aligned}
   \partial_t u^\nu + (u^\nu \cdot \nabla)u^\nu + \nabla p^\nu &= \nu \Delta u^\nu + \bar{f} \\
\div u^\nu &= 0 
\end{aligned}
\right.
\end{equation}
with the \emph{same} body force $\bar{f}$, which is rough but far into the subcritical regime where~\eqref{eq:NS_forced} is well-posed. 
It is moreover natural to allow for the possibility of some ``observational error" 
in the problem.
In our setting, this will be incorporated into the initial condition\footnote{It is also possible to perturb the force $\bar{f}$; however, we explain in Appendix~\ref{sec:forceperturbations} some ways in which this would trivialize the problem. Additionally, one may wish not to modify the force at all, with a view toward removing the force completely.}
\begin{equation}
    u^\nu|_{t=0} = \bar{u}|_{t=0} + u_0^\nu \, ,
\end{equation}
where $\bar{u}$ is the unique radial solution to~\eqref{eq:euler_1} and $u_0^\nu$ is a perturbation of characteristic size $\varepsilon$ in a certain norm, belonging to the well-posedness regime for both~\eqref{eq:euler_1} and~\eqref{eq:NS_forced}. (Below, $\bar{u}|_{t=0} = 0$, so that $u_0^\nu$ is also the initial datum.) Without $u_0^\nu$, the unique Navier-Stokes solution $u^\nu$ would be radial and converge to the unique radial solution $\bar{u}$. We further suppose that $u_0^\nu$ is $m_0$-fold radially symmetric for a suitable $m_0 \geq 2$. 





Our main conclusion is the existence of a precise threshold $\varepsilon \sim \nu^{\kappa_{\rm c}}$, below which the vanishing viscosity procedure selects a unique, radial solution $\bar{u}$, and at which the procedure produces 
non-unique, non-radial solutions:

\begin{theorem}
    \label{thm:main}
    Let $p \in (2,+\infty)$. There exists a force
    \begin{equation}
    \bar{f} \in L^1_t (L^2 \cap W^{1,p})_x(\R^2 \times (0,1)) \, ,
    \end{equation}
    a norm $\| \cdot \|_{Y_\nu}$ (a dimensionless variant of $\| \cdot \|_{H^{2+s}}$), and a critical exponent $\kappa_{\rm c} > 0$ such that the following hold.
    
    Let $(u_0^\nu)_{\nu \in (0,\nu_0)}$ be a family of initial data and $u^\nu$ be the unique Navier-Stokes solutions with force $\bar f$ and initial data $u_0^\nu$.
    \begin{enumerate}[leftmargin=*]
        \item\label{item:mainthm1} If 
        \begin{equation}
        \| u_0^{\nu} \|_{Y_\nu} = o(\nu^{\kappa_{\rm c}}) \text{ as } \nu \to 0^+ \, ,
        \end{equation}
        then $u^\nu$ converge to the radial solution $\bar{u}$ of \eqref{eq:euler_1} 
        in the inviscid limit:
		\begin{equation}
            \label{eq:convergence1mainthm}
		u^\nu \to \bar{u} \text{ as } \nu \to 0^+ \qquad \text{ in } L^\infty_t (L^2 \cap W^{1,p})_x(\R^2 \times (0,1)) \, .
		\end{equation}

        \item\label{item:parttwotheorem} There exist $c_0,c_1 > 0$ and a particular family of initial data $(u^{\nu}_{\rm in})_{\nu \in (0,\nu_0)}$ with  \begin{equation}
            \| u^{\nu}_{\rm in} \|_{Y_\nu} = c_0 \nu^{\kappa_{\rm c}}
        \end{equation}
        such that if
           $ \| u_0^\nu - u^{\nu}_{\rm in} \|_{Y_\nu} \leq c_1 \nu^{\rm \kappa_{\rm c}}$, then {there exists a subsequence $\nu_k \to 0^+$ such that}
  ${u}^{\nu_k}$ converges to a non-radial solution 
  $u^{\rm E}$ of \eqref{eq:euler_1} (possibly depending on the~sequence): 
        \begin{equation}
            \label{eq:convergence2mainthm}
        {u}^{\nu_k} \to u^{\rm E} \neq \bar{u} \text{ as } k \to +\infty \quad   \text{ in } L^\infty_t (L^2 \cap W^{1,p})_x(\R^2 \times (0,1)) \, . 
        \end{equation}

        
    \end{enumerate}
\end{theorem}

The precise nature of the force $\bar{f}$, the critical exponent
$\kappa_{\rm c}$~\eqref{eq:criticalexponent}, and dimensionless norm~$\| \cdot \|_{Y_\nu}$~\eqref{eq:defdimlessnormynu} are described in Sections~\ref{sec:inviscidnonuniqueness}-\ref{sec:maintheoremrevisited}.

One might consider $\varepsilon \sim \nu^{\rm \kappa_c}$ a \emph{uniqueness threshold} or a \emph{predictability threshold}: It quantifies how much one needs to reduce the observation error to predict a single, deterministic limiting solution $\bar{u}$. Rather than giving a ``yes" or ``no" answer, our result clarifies in a precise way \emph{when} vanishing viscosity is a selection principle, at least in the Vishik scenario. To our knowledge, this result is the first of its type in the PDE theory of incompressible fluids.


By now, there is a well established method, convex integration, to exhibit non-unique weak Euler solutions. Whether \emph{any} of the non-unique solutions exhibited in, e.g.,~\cite{isett2018proof,OnsagerAdmissible,giri20233} can be obtained in the vanishing viscosity limit is an open problem.\footnote{However, non-unique Euler solutions are obtained as inviscid limits of low-regularity (\emph{below} the well-posedness threshold) Navier-Stokes solutions in~\cite{BuckmasterVicolAnnals}.} We solve the analogous problem in the context of Vishik's non-uniqueness scenario~\cite{vishik2018instability1}.

Our methods are based on a precise understanding of the dynamics of~\eqref{eq:euler_1} and~\eqref{eq:NS_forced} near a particular unstable self-similar solution. 
We develop a framework to incorporate a singular perturbation in the context of such solutions. 
The key components are
\begin{itemize}
    \item a precise spectral analysis of the unstable background solution,
    \item an understanding of the regularized but \emph{unperturbed} background,
    \item an approximate controllability argument, based on backward uniqueness, to overcome the viscosity-dominated initial layer, and
    \item a nonlinear instability argument in the regime where the viscosity is negligible.
\end{itemize}
We expect our framework to be broadly applicable to the study of selection principles in the context of self-similarly unstable solutions.

We now describe the picture more precisely. 

\subsection{The inviscid non-uniqueness and the force $\bar{f}$}
\label{sec:inviscidnonuniqueness}


To explain Theorem~\ref{thm:main}, we must recall the non-uniqueness scenario introduced by Vishik in~\cite{vishik2018instability1,vishik2018instability2}. His scenario is centered on a \emph{self-similarly unstable vortex}. 

The Euler equations have a 2-parameter scaling symmetry
\begin{equation}
    \label{eq:scalingsymmetry}
    u_{\ell,\tau}(x,t) := \frac{\ell}{\tau} u \big(\frac{x}{\ell},\frac{t}{\tau} \big) \, , \quad f_{\ell,\tau}(x,t) := \frac{\ell}{\tau^2} f \big(\frac{x}{\ell},\frac{t}{\tau} \big) \, ,  \qquad \forall\ell , \tau > 0 \, ,
\end{equation}
corresponding to the physical dimensions
\begin{equation}
    [x] = L \, , \quad [t] = T \, , \quad [u] = \frac{L}{T} \, , \quad [f] = \frac{L}{T^2} \, .
\end{equation}
We identify a 1-parameter sub-family of~\eqref{eq:scalingsymmetry} by choosing a relationship
\begin{equation}
    \label{eq:relationshipbetweenLandT}
    L^\alpha = T \, ,
\end{equation}
which results in the scaling symmetry
\begin{equation}
    \label{eq:NSscalingsym}
    u_\ell(x,t) := \ell^{1-\alpha} u\big( \frac{x}{\ell}, \frac{t}{\ell^\alpha} \big) \, , \quad f_\ell(x,t) := \ell^{1-2\alpha} f\big( \frac{x}{\ell}, \frac{t}{\ell^\alpha} \big) \, , \qquad \ell > 0 \, .
\end{equation}
It is moreover convenient to consider the vorticity $\omega = \curl u$ and vorticity source $\bar{g} := \curl \bar{f}$. The 2D Euler equations in vorticity formulation are
\begin{equation}
    \p_t \omega + u \cdot \nabla \omega =  \bar{g} \, , \quad u = \nabla^\perp \Delta^{-1} \omega \, .
\end{equation}
The vorticity and its source have dimensions $[\omega] = 1/T$ and $[g] = 1/T^2$, and the associated scaling symmetry is
\begin{equation}
    \label{eq:omegascalingsym}
    \omega_\ell(x,t) := \ell^{-\alpha} \omega\big( \frac{x}{\ell}, \frac{t}{\ell^\alpha} \big) \, , \quad g_{\ell}(x,t) = \ell^{-2\alpha} \big( \frac{x}{\ell}, \frac{t}{\ell^\alpha} \big) \, , \qquad \ell > 0 \, .
\end{equation}

We introduce the self-similarity variables
\begin{equation}
    \xi = \frac{x}{t^{1/\alpha}} \, , \quad \tau = \log t \, ,
\end{equation}
\begin{equation}
    \label{eq:urescalingnotation}
    u(x,t) = t^{1/\alpha-1} U(\xi,\tau) \, , \quad f(x,t) = t^{1/\alpha-2} F(\xi,\tau) \, ,
    \end{equation}
    \begin{equation}
    \omega(x,t) = \frac{1}{t} \Omega(\xi,\tau) \, , \quad g(x,t) = \frac{1}{t^2} G(\xi,\tau) \, ,
\end{equation}
in which the Euler equations become
\begin{equation}
    \label{eq:EulerSSvelocity}
\begin{aligned}
    \p_\tau U + \big(\frac{1}{\alpha} - 1 - \frac{\xi}{\alpha} \cdot \nabla \big) U + U \cdot \nabla U + \nabla P &= F \\
    \div U = 0 \, ,
\end{aligned}
\end{equation}
or, in vorticity formulation,
\begin{equation}
    \label{eq:EulerSS}
    \p_\tau \Omega - \Omega - \frac{\xi}{\alpha} \cdot \nabla \Omega + U \cdot \nabla \Omega = G \, , \quad U = \nabla^\perp \Delta^{-1} \Omega \, .
\end{equation}
$U$ and $\Omega$ are called the velocity and vorticity ``profiles". This change of variables automatically encodes the low regularity features of the problem.

We typically use the convention that lowercase/uppercase variables represent the same function in physical/self-similar variables.


Steady states $\bar{U}$ of~\eqref{eq:EulerSSvelocity} correspond to \emph{self-similar solutions} $\bar{u}(x,t) = t^{1/\alpha-1}\bar{U}(\xi) $ of the Euler equations. In practice, it is difficult to understand self-similar solutions to the unforced Euler equations. However, it is possible to generate self-similar solutions by choosing, for example, a radial vorticity profile $\bar{\Omega}(|\xi|)$ and defining $\bar{g} := \p_t \bar{\omega}$. With the correct choice of profile, the linearized Euler equations around $\bar{U}$ exhibit a growing mode:

\begin{proposition}
\label{prop:lin_part}
    Let $\alpha \in (0,1)$. There exists a radially symmetric $\bar{U} = V(r) e_\theta \in C^\infty_0(\R^2)$ and $m_0 \geq 2$ such that the following holds. The linearized operator
    \begin{equation}
        \label{eq:lssudef}
        L_{\rm ss} : D(L_{\rm ss}) \subset L^2_{m_0,{\rm df}} \to L^2_{m_0,{\rm df}},
    \end{equation}
    where $L^2_{m_0,{\rm df}}$ is the space of $m_0$-fold rotationally symmetric divergence-free vector fields,
    \begin{equation}
        - L_{\rm ss} U := \big(\frac{1}{\alpha} - 1 - \frac{\xi}{\alpha} \cdot \nabla \big) U + \mathbb{P} ( \bar{U} \cdot \nabla U + U \cdot \nabla \bar{U} ),
    \end{equation}
    \begin{equation}
        D(L_{\rm ss}) := \{ U \in L^2_{m_0,{\rm df}} : \big( \bar{U} - \frac{\xi}{\alpha} \big) \cdot \nabla U \in L^2 \},
    \end{equation}
    has an unstable eigenvalue. More specifically, $L_{\rm ss}$ has either
\begin{enumerate}
\item two algebraically simple unstable eigenvalues $\lambda, \bar{\lambda}$ (not identical), or
\item a semisimple eigenvalue $\lambda \in \R$ of multiplicity two.
\end{enumerate}
Moreover,
\begin{equation}
    \sigma(L_{\rm ss}) \setminus \{ \lambda, \bar{\lambda} \} \subset \{ z : \in \mathbb{C} : \Re z \leq 0 \} \, ,
\end{equation}
\begin{equation}
    a := \Re \lambda \geq \frac{4}{\alpha} \, ,
\end{equation}
and the spectral subspace(s) associated to $\lambda$ and $\bar{\lambda}$ are spanned by eigenfunctions $\eta = q(r) e^{im_0 \theta} \in H^{5}(\R^2)$ and $\bar{\eta}$. Finally,
$L_{\rm ss}$ generates a $C^0$-semigroup satisfying
\begin{equation}
    \| e^{\tau L_{\rm ss}} \|_{L^2_{m_0,{\rm df}} \to L^2} \lesssim e^{\tau a} \, \quad \forall \tau \geq 0 \, .
\end{equation}
\end{proposition}


This statement is essentially due to~\cite{vishik2018instability1,vishik2018instability2} (see~\cite[Appendix A]{albritton2023instability}, with truncation in~\cite{MR4429263} and properties of the eigenfunctions in~\cite{MR4616679}). However, it will be important that we keep the simplicity of the eigenvalue, which was not necessary in previous constructions. We also choose to work in velocity rather than vorticity formulation; this is more difficult for the spectral problem but more convenient for the nonlinear argument. We prove Proposition~\ref{prop:lin_part} in Section~\ref{sec:spectral_pb} and Appendix~\ref{sec:spectralperturbationtheory}, which moreover contains spectral theory for a class of singularly perturbed operators.



In Theorem~\ref{thm:main}, we fix $\alpha \in (0,1)$, a vortex profile $\bar{U}$ as in Proposition~\ref{prop:lin_part} (with corresponding $m_0 \geq 2$) and $\bar{u}$ in the notation~\eqref{eq:urescalingnotation}. Notably, $\bar{u}|_{t=0} = 0$. The force $\bar{f} := \p_t \bar{u} \in L^1_t (L^2 \cap W^{1,p})_x$ on finite time intervals provided $1 \leq p < 2/\alpha$. 

We now explain the connection with non-uniqueness. For any $z \in \mathbb C$, the velocity field 
\begin{equation}
    U^{\rm lin} = e^{\tau \lambda} z \eta + \text{ complex conjugate}
\end{equation}
is a solution of the linearized equations $\p_\tau U^{\rm lin} = L_{\rm ss} U^{\rm lin}$ growing $\sim e^{\tau a}$. 
The next step is to construct a non-linear trajectory which tracks the behavior of $U^{\rm lin}$, namely,
\begin{equation}
    U = \bar{U} + U^{\rm lin} + O(e^{2\tau a}) \text{ as } \tau \to -\infty \, .
\end{equation}
Since
\begin{equation}
    \tau \to -\infty \quad \iff \quad t \to 0^+ \, ,
\end{equation}
$u$ is a solution which separates instantaneously from the background solution $\bar{u}$. The separation rate is $\sim t^a$ when measured in critical norms, e.g., $\| \omega(\cdot,t) \|_{L^{2/\alpha}}$, which is invariant under the rescaling~\eqref{eq:omegascalingsym}.

We may regard $U$ as a trajectory on a two-dimensional unstable manifold of~$\bar{U}$, though the existence of such a manifold 
has not yet been proven. For this reason, while the solution $u^{\rm E}$ in Theorem~\ref{thm:main} is expected to be ``one of Vishik's solutions", the available constructions~\cite{vishik2018instability1,albritton2023instability,castro2024proof} are not precise enough for this claim to make sense. Specifically, one should demonstrate that solutions $\bar{U} + U^{\rm lin} + o(e^{\tau a})$ are determined entirely by the $U^{\rm lin}$ contribution. We will clarify this point in forthcoming work.

\subsection{The viscous regularization and the critical exponent $\kappa_{\rm c}$}
We aim to understand the above scenario with viscosity. In similarity variables, the Navier-Stokes momentum equations become
\begin{equation}
    \label{eq:similarityvariablesintroviscous}
\p_\tau U + \big(\frac{1}{\alpha} - 1 - \frac{\xi}{\alpha} \cdot \nabla \big) U + U \cdot \nabla U + \nabla P = 
     \nu e^{-\gamma \tau} \Delta U + F \, ,
\end{equation}
where\footnote{The regime in which $\gamma < 0$ was investigated for the 2D hypodissipative Navier-Stokes equations in~\cite{MR4616679}.}
\begin{equation}
    \gamma := \frac{2}{\alpha} - 1 > 0 \, .
\end{equation}
Let $T_\nu = \nu^{1/\gamma}$ be the ``crossover time". The equation~\eqref{eq:similarityvariablesintroviscous} suggests that the viscosity is \emph{negligible} when $t \gg T_\nu$ ($\nu e^{\gamma \tau} \gg 1$) but \emph{dominant} when $t \ll T_\nu$. This is consistent with the observation that the viscosity restores uniqueness. A key new difficulty will be to account for the behavior in the viscosity-dominated \emph{initial layer}, which we view in the original variables.\footnote{One could go farther and say that viscosity is negligible in the spatial region $\{ |x| \gg \nu^{\alpha/\gamma} \}$, but this level of detail will not be necessary.}

We would like to construct solutions which ``land" on the unstable trajectory $u + u^{\rm lin} + \cdots$ at some time, say, $100 T_\nu$, after which the solution evolves essentially via the Euler equations. \emph{A priori}, it is not evident that this is possible, as the initial layer could ``conspire" to prevent it from happening.

Later, we will see that the solution \emph{is} substantially modified by the inclusion of heat (it moves away by $O(1)$ in critical norms). However, as a zeroth order heuristic, suppose that the action of the initial layer on initial data is to ``do nothing": It maps initial vorticity $\omega_0(x)$ at time zero to $\bar{\omega}(x,t) + \omega_0(x)$ at time $100 T_\nu$, after which we evolve by the Euler equations. At time $100T_\nu$, we hope the solution will be equal to $\bar{\omega} + \omega^{\rm lin}(x,t)$. Therefore, we should have $\omega_0 \approx \omega^{\rm lin}|_{t = T_\nu}$. Since $\| \omega^{\rm lin}|_{t = 100T_\nu} \|_{L^{2/\alpha}} \sim T_\nu^{a} \sim \nu^{a/\gamma}$, we see that the initial data should be size $\varepsilon \sim \nu^{a/\gamma}$ for the construction to hold:
\begin{equation}
    \label{eq:criticalexponent}
    \kappa_{\rm c} = \frac{a}{\gamma} \, .
\end{equation}
The critical exponent is the ratio between the growth rate $a$ and $\gamma = (2-\alpha)/\alpha$, the relative difference between the viscous and self-similar scaling exponents, which sets $T_\nu$.

The ``do nothing" heuristic is not correct in its description of the solution but is correct about its size. Since $t^a$ is the maximal possible growth rate near $\bar{u}$, if $\varepsilon = o(\nu^{a/\gamma})$, then uniqueness will hold in the inviscid limit.

\subsection{The dimensionless norm $Y_\nu$}
\label{sec:maintheoremrevisited}



We now complete the specification of the objects in Theorem~\ref{thm:main}.

Let $\sigma > 0$ such that the convergence in~\eqref{eq:conv_background} in Lemma~\ref{lemma:conv_background} holds in $L^{2+\sigma}$ 
and $0 < s < \min(2-3\alpha/2,1)$ such that $s = 1+2/(2+\sigma)$. By Sobolev embedding, $H^{2+s} \hookrightarrow W^{2,2+\sigma}$.

Since the scaling $L^\alpha = T$ is relevant for the problem, $\nu$ is a dimensional constant with dimensions $[\nu] = L^2/T = L^{2-\alpha}$. Then
\begin{equation}
L_\nu := \nu^{\frac{1}{2-\alpha}} 
\end{equation}
is a viscous length scale.
Define the dimensionless norm
\begin{equation}
    \label{eq:defdimlessnormynu}
\| u_0^\nu \|_{Y_\nu} := L_\nu^{\alpha+s} \| u_0^\nu \|_{\dot H^{2+s}} + L_\nu^{\alpha-2} \| u_0^\nu \|_{L^2} \, ,
\end{equation}
so that $\| (u_0)_{L_\nu} \|_{Y_\nu} = \| u_0 \|_{Y_1}$ in the scaling notation of~\eqref{eq:NSscalingsym}. 

\textbf{Theorem~\ref{thm:main} holds with the choices of $\alpha$, $\bar{U}$, $\bar{f} := \p_t \bar{u}$, $m_0$, $\kappa_{\rm c}$, and $Y_\nu$ specified in Sections~\ref{sec:inviscidnonuniqueness}-\ref{sec:maintheoremrevisited}.} The exponent $p$ in Theorem~\ref{thm:main} is any $p\in(2,2/\alpha)$.

Moreover, in Case (1) (Rigidity), the convergence in~\eqref{eq:convergence1mainthm} is
\begin{equation}
u^\nu \to \bar{u} \text{ in } L^\infty_t W^{1,q}_x(\R^2 \times (0,T)) \, , \quad \forall q \in [2,2/\alpha) \, , \quad \forall T > 0 \, ,
\end{equation}
\begin{equation}
    u^{\nu} \overset{\ast}{\rightharpoonup} \bar{u} \text{ in } L^\infty_t W^{1,\frac{2}{\alpha}}_x(\R^2 \times (0,T)) \, , \quad \forall T > 0 \, .
\end{equation}
In Case (2) (Non-uniqueness), the convergence in~\eqref{eq:convergence2mainthm} is
\begin{equation}
u^{\nu_k} \to u^{\rm E} \text{ in } L^\infty_t W^{1,q}_x(\R^2 \times (0,1)) \, , \quad \forall q \in [2,2/\alpha) \, ,
\end{equation}
\begin{equation}
    u^{\nu_k} \overset{\ast}{\rightharpoonup} u^{\rm E} \text{ in } L^\infty_t W^{1,\frac{2}{\alpha}}_x(\R^2 \times (0,1)) \, .
\end{equation}
That is, the convergence is strong in supercritical spaces and weak in the critical space.

\subsection{Strategy of the proof}

The proof gives a detailed description of the dynamics of $u^\nu$. However, it will be more convenient to describe the dynamics by normalizing to $\nu=1$ according to the rescaling
\begin{equation}
	\label{eq:rescaling}
u^\nu(x,t) = U_\nu u(x/L_\nu, t/T_\nu) \, , 
\end{equation}
where
\begin{equation}
    \label{eq:LnuTnuUnudef}
L_\nu = \nu^{\frac{1}{2-\alpha}} \, , \quad T_\nu = L_\nu^{\alpha} = \nu^{\frac{\alpha}{2-\alpha}} \, , \quad U_\nu = \frac{L_\nu}{T_\nu} = L_\nu^{1-\alpha} = \nu^{\frac{1-\alpha}{2-\alpha}} \, ,
\end{equation}
since $[\nu] = L^{2-\alpha}$. Then $u$ solves the Navier-Stokes equations
\begin{equation}
\left\lbrace
    \label{eq:normalizedviscosityNS}
\begin{aligned}
\p_t u + u \cdot \nabla u + \nabla p &= \Delta u + \bar{f} \\
\div u &= 0 \\
u|_{t=0} &= u_0 \, ,
\end{aligned}
\right.
\end{equation}
where
\begin{equation}
u_0^\nu(x) = U_\nu u_0 (x/L_\nu) \, .
\end{equation}
It will be more convenient to \emph{define} $u_0$, \emph{describe} $u$, and finally unfold the rescaling~\eqref{eq:rescaling} to recover $u_0^\nu$ and the solution $u^\nu$ afterward. Notably,
\begin{align}
\label{eq:norms_relation}
\| u_0^\nu \|_{\dot H^{2+s}}  =  L_\nu^{- s - \alpha}  & \| u_0 \|_{\dot H^{2+s}}, \quad  \| u_0^\nu \|_{L^p} =  L_\nu^{1 + \frac{2}{p} - \alpha}  \| u_0 \|_{L^p}, \\ &
\| \omega_0^\nu \|_{L^p} = L_\nu^{\frac{2}{p} -\alpha} \| \omega_0 \|_{L^p} \, ,
\end{align}
from which we obtain $\| \omega_0^\nu \|_{L^{2/\alpha}} = \| \omega_0 \|_{L^{2/\alpha}}$ ($\| \cdot \|_{L^{2/\alpha}}$ is a \emph{critical norm} for the vorticity).





From this perspective, the problem of controlling $u^\nu$ until time $O(1)$ becomes a problem of controlling $u$ until (long) time $O(T_\nu^{-1})$. The ``initial layer" lasts until time $O(1)$ in the new variables. 

For the remainder of the introduction, we will focus on the non-uniqueness statement (2), as the corresponding uniqueness statement (1) follows by a similar but simpler analysis. Additionally, we regard the background vortex as \emph{fixed}, and various implied constants may depend on it.

\subsubsection{The modified background}
\label{sec:intromodifiedbackground}
To understand the long-time behavior of \eqref{eq:normalizedviscosityNS}, we first need to understand the following: What happens to the background $\bar u$ itself under viscosity? Hypothetically, if the viscosity were to destabilize $\bar{u}$ in a strong way, then solutions would move away from the inviscid construction, and we would lose control of the non-uniqueness generated by the instability of such background.
 
To quantify this possibility, we consider the solution $\tilde{u}$ to the forced Navier-Stokes equations \eqref{eq:normalizedviscosityNS}
with initial datum $0$. Namely, $\tilde u$ solves the equation of $\bar u$ with the sole addition of the Laplacian term.
We call $\tilde{u}$ the background modified by viscosity, or simply the \emph{modified background}. In our context, $\tilde{u}$ is radial and solves the forced heat equation
$$\partial_t \tilde{u} = \Delta \tilde{u} + \bar{f} \quad \text{ in } \RR \times \R^+ $$
with initial datum $0$.

The whole construction is in a neighborhood of the modified background, so it is crucial to understand $\tilde{u}$ in detail. This is the subject of Section~\ref{sec:modifiedbackground}. The main difficulty is in the long-time behavior of $\tilde{u}$. 
Therefore, the crucial observation is that the modified background asymptotically returns to the inviscid background: 
\begin{equation}
\label{eqn:backgrconv}
    \tilde{U} \to \bar{U} \text{ in } H^{2+s} \text{ as } \tau \to +\infty \, .
\end{equation}
That is, the background $\bar{u}$ is somewhat ``stable" with respect to viscosity. In supercritical regularity, e.g., $\Omega \in L^p$, $p <2/\alpha$, convergence holds due to the scaling, but this is not refined enough for our purposes. A more substantial analysis, involving a further cancellation, is required to establish convergence in $H^{2+s}$, which will be necessary for our method. 
Actually, we expect the convergence \eqref{eqn:backgrconv} to be false in higher regularity, for instance already in $H^3$, since the presence of the self-similar transport term $- (1 + \frac{1}{\alpha} \cdot \nabla)$ in the equation for $\tilde U$ wants to concentrate the solution at the origin, making its derivatives grow.


\subsubsection{The initial layer}

We next consider how perturbations behave as they pass through initial layer. At this level, we rewrite
\begin{equation}
    u = \tilde{u} + {\rm perturbation} \, ,
\end{equation}
where the perturbation satisfies a Navier-Stokes problem with lower order terms and initial data $u_0 = O(\varepsilon)$. The perturbation is therefore accurately described by $\varepsilon v$, where $v$ is a solution to the linearized Navier-Stokes equations~\eqref{eq:initialayer1} around $\tilde{u}$. Generally speaking, it seems too difficult to say anything substantive about $v$ besides estimating its growth. What is important for us is that \emph{the image of the solution map for the linearized problem~\eqref{eq:initialayer1} is dense in $H^{2+s}$}; this might be regarded as a form of approximate controllability. We can therefore choose $v$ such that perturbation is close to $\varepsilon u^{\rm lin}$, a growing solution of the linearized Euler equations, at time~$T$. 



\begin{theorem}[The initial layer]
\label{thm:initiallayer}
Let $s \in (0,1)$ and $\tilde{u} \in L^2_t H^{2+s}_x(\R^2 \times (0,T))$ be a divergence-free vector field. Let $T, \delta > 0$ and $\phi \in H^{2+s}_{\rm df}$ be a desired end state.
\begin{enumerate}[label=(\roman*)]
\item\label{item:initialayer1} There exists $v_0 \in H^{2+s}_{\rm df}$ such that the solution $v \in C([0,T];H^{2+s})$ to the linearized problem
\begin{equation}
\label{eq:initialayer1}
\left\lbrace
\begin{aligned}
    \partial_t v + \mathbb{P}(\tilde{u} \cdot \nabla v + v \cdot \nabla \tilde{u}) &= \Delta v \\
    \div v &= 0 \\
    v|_{t=0} &= v_0
\end{aligned}
\right.
\end{equation}
satisfies
\begin{equation}
    \label{eq:closetounstablemode}
    \| v(\cdot,T) - \phi \|_{H^{2+s}} \leq \delta \, .
\end{equation}
\item\label{item:initialayer2} For all $0 \leq \varepsilon \ll_{\delta,T,\phi} 1$ and $\psi_0 \in H^{2+s}_{\rm df}$ such that $\|\psi_0\|_{H^{2+s}} \ll_{\delta,T,\phi} 1$, the solution $u \in C([0,T];H^{2+s})$ to the perturbed Navier-Stokes equations\footnote{Under the assumptions,~\eqref{eq:initialayer1} and~\eqref{eq:initialayer2} are both uniquely solvable in $C([0,T];H^{2+s})$ -- see Section~\ref{sec:back_uniq}.}
\begin{equation}
\label{eq:initialayer2}
\left\lbrace
\begin{aligned}
    \partial_t u + \mathbb{P}(\tilde{u} \cdot \nabla u + u \cdot \nabla \tilde{u} + u \cdot \nabla u) &= \Delta u \\
    \div u &= 0 \\
    u|_{t=0} &= \varepsilon v_0 + \psi_0
\end{aligned}
\right.
\end{equation}
satisfies for some $C:=C(T, \delta, \phi)$
\begin{equation}
    \| u - \varepsilon v \|_{L^\infty_t H^{2+s}_x(\R^2 \times (0,T))} 
    \leq C(\varepsilon^2 + \|\psi_0\|_{H^{2+s}} )\, .
\end{equation}
\end{enumerate}
\end{theorem}

This is proved in Section~\ref{sec:back_uniq} by \emph{backward uniqueness} for an adjoint equation (specifically, the $L^2$-adjoint of the equation for the unknown $\langle \nabla \rangle^{2+s} u$). 

The open set of initial data $\tilde{u_0}^\nu$ in Theorem~\ref{thm:main} is, roughly speaking, defined by the property~\eqref{eq:closetounstablemode}, that the linearized solution land close to the growing mode at time $T$. 

\subsubsection{Long-time behavior}
\label{sec:introlongtime}

It remains to control the solution after the initial layer, once the viscosity is negligible. In this setting, we decompose the solution into
\begin{equation}
    U = \tilde{U} + U^{\rm lin} + U^{\rm per}
\end{equation}
where
\begin{equation}
    U^{\rm lin} = e^{(\tau-\tau_0) L_{\rm ss}} P_{\lambda,\bar{\lambda}} (U-\tilde{U})|_{\tau=\tau_0} = O(\varepsilon e^{a \tau}) \, ,
\end{equation}
and $P_{\lambda,\bar{\lambda}}$ is the spectral projection onto the spectral subspaces associated with $\lambda$ and $\bar{\lambda}$.

We aim to control the solution until $U - \bar{U}$ becomes $O(1)$, that is, when the solution exits the linearized regime. Assuming that $U^{\rm per}$ is negligible, this happens at
\begin{equation}
    \tau_{\rm max} = \frac{1}{a} \log \frac{1}{O(1) \varepsilon} \, ,
\end{equation}
that is, when $U^{\rm lin}$ becomes $O(1)$. Our goal is to control $U^{\rm per}$ (more specifically, we keep it smaller than $U^{\rm lin}$) until $\tau_{\rm max}$. 

\begin{theorem}[Long-time behavior]
\label{thm:longtimebehavior}
We adopt the above notation. For every $c \in (0,1)$, there exists $C, \tau_0 > 0$ (depending on $c$) such that the following holds.

Let 
$\Phi \in H^{2+s}$ be a divergence-free initial datum
and $U$ be the solution of Navier-Stokes equations in self-similar variables with force $\bar F$, namely,~\eqref{eq:similarityvariablesintroviscous} with $\nu=1$, and initial condition
\begin{equation}
    U|_{\tau = \tau_0} = \tilde U|_{\tau=\tau_0} + \Phi \, .
\end{equation}
Then
\begin{align}
\label{eq:nonlinear_bound}
    & \| U - \tilde{U} - e^{ (\tau-\tau_0)L_{\rm ss}} P_{\lambda, \bar \lambda} \Phi \|_{W^{2,2+\sigma}} \\ & \quad  \leq e^{a(\tau-\tau_0)} \left( c\| P_{\lambda,\bar{\lambda}} \Phi \|_{W^{2,2+\sigma}} + C \| \Phi -  P_{\lambda,\bar{\lambda}} \Phi\|_{W^{2,2+\sigma}} \right)
\end{align}
whenever
\begin{equation}
\tau \in \left[ \tau_0, \frac{1}{a} \log \frac{1}{C\|\Phi\|_{W^{2,2+\sigma}}} \right]  \, .  
\end{equation}
\end{theorem}


The perturbation $U^{\rm per}$ satisfies
\begin{equation}
    \label{eq:uperequationforintro}
\begin{aligned}
\partial_\tau \Uper & -L_{\rm ss} \Uper = - \mathbb{P}((\tilde{U} - \bar{U}) \cdot \nabla \Uper) - \mathbb{P}((\tilde{U} - \bar{U}) \cdot \nabla \Ulin) \\ 
& \quad -\mathbb{P} (\Uper \cdot \nabla (\tilde{U} - \bar{U}) + \Uper \cdot \nabla \Uper) \\ & \quad + \mathbb{P}(\Uper \cdot \nabla \Ulin + \Ulin \cdot \nabla \Uper) +\mathbb{P}(\Ulin \cdot \nabla (\bar{U} - \tilde{U}) - \Ulin \cdot \nabla \Ulin) \\ & \quad + e^{-\tau \gamma} \Delta \Ulin + e^{-\tau \gamma} \Delta \Uper \, ,
\end{aligned}
\end{equation}
which is driven by various sources of error and initialized with the initial condition
\begin{equation}
U^{\rm per}|_{\tau=\tau_0} = \Phi - P_{\lambda,\bar{\lambda}} \Phi = O(\varepsilon \delta) \, .
\end{equation}
The analysis of this problem has the most in common with the nonlinear instability arguments in~\cite{vishik2018instability1,albritton2023instability,MR4616679}, which themselves draw from, e.g.,~\cite{bardosguostrauss,friedlanderstraussvishikearly,friedlandervishik}. However, we must also overcome difficulties unique to our setting. These will discussed further in Section~\ref{sec:nonlinear}. For now, we mention the following:
\begin{itemize}[leftmargin=*]
    \item Velocity is more convenient than vorticity formulation for the (i) ``derivative loss" from $e^{-\tau \gamma} \Delta U^{\rm per}$ in the bootstrap argument, and (ii) low frequencies in the Biot-Savart law.
    \item The semisimplicity (no Jordan blocks) of $\lambda$ is necessary to ensure that the right-hand side does not ignite any contributions which grow faster that $e^{\tau a}$, the growth of $U^{\rm lin}$.
    \item The convergence $\tilde{U} \to \bar{U}$ ensures that terms such as $U^{\rm per} \cdot \nabla (\bar{U} - \tilde{U})$ are negligible. However, the convergence is only in $W^{2,2+\sigma}$ (expected to be false in $W^{2,\infty}$), so the term $U^{\rm per} \cdot \nabla [\nabla (\bar{\Omega} - \tilde{\Omega})]$ in the equation for $\nabla \Omega^{\rm per}$ is not evidently lower order. To overcome this, we ``share" one derivative with $U^{\rm per}$ through the weight $1/r$.
\end{itemize}











\subsubsection{Conclusion}
We now explain how to prove Theorem~\ref{thm:main} or, more specifically, how to exhibit \emph{one} non-unique solution in the inviscid limit. Fix\footnote{One should make a choice, since $\eta$ is only defined up to multiplication by $z \in \mathbb{C} \setminus \{ 0 \}$.}
\begin{equation}
    \label{eq:ulindefforconclusion}
    U^{\rm lin} = \Re (e^{\tau \lambda} \eta(\xi)) \, , \quad  u^{\rm lin} = t^{\frac{1}{\alpha}-1} \Re \left[ t^\lambda \eta \Big( \frac{x}{t^{\frac{1}{\alpha}}}\big)  \right] \, . 
\end{equation}
As discussed above~\eqref{eq:normalizedviscosityNS}, it will be convenient to work in variables $(y,s)$ with normalized viscosity. For $|\varepsilon| \ll 1$, let $w^\varepsilon(y,s)$ be the solution to the Navier-Stokes equations with $\nu=1$:
\begin{equation}
\left\lbrace
\begin{aligned}
    \p_s w^\varepsilon + w^\varepsilon \cdot \nabla w^\varepsilon + \nabla p^\varepsilon &= \Delta w^\varepsilon + \bar{f} \\
    \div w^\varepsilon &= 0 \\
    w^\varepsilon|_{s=0} &= \varepsilon v_0 \, ,
\end{aligned}
\right.
\end{equation}
where $v_0$ in the initial datum in Theorem~\ref{thm:initiallayer}, whose solution to the linearized PDE satisfies $v(\cdot,S) = u^{\rm lin}(\cdot,S) + O(\delta)$ at fixed $S \gg 1$, which we think of as beyond the initial layer.\footnote{For reference, the parameters are chosen in the following order: $c$, $\tau_0$ in Theorem~\ref{thm:longtimebehavior} before $S := T = \log \tau_0$, $\delta$ in Theorem~\ref{thm:initiallayer}.} 
Theorem~\ref{thm:initiallayer} grants us the characterization
\begin{equation}
    w^\varepsilon(\cdot,S) = \tilde{u} + \varepsilon u^{\rm lin} + O(\varepsilon \delta) \, .
\end{equation}



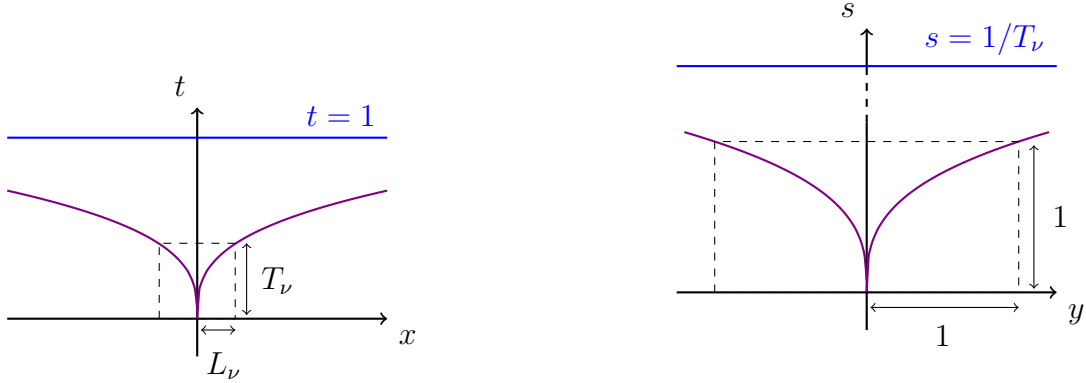
\begin{figure}
    \centering
    \begin{minipage}{0.45\textwidth}
    \vspace{1.4cm}  
    \centering
    \begin{tikzpicture}[scale=1.0]
        \draw[thick,->] (-2.5,0) -- (2.5,0) node[anchor=north west] {$x$}; 
        \draw[thick,->] (0,-0.5) -- (0,2.8) node[anchor=south east] {$t$};  
        \draw[thick,-,blue] (-2.5,2.4) -- (2.5,2.4) node[anchor=south east] {$t=1$};  
        \draw[domain=0:2.5,samples=100,smooth,variable=\x,thick,violet] 
            plot ({\x},{1.4^(2/3)*\x^(1/3)});
        \draw[domain=-2.5:0,samples=100,smooth,variable=\x,thick,violet] 
            plot ({\x},{-1.4^(2/3)*\x^(1/3)});

        \draw[dashed] (-0.5,0) rectangle (0.5,1);

        \node at (0.7,0.5) [anchor=west] {$T_\nu$};

        \node at (0.35,-0.3) [anchor=north] {$L_\nu$};
        \draw[<->] (0.05, -0.15) -- (0.5, -0.15);
        \draw[<->] (0.65, 0.05) -- (0.65, 1);
        
    \end{tikzpicture}
    \end{minipage}
    \hfill
    \begin{minipage}{0.45\textwidth}
    \centering
    \begin{tikzpicture}[scale=1]
        \draw[thick,->] (-2.5,0) -- (2.5,0) node[anchor=north west] {$y$}; 
        \draw[thick,-] (0,-0.5) -- (0,2.25);
        \draw[thick,->] (0,3) -- (0,3.5) node[anchor=south east] {$s$};  
        \draw[thick,dashed] (0,2.25) -- (0,3);
        \draw[thick,-,blue] (-2.5,3) -- (2.5,3) node[anchor=south east] {$s=1/T_\nu$};  
        
        \draw[domain=0:2.4,samples=100,smooth,variable=\x,thick,violet] 
            plot ({\x},{2^(2/3)*\x^(1/3)});
        \draw[domain=-2.4:0,samples=100,smooth,variable=\x,thick,violet] 
            plot ({\x},{-2^(2/3)*\x^(1/3)});

        \draw[dashed] (-2,0) rectangle (2,2);

        \node at (2.3,1) [anchor=west] {$1$};

        \node at (1,-0.3) [anchor=north] {$1$};

        \draw[<->] (0.05, -0.2) -- (2, -0.2);
        \draw[<->] (2.2, 0.05) -- (2.2, 1.95);
    \end{tikzpicture}
    \end{minipage}
        \caption{Visualization of the transformation~\eqref{eq:zoominginbruh}. $u^\nu(x,t)$ (left) solves the Navier-Stokes equations with viscosity $\nu$, and we observe it at time $t=1$. $w(y,s)$ (right) solves the Navier-Stokes equations with unit viscosity, and we observe it until time $s=1/T_\nu$. The purple curve $\{ |\xi| = {\rm const.}\}$ remains unchanged.}
    \label{fig:rescalingfig}
\end{figure}


Subsequently, Theorem~\ref{thm:longtimebehavior} grants control on $w^{\varepsilon}$ for $s \geq S$ in a decomposition
\begin{equation}
    \label{eq:decompositionbeforerescaling}
    w^{\varepsilon} = \tilde{u} + \varepsilon u^{\rm lin} + \underbrace{w^{\rm per,\varepsilon}}_{O(s^a \varepsilon \delta)} \quad \text{ on } [S,s_{\rm max}] \, ,
\end{equation}
where $s_{\rm max}(\varepsilon) = \varepsilon^{-\frac{1}{a}} / O(1)$ and the big Oh is measured in critical spaces. 
Choose $\varepsilon = \nu^{\frac{a}{\gamma}} = T_\nu^a$, so that $s_{\rm max} = T_\nu^{-1}/O(1)$. (Remember: $T_\nu = \nu^{\frac{1}{\gamma}}$.) Then the big Oh in~\eqref{eq:decompositionbeforerescaling} is, more specifically, the estimate
\begin{equation}
    s^{1-\frac{2}{\alpha}} \| w^{\rm per,\varepsilon} \|_{L^2} +  \| \nabla w^{\rm per,\varepsilon} \|_{L^{2/\alpha}} +  s^{1+\frac{s_0}{\alpha}} \| \nabla^2 w^{\rm per,\varepsilon} \|_{L^{2+\sigma}} \lesssim O(s^{a} \varepsilon \delta) \, ,
\end{equation}
for all $s \in [S, T_\nu^{-1} / O(1)]$. (The powers of $s$ can be obtained from~\eqref{eq:defdimlessnormynu}.) 



We now rescale according to~\eqref{eq:rescaling}:
\begin{equation}
    \label{eq:zoominginbruh}
    u^\nu(x,t) := U_\nu w^\varepsilon(x/L_\nu, t/T_\nu) \, ,
\end{equation}
and similarly with $\tilde{u}^\nu$, $u^{\rm lin,\nu}$. This transformation is visualized in \Cref{fig:rescalingfig}. 
The decomposition~\eqref{eq:decompositionbeforerescaling} becomes
\begin{equation}
    \label{eq:decompositiongetsrescaled}
    u^\nu = \tilde{u}^\nu + \varepsilon u^{\rm lin,\nu} + O(t^a \delta) \, , \quad \forall t \in [T_\nu S, 1/O(1)] \, ,
\end{equation}
valid after the initial layer, which is shrinking as $\nu \to 0^+$. The behavior of $u^\nu$ is represented in Figure~\ref{fig:figure_solutions}. The key point is therefore to analyze the convergence of~\eqref{eq:decompositiongetsrescaled} as $\nu \to 0^+$. We have, in various topologies,
\begin{equation}
    \tilde{u}^\nu \to \bar{u}
\end{equation}
\begin{equation}
    \label{eq:differentphasesrule}
    \varepsilon u^{\rm lin,\nu} = T_\nu^a \times t^{\frac{1}{\alpha}-1} \Re \left[ \Big( \frac{t}{T_\nu} \Big)^{a+ib} \eta \Big( \frac{x}{t^{\frac{1}{\alpha}}}\Big)  \right] = u^{\rm lin}
\end{equation}
provided we choose a subsequence $\nu_k \to 0^+$ so that $T_{\nu_k}^{ib} = 1$. This guarantees
\begin{equation}
    u \to u^{\rm E} = \bar{u} + u^{\rm lin} + O(t^a \delta) \text{ on } \R^2 \times (0,1/O(1)) \, .
\end{equation}
By compactness, the solution is a weak solution of the forced Euler equations. Finally, since $\delta \ll 1$, $u^{\rm E}$ is not identically equal to $\bar{u}$.

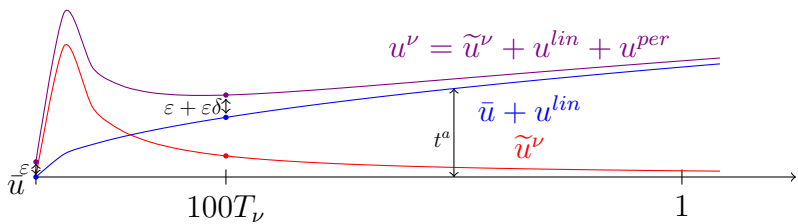
\begin{figure}
    \centering

\begin{tikzpicture}[scale=1]

\draw[->] (0,0) -- (10,0) node[right] {};
\draw (0,-0.1) node[left] {$\bar u$};
\draw (0,0.1) -- (0,-0.1);
\draw (2.5,0.1) -- (2.5,-0.1) node[below] {$100 T_\nu$};
\draw (8.5,0.1) -- (8.5,-0.1) node[below] {$1$};

\draw[domain=0:9,smooth,variable=\x,blue] plot ({\x},{0.5*sqrt(\x)});
\draw (6.5,0.9) node[blue] {$\bar u + u^{lin}$};

\fill[blue] (2.5, {0.5*sqrt(2.5)}) circle (1pt);

\draw[domain=0:9,smooth,variable=\x,red] plot ({\x},{7* (\x/0.1)/(1+(\x/0.1)^2)});
\draw (6.5,0.4) node[red] {$\tilde u^\nu$};

\fill[red] (2.5, {7* (2.5/0.1)/(1+(2.5/0.1)^2)}) circle (1pt);

\draw[domain=0:9,smooth,variable=\x,violet] 
    plot ({\x},{0.2*exp(-\x) + 0.5*sqrt(\x) + 7* (\x/0.1)/(1+(\x/0.1)^2)});
\draw (6.5,1.75) node[violet] {$u^\nu =  \tilde u^\nu + u^{lin} + u^{per} $};

\fill[violet] (2.5, {0.2*exp(-2.5) + 0.5*sqrt(2.5) + 7* (2.5/0.1)/(1+(2.5/0.1)^2)}) circle (1pt);

\fill[violet] (0, {0.2*exp(0) + 0.5*sqrt(0) + 7* (0/0.1)/(1+(0/0.1)^2)}) circle (1pt);

\fill[blue] (0, 0) circle (1pt);

\draw[<->] (2.5, {0.5*sqrt(2.5)+0.05}) -- (2.5, {0.2*exp(-2.5) + 0.5*sqrt(2.5) + 7* (2.5/0.1)/(1+(2.5/0.1)^2)-0.05});

\draw (2.07,0.93) node[black, scale=0.7] {$\varepsilon + \varepsilon \delta$};

\draw[<->] (5.5, 0) -- (5.5, {0.5*sqrt(5.5)});
\draw (5.35, 0.5) node[black, scale=0.7] {$t^a$};

\draw[<->] (0, 0.03) -- (0, 0.175);
\draw (-0.15, 0.1) node[black, scale=0.7] {$\varepsilon$};

\end{tikzpicture}

\caption{Schematic depiction of the temporal evolution of various solutions. The horizontal axis represents $\bar{u}$. The blue curve represents an inviscid non-unique solution, which grows like $t^a$ in critical norms. The red curve represents the unperturbed viscous solution $\tilde{u}^\nu$, which deviates by $O(1)$ in critical norms but converges to $\bar{u}$. The purple curve represents a perturbed viscous solution $u^\nu$, which approximately lands on the inviscid solution after the initial layer.}
\label{fig:figure_solutions}
\end{figure}

\begin{remark}
\begin{enumerate}[leftmargin=*]
    \item When $b \neq 0$, the phase parameter in~\eqref{eq:differentphasesrule} indicates that different subsequences can converge to different non-unique solutions.
    \item The choice $\varepsilon \sim \nu^{\kappa_{\rm c}}$ guarantees that the growth of the perturbation from $O(\varepsilon)$ to $O(1)$ in critical norms happens at time $O(1)$. For $\nu^{\kappa_{\rm c}} \ll \varepsilon \ll 1$, our theory yields growth from $O(\varepsilon)$ to $O(1)$ in times $T$ in the range $T_\nu \ll T \ll 1$. 
    At the bottom of this range, with the choice $\varepsilon = O(1)$, no growth happens in time $O(T_\nu)$.
    \item $c_0$ in Theorem~\ref{thm:main} normalizes the observation time to $T=1$. A time rescaling demonstrates that a smaller (resp. larger) $c_0$ yields a later (resp. earlier) observation time.
\end{enumerate}
\end{remark}





\subsection{Further context and motivation}
\label{sec:furthercontext}





\subsubsection{Weak Euler solutions}
Onsager proposed that weak solutions to the 3D Euler equations which dissipate energy could be relevant to the description of turbulent flows at infinite Reynolds number~\cite{OnsagerStatisticalHydrodynamics1949}. In the past twenty years, many mathematical breakthroughs have revealed that weak Euler solutions exhibit both anomalous dissipation and significant non-uniqueness. The method of convex integration developed in~\cite{delellis2009euler,delellis2013dissipative} led to the resolution of Onsager's conjecture~\cite{isett2018proof} and recently also in~2D \cite{giriradu}. 
Nonetheless, many foundational questions remain~\cite{de2022weak,BuckmasterPhenomenologies2020,EyinkSurvey2024}. 
One such question is to describe the non-unique solutions of the Euler equations which can be obtained in the vanishing viscosity limit from Leray solutions of the Navier-Stokes equations.


In dimension two, since Yudovich's well-posedness theory for 2D Euler with bounded vorticity \cite{yudovich1963non}, there has been intense work to study weak solutions with vorticity in $L^\infty_t L^p_x$ for $p<\infty$.
Existence of such solutions was obtained in a series of results, including \cite{diperna1987concentrations}, \cite{delort1991existence} (see \cite{CriSte21} and the references therein).
Several approaches were designed to understand the (non)uniqueness of such solutions:

On the one hand, we have Vishik's approach based on instability in self-similar variables~\cite{vishik2018instability1,vishik2018instability2}, which successfully answered the uniqueness question for the forced 2D Euler equations, as described in~\Cref{sec:inviscidnonuniqueness}. \\
~\cite{albritton2023instability} contains an exposition of Vishik's work, and two alternative proofs have appeared recently~\cite{castro2024proof}, \cite{dolce2024selfsimilarinstabilityforcednonuniqueness}.
(Currently, the simplest proof would be to combine the regularized piecewise constant vortex in~\cite{castro2024proof} with the Golovkin trick in~\cite{dolce2024selfsimilarinstabilityforcednonuniqueness}.) See also~\cite{castro2025unstable} for the SQG setting. An unforced scenario has been proposed in~\cite{MR4182316} based on numerical experiments which exhibit multiple self-similar solutions. See~\cite{jiasverakillposed,guillodsverak,MR4429263} for related work on non-uniqueness of Leray solutions to the Navier-Stokes equations.

On the other hand, convex integration can also be adapted to the 2D vorticity equations. In \cite{MR4645737}, a convex integration scheme with intermittency was designed to prove non-uniqueness of the 2D Euler equations \eqref{eq:euler_1} with vorticity in the Lorentz space $\omega \in L^\infty_t L^{1,\infty}_x$ and refined in \cite{MR4725248} 
 to reach class Hardy class $\omega \in L^\infty_t H^p_x$, $0 < p < 1$. In \cite{MR4788287}, non-uniqueness was obtained with $L^p$ initial vorticity, in the class of admissible $L^2$ velocities. Due to an inherent obstruction arising from the mechanism used to cancel the error and the Sobolev embedding theorem, the classical convex integration scheme cannot achieve \(  L^\infty_t L^{1}_x \) integrability for the vorticity. 
Very recently, the work \cite{BruColKum24b} introduced a new way to perform convex integration, inspired by a sharp passive scalar result  \cite{BruColKum24a}, not anymore based on highly oscillatory periodic perturbations but using new building blocks based on the Lamb-Chaplygin dipole.

Vanishing viscosity solutions of the 2D Euler equations have been extensively studied since they are known to exhibit more rigid properties compared to generic weak solutions to \eqref{eq:euler_1}, see \cite{MR3551263,constantin2019vorticity,lopes2006weak,lanthaler2021conservation,de2024no,MR3366055} and the references therein. For instance, it is proven that a vanishing viscosity solution with initial vorticity \(\omega_0 \in L^p\), \(p > 1\), automatically conserves energy and is renormalized.  Energy conservation follows from  strong $L^2$ compactness at the velocity level when the vorticity is in $L^p$ and generalizes to other regularizations (for instance, smoothing out the initial datum). 
Our paper gives the first instance where, for the forced 2D Euler equations, vanishing viscosity solutions maintain the non-uniqueness seen in the inviscid case. 
While convex integration gives rise to a huge class of flexible, non-unique solutions, it is an open question to understand which solutions can be seen in the vanishing viscosity limit. 

Finally, there is a rich literature studying inviscid limits to the compressible Euler equations and related conservation laws. 
In this vein, perhaps the most relevant work is the analysis of the inviscid limit to self-similar shock formation in~\cite{Chaturvedi2022}.

\subsubsection{Spontaneous stochasticity}
Theorem~\ref{thm:main} is partially inspired by the notion of \emph{spontaneous stochasticity}, investigated primarily for the flow map \cite{Bernard1998,Gawdzki2000,Eyink2014} and more recently for the velocity field itself~\cite{Mailybaev2016,thalabard2020butterfly}. This concept is moreover connected to predictability in weather and climate~\cite{palmer2014real}.

We focus on the example of the turbulent mixing layer in 2D which emerges from the Kelvin-Helmholtz instability~\cite{thalabard2020butterfly}. Consider a vortex sheet, whose associated velocity is the discontinuous shear profile $(U {\rm sgn}(y),0)$. This is both a singular steady state and a self-similar solution under the scaling with $L = T$, set by the dimensional constant $U$. The authors consider random perturbations of the sheet by characteristic distance $\varepsilon$ and solve the hyperviscous Navier-Stokes equations. When $\varepsilon \sim \nu$, it is observed numerically that the measure of solutions does not collapse to a Dirac on the sheet but remains a genuine measure, capturing turbulent mixing solutions. 
The uniqueness threshold $\varepsilon \sim \nu$ in~\cite{thalabard2020butterfly} is not set by an unstable mode, in contrast to the present work. Speculatively, it is determined by an attractor of the Euler equations which is somehow ``steady" in similarity variables with $L = T$. The separation rate matches that of the non-unique solutions obtained by convex integration in~\cite{Szkelyhidi2011} and saturates the upper bounds in~\cite{kalinin2024scale}.




Our perturbations are not random, nor do we succeed in characterizing the viscous dynamics in a full neighborhood of the self-similar solution. Nonetheless, our approach might be considered a deterministic variant of the spontaneous stochasticity paradigm.



Finally, we highlight the work~\cite{Vasseur2023} on the quantifying predictability in the inviscid limit with boundary. The authors study the inviscid limit to a constant shear flow $(A,0)$ in a 3D channel.\footnote{One way to view this is as a boundary variant of the shear layer in the Kelvin-Helmholtz instability.} Their goal is to estimate the time-scale on which the solution is observable, that is, until the perturbation grows to the size of the background flow. They prove that the energy of perturbations grows at most like $A^3 T$ in the inviscid limit, so that the solution is observable up to time $1/A$. The authors also exhibit non-unique solutions via convex integration which saturate this bound. 


\subsubsection{Selection principle in passive scalars}
The selection principle via vanishing viscosity has been investigated in recent years also for the transport equation with a divergence free vector field. When the vector field has Sobolev regularity, namely in the framework of the DiPerna-Lions theory of regular Lagrangian flows, vanishing viscosity is known to select the unique Lagrangian solution, while distributional solutions can be non-unique (see {\cite{MR4496901}} for a quantitative convergence result and  \cite{mescolini2025vanishing} in the BV setting). 
As soon as the vector field is less regular, for instance $\alpha$-Holder continuous in space,
the Obukhov-Corrsin theory of scalar turbulence expects nonunique solutions of the transport equation in the class $C^\beta$ for any $\beta<1-2\alpha$, lack of selection by vanishing viscosity and anomalous dissipation.
This was proven to actually happen in a series of works, see for instance \cite{MR4381138,MR4662772,armstrong2023anomalous,elgindi2024norm,burczak2023anomalous}.
In recent years, also other regularization mechanisms were considered and lack of selection was proven \cite{MR4037474,MR4396067}.

The aforementioned passive scalar examples were then used in \cite{MR4595604,MR4799447} to
build new examples of solutions to
the forced 3D Navier-Stokes equations with vanishing viscosity, which exhibit lack of selection, anomalous dissipation and which enjoy uniform bounds in up to Onsager-critical spaces. At difference from the current paper, however, such results are inherently 3-dimensional and they require a $\nu$-dependent forcing term. Moreover, in the current paper we exploit the self-similar type structure of the solution to have a more detailed description of the behavior of our solutions and its asymptotic expansion at $0$ along the linearly unstable mode.

%% file: spectral.tex
\section{Spectral problem}
\label{sec:spectral_pb}

In this section, we prove Proposition~\ref{prop:lin_part}. Compared to previous works \cite{vishik2018instability2,albritton2023instability}, our main contributions will be (i) to analyze the linearized operators in $L^2$ of \emph{velocity}, and (ii) track the simplicity of the eigenvalue. Both are exploited in the nonlinear arguments of Section~\ref{sec:nonlinear}.

The analysis in velocity formulation is more challenging than vorticity formulation: First, one cannot ignore the pressure. Second, the operator does not evidently decompose into ``nearly anti-symmetric plus compact" except when analyzed mode-by-mode. 
Consequently, the ``free" semigroup bound~\eqref{eq:semigroupestimatebadonethough} granted by the abstract spectral theory is not uniform in the Fourier modes! We obtain the semigroup estimate only after estimating the resolvent uniformly-in-$m$. Luckily, these issues were addressed by Gallay and Smets in~\cite{gallay2019linear}, who analyzed the yet more difficult \emph{columnar vortices} (see also~\cite{Gallay2020,dallaswojtekvortexcolumns}).

Our situation is further complicated because we combine these difficulties with a spectral perturbation argument to incorporate a term $\kappa \cdot \nabla_\xi U$ in the self-similarity variables.

\subsection{Unstable vortex in vorticity formulation}
\label{sec:unstable_vortex}
Let
\begin{equation}
\bar{u} = \Omega(r) x^\perp = V(r) e_\theta
\end{equation}
    be a vortex, with corresponding vorticity profile\footnote{$\Omega$ is called $\zeta$ in \cite{albritton2023instability}; we call it $\Omega$ for consistency with \cite{gallay2019linear}.} 
\begin{equation}
\bar{\omega} = W(r) \, .
\end{equation}

In what follows, we suppose $\bar{u} \in C^2$ satisfies a mild growth condition: 
\begin{equation}
	\label{eq:decayconditions}
|\bar{u}| + r|\nabla \bar{u}| + r^2|\nabla^2 \bar{u}| \lesssim r^{1-\alpha} \, ,\quad \text{ for some } \alpha > 0 \, .
\end{equation}


Recall
\begin{equation}
\Omega' = (V/r)' = V'/r - \Omega/r
\end{equation}
\begin{equation}
W = \frac{1}{r} \p_r (rV(r)) = \Omega + V' = r\Omega' + 2\Omega \, .
\end{equation}

We consider the linearized Euler equations in vorticity formulation
\begin{equation}
\p_t \omega + \bar{u} \cdot \nabla \omega + u \cdot \nabla \bar{\omega} = 0 \, , \quad u = \nabla^\perp \Delta^{-1} \omega \, .
\end{equation}
The corresponding linearized operator $\mathcal{L}$ is given by
\begin{equation}
-\mathcal{L} = \bar{u} \cdot \nabla \omega + u \cdot \nabla \bar{\omega} \, .
\end{equation}
On the whole space $\R^2$, it is necessary to be careful with the functional set-up. To ensure that $u$ is well defined,\footnote{This requirement was discussed in 
\cite[Remark 1.0.5 and Lemma 2.4.1]{albritton2023instability}. Alternatively, one can consider the operator on the space $L^Q$, $Q \in (1,2)$, which works even for non-vortex backgrounds. $\mathcal{M}$ below will no be anti-symmetric, but its spectrum will be purely imaginary.} we consider $\mathcal{L}$ as an unbounded operator on the space $L^2_{m_0}$ of $m_0$-fold symmetric ($m_0 \geq 2$) square integrable functions:
\begin{equation}
\mathcal{L} : D(\mathcal{L}) \subset L^2_{m_0} \to L^2_{m_0} \, .
\end{equation}
Its domain is
\begin{equation}
D(\mathcal{L}) = \{ \omega \in L^2_{m_0} : \bar{u} \cdot \nabla \omega \in L^2_{m_0} \} \, .
\end{equation}
We have the decomposition
\begin{equation}
\mathcal{L} = \mathcal{M} + \mathcal{K}
\end{equation}
\begin{equation}
- \mathcal{M} = \bar{u} \cdot \nabla \omega \, , \quad - \mathcal{K} = u \cdot \nabla \bar{\omega} \, ,
\end{equation}
where $\mathcal{M}$ is anti-symmetric and $\mathcal{K}$ is compact. As a consequence, the spectrum consists of essential spectrum on the imaginary axis and isolated eigenvalues of finite multiplicity, whose only accumulation points must be in the essential spectrum.

The space $L^2_{m_0}$ decomposes into a Hilbert sum of invariant subspaces $U_m := \{ \omega \in L^2 : \omega = \omega_m(r) e^{im\theta} \}$ (where $m \in m_0\mathbb{Z}$ is a wavenumber). Therefore, any eigenvalue $\lambda$ must have a corresponding eigenfunction $\omega$ in a subspace $U_m$. Then
\begin{equation}
(\lambda + im\Omega) \omega_m + u^r_m W' = 0 \, .
\end{equation}
(The operator is trivial when $m=0$.) When $|\Re \lambda| > 0$, we have
\begin{equation}
	\label{eq:goodforbootstrapping}
\omega_m + (\lambda + im\Omega)^{-1} u^r_m W' = 0 \, .
\end{equation}


    \begin{theorem}[\cite{vishik2018instability1},\cite{vishik2018instability2}, \cite{albritton2023instability}]
	\label{thm:vishiknotestheorem}
There exists a smooth vortex $\bar{\omega}$ satisfying the condition~\eqref{eq:decayconditions} and the following properties. There exists $m_0 \geq 2$ such that $\mathcal{L} : L^2_{m_0} \to L^2_{m_0}$ has either
\begin{enumerate}
\item two algebraically simple unstable eigenvalues $\lambda, \bar{\lambda}$ (not identical) 
\item a semisimple eigenvalue $\lambda \in \R$ of multiplicity two.
\end{enumerate}
In either case, each of $U_{m_0}$ and $U_{-m_0}$ has a one-dimensional eigenspace, and
\begin{equation}
	\label{eq:myspectrumisonlytwopoints}
\sigma_{m_0}(\mathcal{L}) \cap \{ \Re \lambda > 0 \} = \{ \lambda, \bar{\lambda} \} \, .
\end{equation}
\end{theorem}
(The notation $\sigma_{m_0}(\mathcal{L})$ emphasizes that we consider the spectrum of $\mathcal{L}$ on $m_0$-fold symmetric functions.)


As in~\cite{MR4429263}, we consider the truncated vortex
\begin{equation}
\bar{u}_R = \Omega(r) \chi(r/R) x^\perp \, ,
\end{equation}
where $\chi$ is an appropriate cut-off function. Let $\bar{\omega}_R = W_R(r)$ be the associated vorticity.

\begin{lemma}[Truncation]
	\label{lem:truncatevortex}
Let $\bar{\omega}$ and $m_0$ be as in Theorem~\ref{thm:vishiknotestheorem}. For every $\varepsilon > 0$, whenever $R \gg_\varepsilon 1$, the linearized operator $\mathcal{L}_R : L^2_{m_0} \to L^2_{m_0}$ around the truncated vortex $\bar{\omega}_R$ satisfies the conclusions of Theorem~\ref{thm:vishiknotestheorem} with~\eqref{eq:myspectrumisonlytwopoints} replaced by
\begin{equation}
\label{eq:myspectrumisonlytwopoints2}
\sigma_{m_0}(\mathcal{L}_R) \cap \{ \Re \lambda > \varepsilon \} = \{ \lambda_R, \bar{\lambda_R} \} \, .
\end{equation}
\end{lemma}

\begin{proof}
Consider $\mathcal{L}_R$ on the invariant subspaces $U_m$, $m \in m_0\mathbb{Z}$. For each individual mode $k \in m\mathbb{Z}$, $\mathcal{L}_R^{(m)}$ is a bounded perturbation of $\mathcal{L}^{(m)}$, and we can apply the standard perturbation theory. The main difficulty is therefore to control the spectrum when $|m| \gg 1$. However, the compact operator $\omega_m \mapsto u^r_m W_R'$ satisfies
\begin{equation}
\| u^r_m W_R' \|_{L^2(r \, dr)} \lesssim \frac{1}{|m|} \| \omega_m \|_{L^2(r \, dr)} \, , \quad \forall \omega_m \in L^2(r \, dr) \, .
\end{equation}
Therefore, $\lambda - \mathcal{L}^{(m)}_R$ is invertible whenever $|\Re \lambda| \gg |m|^{-1}$.
\end{proof}

The analysis of the Rayleigh equation actually yields that there is no unstable eigenvalue for large $|m|$.

\subsection{Velocity formulation}

In this section, we will suppose
\begin{equation}
\Omega \in C^2([0,+\infty)) \, , \quad \Omega'(0) = 0 \, , \quad \| r^2|\Omega| + r^3 |\Omega'| + r^4 |\Omega''| \|_{L^\infty} < +\infty \, .
\end{equation}
(Under these assumptions, the estimates on $P$ and $B_m$ in \cite{gallay2019linear}, Section~3 and~4, remain valid.)


We consider the linearized operator in $L^2$ of velocity:
\begin{equation}
\begin{aligned}
-Lu &:= \bar{u} \cdot \nabla u + u \cdot \nabla \bar{u} + \nabla P \\
&= \mathbb{P} (\bar{u} \cdot \nabla u + u \cdot \nabla \bar{u}) \, ,
\end{aligned}
\end{equation}
where $P = (-\Delta)^{-1} \div \div (\bar{u} \otimes u + u \otimes \bar{u})$.
 Its domain is
\begin{equation}
D(L) := \{ u \in L^2_{\rm df} : \mathbb{P}(\bar{u} \cdot \nabla u) \in L^2_{\rm df} \} \, .
\end{equation}
With this domain, $L : L^2_{\rm df} \to L^2_{\rm df}$ is closed and densely defined.

In polar coordinates,
\begin{equation}
\bar{u} \cdot \nabla u = \Omega \p_\theta u = \Omega \p_\theta u^r e_r + \Omega \p_\theta u^\theta e_\theta + \Omega u^r e_\theta - \Omega u^\theta e_r
\end{equation}
\begin{equation}
u \cdot \nabla \bar{u} = u^r \p_r V e_\theta - \Omega u^\theta e_r \, ,
\end{equation}
and
\begin{equation}
	\label{eq:linearizeder}
-Lu \cdot e_r = \Omega \p_\theta u^r - 2 \Omega u^\theta + \p_r p
\end{equation}
\begin{equation}
	\label{eq:linearizedetheta}
-Lu \cdot e_\theta = \Omega \p_\theta u^\theta + Wu^r + \frac{1}{r} \p_\theta p \, .
\end{equation}
The pressure satisfies
\begin{equation}
- \p_r^* \p_r p - \frac{1}{r^2} \p_\theta^2 p = 2(\p_r^* \Omega) \p_\theta u^r - 2 \p_r^* (\Omega u^\theta) \, ,
\end{equation}
where $\p_r^* = r^{-1} \p_r (r \cdot)$, the formal $L^2(r \, dr)$-adjoint of $- \p_r$, which can be obtained from~\eqref{eq:linearizeder} and~\eqref{eq:linearizedetheta}, or alternatively, is a realization of
\begin{equation}
-\Delta P = \div \div (2u \otimes \bar{u}) = \div( 2\bar{u} \cdot \nabla u) = 2\nabla \bar{u} : (\nabla u)^T \, .
\end{equation}



The operator is split into
\begin{equation}
-Au = \Omega (\p_\theta u^r e_r + \p_\theta u^\theta e_\theta) - r \Omega' u^r e_\theta
\end{equation}
\begin{equation}
-Bu = - 2\Omega u^\theta e_r + 2(r\Omega)' u^r e_\theta + \nabla p \, .
\end{equation}
Both operators preserve the divergence-free condition. (This is the reason for incorporating the term $-r \Omega' u^r e_\theta$ into $A$.)

We can project the operator onto the invariant subspaces
\begin{equation}
\label{eq:def_Xm}
X_m := \{ u \in L^2_{\rm df} : u = u_m e^{im\theta} = (u^r_m(r) e_r + u^\theta_m(r) e_\theta) e^{im\theta} \} \, , \quad m \in \mathbb{Z} \, .
\end{equation}
Each invariant subspace is canonically isomorphic to
\begin{equation}
Y_m := \left\{ (u^r_m(r),u^\theta_m(r)) \in (L^2(\R^+, r \, dr))^2 : \left( \p_r + \frac{1}{r} \right) u^r_m + \frac{im}{r} u^\theta_m = 0 \right\} \, .
\end{equation}
We write $L_m : Y_m \to Y_m$ to denote the `restriction' of the operator $L$ to this invariant subspace, on which it is bounded. For each $u \in Y_m$, we have (omitting the $m$ in $(u^r_m,u^\theta_m)$ when convenient to relax the notation) 
\begin{equation}
(-L_m u)^r = im\Omega u^r - 2\Omega u^\theta + \p_r p
\end{equation}
\begin{equation}
(-L_m u)^\theta = im\Omega u^\theta + W u^r + \frac{im}{r} p \, .
\end{equation}
The decomposition of $L_m$ is (with mild abuse of notation)
\begin{equation}
	\label{eq:Amdef}
- A_m u = \begin{bmatrix}
im\Omega & \\
- r\Omega' & im\Omega
\end{bmatrix} u
\end{equation}
\begin{equation}
	\label{eq:Bmdef}
- B_m u = \begin{bmatrix}
-2\Omega u^\theta + \p_r p \\
2(r\Omega)' u^r + \frac{im}{r} p
\end{bmatrix} \, .
\end{equation}


\begin{proposition}[Proposition 2.2 in~\cite{gallay2019linear}]
	\label{pro:decompositioncompactness}
We have the following properties:
\begin{enumerate}
\item $\sigma(A_m) = \overline{{\rm range}(-im\Omega)}$
\item $B_m$ is compact.
\item $\sigma(L_m) = \sigma_{\rm ess}(L_m) \cup \sigma_{\rm dis}(L_m)$ consists of essential spectrum on the imaginary axis and discrete eigenvalues.
\end{enumerate}
\end{proposition}

The main point is the compactness of $B_m$, which is not entirely obvious. This is Proposition 2.2 in \cite{gallay2019linear}, with $u^z = 0$ and $k=0$.

Proposition~\ref{pro:decompositioncompactness} immediately produces a semigroup estimate
\begin{equation}
	\label{eq:semigroupestimatebadonethough}
\| e^{tL_m} \|_{L^2 \to L^2} \lesssim_{\delta,m} e^{t(a_m+\delta)} \, , \quad \forall \delta > 0 \, , m \in \mathbb{Z} \, ,
\end{equation}
where $a_m := \sup \Re \sigma(L_m)$. The constant in~\eqref{eq:semigroupestimatebadonethough} is not guaranteed to be uniform-in-$m$. At this point, we will often drop the measure $r\,dr$ from the notation.

It will be convenient to introduce notation for the Biot-Savart law:
\begin{equation}
\Delta_m \psi_m := \left( \p_r^2 + \frac{1}{r} \p_r - \frac{m^2}{r^2}\right) \psi_m = \omega_m
\end{equation}
\begin{equation}
u^r_m = -\frac{im}{r} \psi_m \, , \quad u^\theta_m = \p_r \psi_m
\end{equation}
\begin{equation}
\omega_m = - \frac{im}{r} u^r_m + (\p_r + \frac{1}{r}) u^\theta_m \, .
\end{equation}
More succinctly, $\nabla_m = (\p_r,im/r)$, $u_m = \nabla_m^\perp \psi$, and $\omega_m = (-im/r,\p_r^*) \cdot u_m$.

First, we verify that vorticity eigenfunctions beget velocity eigenfunctions, and vice versa.

\begin{lemma}[Equivalence of eigenfunctions]
\label{lemma:equivalence_egfn}
Suppose that $u = u_m(r) e^{im\theta} \in L^2$, $m \geq 2$, is a stable or unstable eigenfunction in velocity formulation. Then $\omega = \curl u \in L^2$.

Conversely, suppose that $\omega = \omega_m(r) e^{im\theta} \in L^2$, $m \geq 2$, is a stable or unstable eigenfunction in vorticity formulation. Then $u_m = \nabla_m^\perp \Delta_m^{-1} \omega_m \in L^2$.
\end{lemma}
\begin{proof}
To go from $L^2$ velocity to $L^2$ vorticity, we must bootstrap regularity. Write
\begin{equation}
u = (\lambda - A_m)^{-1} B_m u_m \, ,
\end{equation}
where $\lambda - A$ is invertible because $\Re \lambda \neq 0$ (check the determinant of $A$ in polar coordinates), and utilize that $B$ gains a derivative. To prove $L^2$ estimates on $B_m$ and its derivative we follow the argument in Lemma 3.7 in \cite{gallay2019linear}; in particular, defining $R_1 \coloneqq 2r (-(r\Omega)'' u_r + im\Omega' u_\theta))$, multiplying by $r^2$ the incompressibility condition on $B_m$ and differentiating with respect
to $r$, we get
\begin{equation}
\label{eq:Benestimates}
-\partial_{rr} B_r - 3r^{-1} \partial_r B_r + (m^2-1)r^{-2} B_r = imr^{-1} R_1.
\end{equation}
We now perform energy estimates on \eqref{eq:Benestimates} to get bounds on $B_r$, multiplying by $r\bar B_r$ and $r^3 \bar B_r$ and integrating by parts. To bound $B_\theta$, we observe that $\partial_r^* B_\theta = im r^{-1} B_r + R_1$, and we obtain the bound
\begin{equation}
\label{eq:B_gallay}
\|\partial_r B_m\|_{L^2} + \|r \partial_r B_m \|_{L^2} + m\| r^{-1} B_m\|_{L^2} \leq C(1+m) \| u_m\|_{L^2}.
\end{equation}
To go from $L^2$ vorticity to $L^2$ velocity, we utilize the Rayleigh equation~\eqref{eq:goodforbootstrapping}.
\end{proof}





\begin{lemma}[Uniform resolvent estimates]
Under either condition
\begin{itemize}
\item[I.]  $|\Re \lambda| \gg \frac{1}{m}$
\item[II.] $|\Im \lambda| \gg m \| \Omega \|_{L^\infty}$
\end{itemize}
the resolvent $R(\lambda,L_m) : L^2 \to L^2$ exists and satisfies
\begin{equation}
\| R(\lambda,L_m) \|_{L^2 \to L^2} \lesssim \| \frac{1}{\lambda + im\Omega} \|_{L^\infty} + \| \frac{1}{\lambda + im\Omega} \|_{L^\infty}^2 \, . 
\end{equation}

\end{lemma}
\begin{proof}
For this, we will require the Rayleigh equation.
The resolvent problem for $u_m$ can be rewritten as
\begin{equation}
(\lambda + im\Omega) \Delta_m \psi_m - \frac{im}{r} \psi_m W' = g_m \, ,
\end{equation}
where $W$ is the vorticity profile of the vortex, and $g_m$ is curl of $f_m$:
\begin{equation}
g_m = - \frac{im}{r} f^r_m + (\p_r + \frac{1}{r}) f^\theta_m \, .
\end{equation}

$L^2$ estimates on $\nabla_m \psi_m$ are the desired estimates on $u_m$.

We divide through by $\lambda + im\Omega$:
\begin{equation}
	\label{eq:Rayleighequationdivide}
\Delta_m \psi_m = \frac{im}{r} \frac{1}{\lambda+im\Omega} \psi_m W' + \frac{g_m}{\lambda+im\Omega} \, .
\end{equation}
Then we multiply by $\bar{\psi}_m$ and integrate by parts:
\begin{equation}
	\label{eq:energyestimateforpsi}
\begin{aligned}
&\| \p_r \psi_m \|_{L^2}^2 + m^2 \| r^{-1} \psi_m \|_{L^2}^2 \leq \frac{1}{m} \| \frac{1}{\lambda + im\Omega} \|_{L^\infty} \| mr^{-1} \psi_m \|_{L^2}^2 \| rW' \|_{L^\infty} \\
&\quad + \left| \Re \int  \frac{im}{r} f^r_m \frac{\bar{\psi}_m}{\lambda+im\Omega}  \, r \,dr \right| + \left| \Re \int \p_r^* f^\theta_m \frac{\bar{\psi}_m}{\lambda+im\Omega} \, r \, dr \right| \, .
\end{aligned}
\end{equation}
Under conditions I or II, the first term on the right-hand side can be absorbed into the left-hand side. The second term on the right-hand side is estimated by
\begin{equation}	
	\label{eq:estimate1toreplace}
	C \| \frac{1}{\lambda + im\Omega} \|_{L^\infty} \| f^r_m \|_{L^2} \, \| mr^{-1} \bar{\psi}_m \| \, .
\end{equation}
For the third term, we integrate by parts. The relevant estimate is
\begin{align}
	\label{eq:estimate2toreplace}
	& C \| \frac{1}{\lambda + im\Omega} \|_{L^\infty} \| f^\theta_m \|_{L^2} \, \| \p_r \psi \|_{L^2} \\ & \quad + C \| \frac{1}{\lambda + im\Omega} \|_{L^\infty}^2 \| f^\theta_m \|_{L^2} \, \| mr^{-1} \psi_m \|_{L^2} \| r\Omega' \|_{L^\infty} \, .
\end{align}
The above estimates can be converted into a solvability proof.
\end{proof}

The following resolvent estimates will be useful when we perturb to the self-similar problem:
\begin{corollary}
	\label{cor:Linversecorollary}
Under conditions I or II above, we have
\begin{equation}
\| R(\lambda,L_m) (\lambda - A_m) \|_{L^2 \to L^2} \lesssim 1 + \| \frac{1}{\lambda + im\Omega} \|_{L^\infty}^2 \, .
\end{equation}
\end{corollary}
\begin{proof}
Let $f = (\lambda - A_m) h$, that is,
\begin{equation}
f^r = (\lambda + im \Omega) h^r \, , \quad f^\theta = (\lambda + im \Omega) h^\theta - r \Omega' h^r \, .
\end{equation}

We substitute $f = A_m h$ into~\eqref{eq:energyestimateforpsi}. Then $\lambda + im\Omega$ terms in the numerator and denominator cancel, and the estimate~\eqref{eq:estimate1toreplace} is replaced by $C \| h^r_m \|_{L^2} \, \| mr^{-1} \bar{\psi}_m \|_{L^2}$, 
while~\eqref{eq:estimate2toreplace} is replaced by
\begin{equation}
\begin{aligned}
&C \| h^\theta_m \|_{L^2}  \, \| \p_r \psi \|_{L^2}  + C \| \frac{1}{\lambda + im\Omega} \|_{L^\infty}  \| h^\theta_m \|_{L^2}  \, \| mr^{-1} \psi_m \|_{L^2}  \\
&\quad + C \| \frac{1}{\lambda + im\Omega} \|_{L^\infty} \| h^r_m \|_{L^2}  \, \| \p_r \psi \|_{L^2}  + C \| \frac{1}{\lambda + im\Omega} \|_{L^\infty}^2 \| h^r_m \|_{L^2}  \, \| mr^{-1} \psi_m \|_{L^2}. 
\end{aligned}
\end{equation}
\end{proof}

\subsection{Self-similar operator}

From now on, we work with a smooth truncated vortex from Lemma~\ref{lem:truncatevortex}.

We now examine
\begin{equation}
-\tilde{L}^{(\kappa)} = \kappa (-1+\frac{1}{\alpha} - \frac{\xi}{\alpha} \cdot \nabla_\xi) - L = - \kappa (1-\frac{2}{\alpha}) + \underbrace{\frac{\kappa}{\alpha} (- 1 - \xi \cdot \nabla_\xi) - L}_{= -L^{(\kappa)}} \, .
\end{equation}
The operator $L^{(\kappa)}$ has a different domain than $L$:
\begin{equation}
D(L^{(\kappa)}) = \{ u \in L^2_{\rm df} : (-\frac{\kappa}{\alpha} \xi \cdot \nabla_\xi + \bar{u} \cdot \nabla) u \in L^2_{\rm df} \} \, .
\end{equation}
We again study the problem mode-by-mode and decompose
\begin{equation}
L_{m}^{(\kappa)} = A_{m}^{(\kappa)} + B_m\, ,
\end{equation}
\begin{equation}
A_m^{(\kappa)} = \frac{\kappa}{\alpha} (1 + \xi \cdot \nabla_\xi) + A_m \, .
\end{equation}
The spectrum of $A_m^{(\kappa)}$ is contained in the imaginary axis, as can be demonstrated via an energy estimate for the resolvent problem. 
It was already demonstrated that $B_m$ is compact.
Then $A_m^{(\kappa)} + B_m$ has the same basic spectral stucture as $A_m + B_m$. That is, the essential spectrum of $L_m^{(\kappa)}$ belongs to the imaginary axis, and the remainder consists of isolated eigenvalues of finite multiplicity whose only possible accumulation points must belong to the essential spectrum.

\begin{proposition}
	\label{pro:spectralperturbation}
Let $\varepsilon > 0$. Let $\bar{\omega}$ be a smooth truncated vortex and $m_0 \geq 2$, as in Lemma~\ref{lem:truncatevortex}, with $\{ \Re \lambda > \varepsilon/2 \}$ in~\eqref{eq:myspectrumisonlytwopoints2}.

Let $0 < \kappa \ll_\varepsilon 1$. Then $L^{(\kappa)} : L^2_{m_0} \to L^2_{m_0}$ has either
\begin{enumerate}
\item two algebraically simple unstable eigenvalues $\lambda, \bar{\lambda}$ (not identical) 
\item a semisimple eigenvalue $\lambda \in \R$ of multiplicity two.
\end{enumerate}
In either case, each of $U_{m_0}$ and $U_{-m_0}$ has a one-dimensional eigenspace, and
\begin{equation}
	\label{eq:myspectrumisonlytwopoints3}
\sigma_{m_0}(L^{(\kappa)}) \cap \{ \Re \lambda > \varepsilon \} = \{ \lambda(\kappa), \bar{\lambda}(\kappa) \} \, .
\end{equation}


The operator $R(\lambda,L_m^{(\kappa)})$ exists and enjoys uniform-in-$m$ $L^2 \to L^2$ estimates in the region
\begin{equation}
( \{ |\Re \lambda| \gg \frac{1}{m} \} \cup \{ |\Im \lambda| \gg m \| \Omega \|_{L^\infty} \} ) \cap \{ \Re \lambda > \varepsilon \} \, .
\end{equation}
\end{proposition}

We begin by collecting estimates on $R(\lambda,A^{(\kappa)}_m)$, $R(\lambda,A_m)$, and their difference.

\begin{lemma}
	\label{lem:resolventsarecloseAm}
\begin{enumerate}[label=(\roman*)]
\item We have
\begin{equation}
	\label{eq:resolventestimatesforAm}
\| R(\lambda,A_m) \|_{L^2 \to L^2} \, , \| R(\lambda,A^{(\kappa)}_m) \|_{L^2 \to L^2} \lesssim \frac{1}{|\Re \lambda|} + \frac{1}{|\Re \lambda|^{2}} \, ,
\end{equation}
where $|\Re \lambda| > 0$.
\item Regarding the difference,
\begin{align}
	\label{eq:estimateforthedifference}
\| [R(\lambda,A^{(\kappa)}_m) - R(\lambda,A_m)] f \|_{L^2} & \lesssim \kappa (|\Re \lambda|^{-1} + |\Re \lambda|^{-5} ) \\ & \qquad ((1 + m )\| f \|_{L^2} + \| r \p_r f \|_{L^2}) \, ,
\end{align}
provided that the right-hand side is finite.
\item In particular, for every $m$,
\begin{equation}
R(\lambda,A^{(\kappa)}_m) f \to R(\lambda,A_m) f \text{ as } \kappa \to 0 \, , \quad \forall f \in L^2 \, ,
\end{equation}
uniformly in $\lambda \in \{ |\Re \lambda| > \varepsilon \}$ for every $\varepsilon > 0$.
\item Consequently,
\begin{align}
	\label{eq:estimatesondiffwithBm}
\| [R(\lambda,A^{(\kappa)}_m) - R(\lambda,A_m)] B_m \|_{L^2 \to L^2} & \lesssim (1+m)^2 \kappa \\ &\qquad (|\Re \lambda|^{-1} + |\Re \lambda|^{-5} ) \, .
\end{align}
\end{enumerate}
\end{lemma}
\begin{proof}
(i) By the definition~\eqref{eq:Amdef} of $A_m$, we have
\begin{equation}
\lambda - A_m = 
\begin{bmatrix}
\lambda + im\Omega & \\
- r\Omega' & \lambda + im\Omega
\end{bmatrix}
\end{equation}
\begin{equation}
	\label{eq:identityforinvertingAm}
 R(\lambda,A_m) = \frac{1}{\lambda + im \Omega} \begin{bmatrix}
1 & \\
\frac{r\Omega'}{\lambda + im \Omega} & 1
\end{bmatrix} \, ,
\end{equation}
which is evidently bounded by (a constant multiple of) $|\Re \lambda|^{-1} + |\Re \lambda|^{-2}$. For $R(\lambda,A_m^{(\kappa)})$, we give instead a proof by energy estimates. Write the resolvent problem as two equations:
\begin{equation}
- \frac{\kappa}{\alpha} (1 + r \p_r) u^r_m + (\lambda + im\Omega) u^r_m = f^r_m
\end{equation}
\begin{equation}
- \frac{\kappa}{\alpha} (1 + r \p_r) u^\theta_m + (\lambda + im\Omega) u^\theta_m = f^\theta_m + r \Omega' u^r_m \, .
\end{equation}
An energy estimate for $u^r_m$ yields $\| u^r_m \|_{L^2} \leq |\Re \lambda|^{-1} \| f^r_m\|_{L^2}$. We plug this into the right-hand side of the equation for $u^\theta_m$, whose energy estimate yields
\begin{equation}
\| u^\theta_m \|_{L^2} \leq |\Re \lambda|^{-1} \| f^\theta_m \|_{L^2} + C |\Re \lambda|^{-2} \| f^r_m \|_{L^2} \, ,
\end{equation}
which verifies~\eqref{eq:resolventestimatesforAm}.

(ii) We express the difference as the solution to the equation
\begin{equation}
	\label{eq:expressingthedifference}
(\lambda - A^{(\kappa)}_m) (u^{(\kappa)}_m - u_m) + \frac{\kappa}{\alpha} (1 + r\p_r) u_m = 0 \, ,
\end{equation}
where $u_m = R(\lambda,A_m) f_m$ and $u^{(\kappa)}_m = R(\lambda,A_m^{(\kappa)}) f_m$. To estimate the error, we compute $r \p_r u_m$ via the identity~\eqref{eq:identityforinvertingAm}:
\begin{equation}
\begin{aligned}
r\p_r u_m &= \frac{im r \Omega'}{(\lambda + im \Omega)^2} \begin{bmatrix}
1 & \\
\frac{r\Omega'}{\lambda + im \Omega} & 1
\end{bmatrix} f_m + \frac{1}{\lambda + im \Omega} \begin{bmatrix} 0
& \\
r \left( \frac{r\Omega'}{\lambda + im \Omega} \right)' & 0
\end{bmatrix} f_m \\
&\quad + R(\lambda,A_m) (r\p_r f_m) \, .
\end{aligned}
\end{equation}
(Notice the factor of $m$ arising in the first two terms.) Estimating this coarsely and plugging back into~\eqref{eq:expressingthedifference} yields~\eqref{eq:estimateforthedifference}.

(iii) This is a consequence of (ii) and approximating $L^2(r \, dr)$ functions $f$ by functions for which additionally $r \p_r f \in L^2(r \, dr)$.

(iv) We apply (ii) with $f=B_mu$, for some $u \in L^2(\RR)$. Thanks to the compactness of $B_m$ from \Cref{pro:decompositioncompactness} and the estimate \eqref{eq:B_gallay}, we can bound 
\begin{equation}
\|B_mu\|_{L^2} +  \|r\partial_r (B_m u)\|_{L^2} \leq (1+m) \|u\|_{L^2},
\end{equation}
which, plugged into \eqref{eq:estimateforthedifference}, proves \eqref{eq:estimatesondiffwithBm}.
\end{proof}

\begin{lemma}[Uniform resolvent estimates: Region I]
	\label{lem:uniformregion1}
The operator $R(\lambda,L_m^{(\kappa)})$ exists in the region $\{ \Re \lambda \gg \frac{1}{|m|} \}$
and satisfies
\begin{equation}
	\label{eq:uniformregionresolvent1}
\| R(\lambda,L_m^{(\kappa)}) \|_{L^2 \to L^2} \lesssim \frac{1}{\Re \lambda} + \frac{1}{(\Re \lambda)^2} \, .
\end{equation}
\end{lemma}
\begin{proof}
We consider the resolvent problem
\begin{equation}
(\lambda - L_m^{(\kappa)}) u_m = f_m \, ,
\end{equation}
whose corresponding Rayleigh equation is
\begin{equation}
	\label{eq:Rayleighequationselfsim}
\left( -\frac{\kappa}{\alpha} (2 + r\p_r ) + \lambda + im\Omega \right) \Delta_m \psi_m - \frac{im}{r} \psi_m W' = - \frac{im}{r} f^r_m + \p_r^* f^m_\theta \, .
\end{equation}
As in~\eqref{eq:Rayleighequationdivide}, we obtain an equivalent equation by inverting the operator $ -\frac{\kappa}{\alpha} (2 + r\p_r ) + \lambda + im\Omega$ preceding $\Delta_m \psi$:

For $\Re \lambda > \kappa/\alpha$, we consider two operators
\begin{equation}
S_{\kappa,m}(\lambda) := \left( -\frac{\kappa}{\alpha} (2 + r\p_r) + \lambda + im\Omega \right)^{-1} : L^2(r \, dr) \to L^2(r \, dr)
\end{equation}
\begin{equation}
\tilde{S}_{\kappa,m}(\lambda) := \left( \frac{\kappa}{\alpha} r\p_r + \bar{\lambda} - im\Omega \right)^{-1} : L^2(r \, dr) \to L^2(r \, dr) \, .
\end{equation}
These operators may be represented using either of two (equivalent) formulae. First, we have the Laplace transform formula
\begin{equation}
	\label{eq:formulaforS}
S_{\kappa,m}(\lambda) f = \int_0^{+\infty} e^{-\lambda t} u(\cdot,t) \, dt
\end{equation}
where $u$ is the solution to the initial value problem
\begin{equation}
\begin{aligned}
\p_t u -\frac{\kappa}{\alpha} (2 + r\p_r) u + im \Omega u &= 0 \\
u|_{t=0} &= f \, ,
\end{aligned}
\end{equation}
which may be solved by the method of characteristics. Alternatively, we may express $S_{\kappa,m}(\lambda) f$ by solving an ODE in the variable $r$ by Duhamel's formula (this is a rewriting of the $t$-integral in~\eqref{eq:formulaforS} in terms of an $r$-integral). Analogous formulas hold for $\tilde{S}_{\kappa,m}(\lambda)$. None of them is strictly necessary for the analysis below; however,~\eqref{eq:formulaforS} can be used to clarify the relationship
\begin{equation}
S_{\kappa,m}(\lambda)^* = \tilde{S}_{\kappa,m}(\lambda) \, ,
\end{equation}
that is, \emph{$\tilde{S}_{\kappa,m}(\lambda)$ is the $L^2(r \, dr)$-adjoint of $S_{\kappa,m}(\lambda)$.}


When $u = S_{\kappa,m}(\lambda) f$, we have
\begin{equation}
	\label{eq:doanenergyonme}
-\frac{\kappa}{\alpha} (2 + r\p_r) u + (\lambda + im\Omega) u = f \, .
\end{equation}
An energy estimate for~\eqref{eq:doanenergyonme} yields
\begin{equation}
(\Re \lambda - \frac{\kappa}{\alpha}) \| u \|_{L^2(r \, dr)} \leq \| f \|_{L^2(r \, dr)} \, ,
\end{equation}
and similarly for $\tilde{S}_{\kappa,m}(\lambda)$.

The $L^2$ estimates are not the most relevant for our purposes. Rather, when $\Re \lambda > 0$, the operator $\tilde{S}_{\kappa,m}(\lambda)$ is well defined on the spaces $\dot L^2_m$ and $\dot H^1_m$, consisting of locally integrable functions $\psi_m$ on $\R_+$ with finite norm, namely,
\begin{equation}
\| \psi_m \|_{\dot L^2_m} := \| m\psi_m/r \|_{L^2(r \, dr)}
\end{equation}
\begin{equation}
\| \psi_m \|_{\dot H^1_m} := \| \p_r \psi_m \|_{L^2(r \, dr)} + \| \psi_m \|_{\dot L^2_m} \, .
\end{equation}
In particular, we have the estimates
\begin{equation}
	\label{eq:tildeSest}
 \| \tilde{S}_{\kappa,m}(\lambda) f \|_{\dot L^2_m} \lesssim \frac{1}{\Re \lambda}  \| f \|_{\dot L^2_m} \, ,
\end{equation}
\begin{equation}
	\label{eq:tildeSest2}
 \| \tilde{S}_{\kappa,m}(\lambda) f \|_{\dot H^1_m} \lesssim \left( \frac{1}{\Re \lambda} + \frac{1}{(\Re \lambda)^2} \right) \| f \|_{\dot H^1_m} \, ,
\end{equation}
where the implicit constants are independent of $\kappa$ and $m$. (The sign difference in $S$ and $\tilde{S}$ matters these estimates.) To see this, we consider the equation for $u = \tilde{S}_{\kappa,m}(\lambda) f$:
\begin{equation}
\frac{\kappa}{\alpha}  r\p_r u + (\lambda + im\Omega) u = f \, .
\end{equation}
Then
\begin{equation}
	\label{eq:commutator2}
\frac{\kappa}{\alpha} (1 + r\p_r) \frac{mu}{r} + (\lambda + im\Omega) \frac{mu}{r} = \frac{mf}{r} \, ,
\end{equation}
\begin{equation}
	\label{eq:commutator1}
\frac{\kappa}{\alpha} (1 + r\p_r) (\p_r u) + (\lambda + im\Omega) \p_r u + \frac{imu }{r} (r \Omega') = \p_r f
\end{equation}
and energy estimates yield~\eqref{eq:tildeSest} and~\eqref{eq:tildeSest2} (notably, the energy estimate for~\eqref{eq:commutator1} relies on the estimates for~\eqref{eq:commutator2}). Consequently, the operator $S_{\kappa,m}(\lambda) : (\dot H^1_m)^* \to (\dot H^1_m)^*$ is well defined by duality whenever $\Re \lambda > 0$.


We now return to the equation~\eqref{eq:Rayleighequationselfsim}, which we rewrite as
\begin{equation}
\Delta_m \psi_m = im S_{\kappa,m}(\lambda) \left( \frac{W'}{r} \psi_m \right) - im S_{\kappa,m}(\lambda) \frac{f^r_m}{r} + S_{\kappa,m}(\lambda) \p_r^* f^m_\theta \, ,
\end{equation}
valid whenever $\Re \lambda > 0$. The energy estimate yields
\begin{equation}
\| \nabla_m \psi_m \|_{L^2}^2 = \langle \frac{im}{r} W' \psi_m, \tilde{S}_{\kappa,m}(\lambda) \psi_m \rangle - \langle \frac{im}{r} f^r_m, \tilde{S}_{\kappa,m}(\lambda) \psi_m \rangle + \langle f^\theta_m , \p_r \tilde{S}_{\kappa,m}(\lambda) \psi_m \rangle \, .
\end{equation}
Using~\eqref{eq:tildeSest}, we have
\begin{equation}
	\label{eq:iwillbecome}
\| \nabla_m \psi_m \|_{L^2}^2 \lesssim \frac{\| r W' \|_{L^\infty}}{m \Re \lambda} \| \nabla_m \psi_m \|_{L^2}^2 + \left( \frac{1}{\Re \lambda} + \frac{1}{(\Re \lambda)^2} \right) \| f \|_{L^2} \| \nabla_m \psi_m \|_{L^2} \, .
\end{equation}
When $\Re \lambda \gg 1/m$,~\eqref{eq:iwillbecome} becomes the estimate
\begin{equation}
\| \nabla_m \psi_m \|_{L^2} \lesssim \left( \frac{1}{\Re \lambda} + \frac{1}{(\Re \lambda)^2} \right) \| f \|_{L^2} \, .
\end{equation}
Finally, this can be adapted into a solvability proof when $\Re \lambda > 0$.
\end{proof}

\begin{proof}[Proof of Proposition~\ref{pro:spectralperturbation}]
Let $\varepsilon > 0$.

First, choose $M \gg \varepsilon^{-1}$ such that Lemma~\ref{lem:uniformregion1} guarantees the resolvent estimate~\eqref{eq:uniformregionresolvent1} on $R(\lambda,L^{(\kappa)}_m)$ whenever $|m| \geq M$ (see \Cref{fig:spectrum1}).

Second, restrict $\kappa \ll 1$ such that whenever $|m| \leq M$, Corollary~\ref{cor:Linversecorollary} (estimates on $R(\lambda,L_m) (\lambda - A_m)$), Lemma~\ref{lem:resolventsarecloseAm} (specifically, $m$-dependent estimates~\eqref{eq:estimatesondiffwithBm} on $[R(\lambda,A_m^{(\kappa)}) - R(\lambda,A_m)]B_m$), and Lemma~\ref{lem:stabilityofboundedinvertibility} (stability of bounded invertibility) guarantee that the resolvent $R(\lambda,L_m^{(\kappa)})$ satisfies uniform estimates in the regions
\begin{equation}
\{ |\Im \lambda| \gg M \} \cap \{ \Re \lambda \geq \varepsilon \}
\end{equation}
\begin{equation}
\{ \Re \lambda \gg 1 \} \, .
\end{equation}

Finally, we must account for spectrum in the compact region
\begin{equation}
	\label{eq:myregion}
\{ |\Im \lambda| \lesssim M \} \cap \{ \Re \lambda \in [\varepsilon,C] \}
\end{equation}
for the (finitely many) $|m| \leq M$ and sufficiently large $C$ (see \Cref{fig:spectrum2}). For this, we apply Proposition~\ref{pro:mainspectrallemma}, Proposition~\ref{pro:eigenvalueperturbation}, and its Corollary~\ref{cor:notwoevals} (their assumptions are satisfied, due to Lemma~\ref{lem:resolventsarecloseAm}) and further restrict $\kappa \ll 1$. When $|m| \leq M$ and $m \neq \pm m_0$, this guarantees that $L^{(\kappa)}_m$ has no spectrum in the region~\eqref{eq:myregion}. For $m = \pm m_0$, it guarantees the eigenvalue perturbation result advertised in the Proposition~\ref{pro:spectralperturbation}.
\end{proof}

\begin{corollary}[Semigroup estimate]
Let $0 < \kappa \ll 1$. Let $a = \Re \lambda(\kappa)$. The operator $L^{(\kappa)} : D(L^{(\kappa)}) \subset L^2_{m_0} \to L^2_{m_0}$ generates a continuous semigroup which satisfies the estimate
\begin{equation}
\label{eq:semigroup_estimate}
\| e^{\tau L^{(\kappa)}} \|_{L^2_{m_0} \to L^2_{m_0}} \lesssim e^{\tau a} \, , \quad \forall \tau \geq 0 \, .
\end{equation}
\end{corollary}
\begin{proof}
We decompose $H = {\rm ran} \, P \oplus {\rm ker} \, P$ (not necessarily orthogonal) into the invariant subspaces, where $P$ is the spectral projection onto the eigenvalues $\lambda(\kappa), \bar{\lambda}(\kappa)$. It is sufficient to estimate the semigroup on each subspace. Since the eigenvalues are semisimple (geometric and algebraic multiplicity coincide), we have the advertised semigroup estimate in ${\rm ran} \, P$. The operator on ${\rm ker} \, P$ has uniform resolvent estimates in $\{ \Re \lambda \geq \varepsilon \}$ ($\varepsilon < a$) so we may apply the Gearhardt-Pr{\"u}ss theorem, p. 302, Chapter 5, Theorem 1.11, in \cite{engel2000one}.
\end{proof}

\begin{lemma}[Eigenfunction bootstrapping]
For any eigenvalue $\lambda(\kappa) \in \mathbb C$ of $L^{(\kappa)}$ such that $\Re \lambda(\kappa) > k/\alpha -1$, the corresponding eigenfunction $\eta$ belongs to the space $H^{k+1}(\RR)$.
\end{lemma}
\begin{proof}
The proof follows by Proposition~2.1 in \cite{MR4616679}, which, under the hypothesis $\Re \lambda > k/\alpha -1$ proves that the unstable eigenfunction for the operator in vorticity formulation $\nabla^\perp \cdot \eta$ belongs to $H^k(\RR)$ and $\eta$, the unstable eigenfunction for the operator in velocity formulation, thanks to the equivalence in \Cref{lemma:equivalence_egfn}, belongs to $H^{k+1}(\RR)$. 
\end{proof}

\begin{remark}
\label{eq:bound_eta}
For an unstable eigenvalue such that $\Re \lambda > 4/\alpha$ as in \Cref{prop:lin_part}, we can prove that $\eta \in W^{4,2+\sigma}(\RR)$, which is the level of integrability required by the nonlinear argument in \Cref{sec:nonlinear}.
\end{remark}

\begin{remark}
\label{rmk:reeta}
The particular structure of the eigenfunction, $\eta(r,\theta) = e^{im\theta}f(r)$, given by the fact that we are projecting the operator onto \eqref{eq:def_Xm} implies that
\begin{equation}
    \|\Re(z\eta)\|_{L^2} = \frac{\sqrt 2}{2}, \quad \forall z \in \mathbb C , |z|=1.
\end{equation}
Indeed, let $z = e^{i\psi}$ for some $\psi \in \mathbb R$ and let $\theta_0 = \pi/2m$; since $e^{im\theta + i \psi} f(r) = i e^{im(\theta-\theta_0) + i \psi} f(r)$, we have 
\begin{align} 
\|\Re(z\eta)\|^2_{L^2} & = \frac{1}{2} \intR \int_{0}^{2\pi} |\Re(e^{im\theta + i \psi} f(r))|^2 rd\theta dr \\ & \qquad + \frac{1}{2} \intR \int_{0}^{2\pi} |\Im(e^{im(\theta-\theta_0) + i \psi} f(r))|^2 rd\theta dr \\ & =  \frac{1}{2} \| \eta\|_{L^2}=  \frac{1}{2},
\end{align}
where the last equality follows from a change of variables in the last integral in $\theta$. This property of $\eta$ is not essential for the rest of the paper, but it simplifies the choice of certain numerical functions in \Cref{sec:final}.
\end{remark}

\subsection{Spectral Projection}
\label{sseq:spectral_projection}
We recall that the spectral projection onto an eigenfunction $\eta$, associated to the eigenvalue $\lambda$, is defined as 
\begin{equation}
    P_\lambda \coloneqq \frac{1}{2\pi i} \int_{\gamma(\lambda)} \, R(\zeta,L_k)d\zeta,
\end{equation}
where $\gamma(\lambda)$ is is a contour in the complex plane that encloses only the eigenvalue $\lambda$ and excludes the rest of the spectrum. We can define analogously $P_{\bar\lambda}$, projection onto the eigenfunction $\bar \eta$ corresponding to the eigenvalue $\bar\lambda$, and introduce the notation $P_{\lambda,\bar\lambda}\coloneqq P_{\lambda}+P_{\bar\lambda}$.

We recall from the spectral results on linear operators in Chapter 3 of \cite{kato2013perturbation} that the spectral projections above defined have the following properties: For any $ c_1,c_2 \in \mathbb C$ and $f,g \in L^2_{m_0}(\RR)$,
\begin{align}
\label{eq:prop_linear}
P_\lambda (c_1f + c_2g) = c_1 P_\lambda f + c_2P_\eta g, \ & P_{\bar\lambda} (c_1f + c_2g) = c_1P_{\bar\lambda} f + c_2P_{\bar\lambda} g, \\ 
\|P_\lambda\|_{L^2 \to L^2} + \|P_{\bar\lambda}\|_{L^2 \to L^2}  & \leq C, \text{ for some $C>0$,}\label{eq:prop_bound}
\\
P_{\lambda} \eta = \eta, \ P_\lambda \circ P_\lambda & = P_\lambda, \ P_\lambda \circ P_{\bar \eta }= 0, \label{eq:prop_projector}\\
P_{\bar\lambda} \bar f & = \overline{P_\lambda f}. \label{pietabar}
\end{align}

\begin{remark}
\label{rmk:petabar}
A consequence of the regularity of the eigenfunction is the following fact, which will be useful in \Cref{sec:nonlinear} and \Cref{sec:final}. For a constant $C: = C(\lambda, \eta)$ we have
\begin{equation}
    \| e^{\tau L_{ss}} P_{\lambda,\bar{\lambda}} \Phi \|_{W^{4,2+\sigma}(\RR)} \leq C \| \Phi \|_{L^{2}(\RR)} \qquad \mbox{for any } \tau\geq 0, \Phi\in L^2_{m_0}(\RR) \, .
\end{equation}
Indeed, we rewrite $P_{\lambda,\bar{\lambda}} \Phi$, which is a real-valued function, as a complex-coefficient combination of $\eta$ and $\bar \eta$, as $P_{\lambda,\bar{\lambda}} \Phi= z \eta + \bar z  \bar\eta $ for some $z \in \mathbb C$.
Applying the spectral projection $P_\lambda$, we find $z \eta = P_\lambda \Phi$. Taking the $L^2$ norm on both sides and recalling that $\|\eta\|_{L^2(\RR)}=1$ and the boundedness of the spectral projection \eqref{eq:prop_bound}, we deduce that 
$| z| = \|  P_\lambda \Phi \|_{L^2(\RR)} \leq C \|  \Phi \|_{L^2(\RR)}.$
Hence,
\begin{align}
    \|e^{\tau L_{ss}} P_{\lambda,\bar{\lambda}} \Phi\|_{W^{4,2+\sigma}(\RR)} & = \|z e^{\lambda \tau} \eta + \bar z e^{\bar \lambda \tau}  \bar\eta \|_{W^{4,2+\sigma}(\RR)}  \\ & \leq C|z| e^{a\tau} \| \eta \|_{W^{4,2+\sigma}(\RR)} \leq C e^{a\tau}  \|  \Phi \|_{L^2(\RR)} \, .
\end{align}
\end{remark}

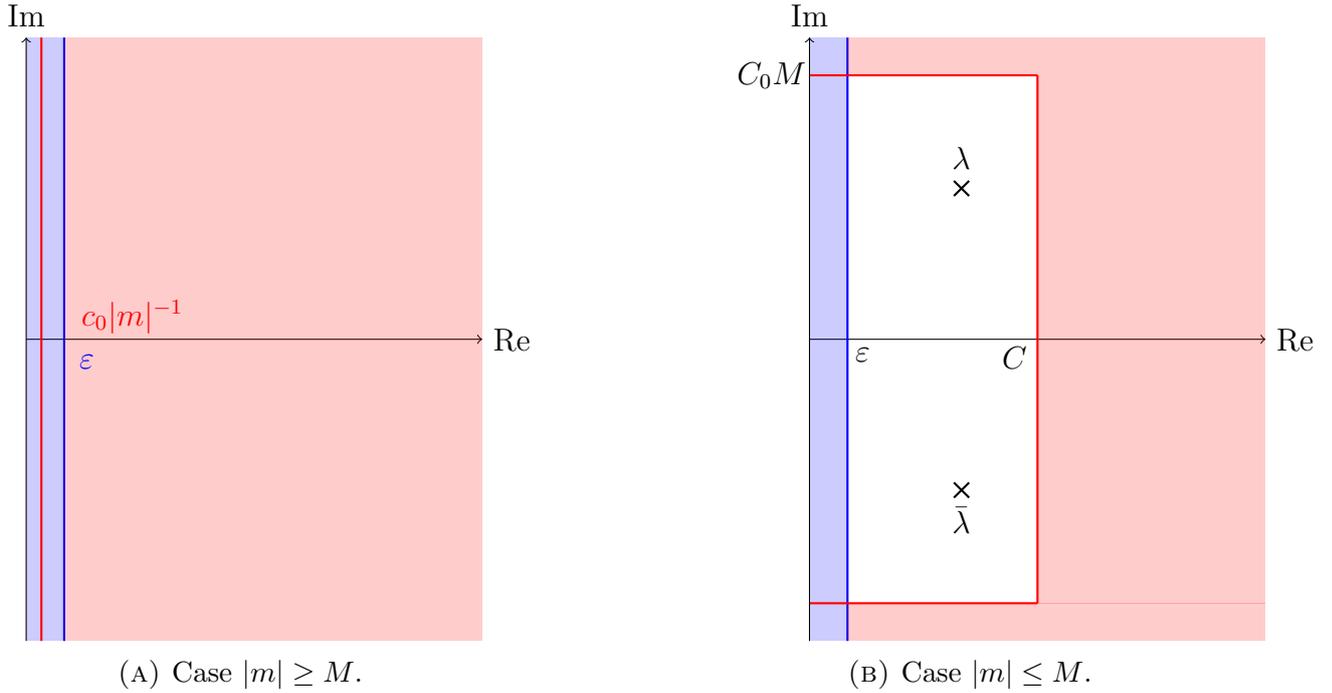
\begin{figure}[H]
    \centering
    \begin{subfigure}{.4\linewidth}
        \centering
        \begin{tikzpicture}
        
        \draw[->] (0,0) -- (6,0) node[right] {$\Re$};
        
        \draw[blue, thick] (0.5,-4) -- (0.5,4);
        
        \fill[blue,opacity=0.2] (0,-4) rectangle (0.5,4);
        
        \draw[red, thick] (0.2,-4) -- (0.2,4);

        \fill[red,opacity=0.2] (0.5,-4) rectangle (6,4);
        
        \draw[->] (0,-4) -- (0,4) node[above] {$\Im$};
        
        \draw (0.8, -0.3) node[blue] {$\varepsilon$};
        \draw (1.4, 0.3) node[red] {$c_0|m|^{-1}$};
        
        \end{tikzpicture}
        \caption{Case $|m|\geq M$.}
        \label{fig:spectrum1}
    \end{subfigure}
    \hfill
    \begin{subfigure}{.4\linewidth}
        \centering
        \begin{tikzpicture}
        
        \draw[->] (0,0) -- (6,0) node[right] {$\Re$};
        
        \draw[blue, thick] (0.5,-4) -- (0.5,4);
        
        \fill[blue,opacity=0.2] (0,-4) rectangle (0.5,4);
        
        \draw[red, thick] (3,-3.5) -- (3,3.5);
        
        \draw[red, thick] (0,-3.5) -- (3,-3.5);
        \draw[red, thick] (0,3.5) -- (3,3.5);
        
        \fill[red,opacity=0.2] (0.5,-4) rectangle (6,-3.5);
        \fill[red,opacity=0.2] (0.5,3.5) rectangle (6,4);
        \fill[red,opacity=0.2] (3,-3.5) rectangle (6,3.5);
        
        \draw[->] (0,-4) -- (0,4) node[above] {$\Im$};
        
        \draw (0.7, -0.2) node[black] {$\varepsilon$};
        
        \draw (2.7, -0.25) node[black] {$C$};
        
        \draw (-0.5, 3.5) node[black] {$C_0M$};
        
        \draw[thick] (1.9,-1.9) -- (2.1,-2.1);
        \draw[thick] (1.9,-2.1) -- (2.1,-1.9);
        \draw (2, -2.4) node[black] {$\bar\lambda$};
        \draw[thick] (1.9,1.9) -- (2.1,2.1);
        \draw[thick] (1.9,2.1) -- (2.1,1.9);
        \draw (2, 2.4) node[black] {$\lambda$};
        
        \end{tikzpicture}
        \caption{Case $|m|\leq M$.}
        \label{fig:spectrum2}
    \end{subfigure}
    \caption{Schematic depiction of the strategy for controlling the resolvent $R(\lambda,L^{(\kappa)}_m$ in $\{ \Re \lambda \geq \varepsilon \}$. In (A) and in the red half-plane in (B), resolvent estimates were proved in the red region by directly estimating the $\kappa$ Rayleigh equation in~\eqref{lem:uniformregion1}. Above and below the box in (B), and inside the box but away from the eigenvalues, resolvent estimates were proved by perturbing the $\kappa=0$ estimates.}
\end{figure}

%% file: modifiedbackground.tex
\section{The modified background}
\label{sec:modifiedbackground}

In this section, we examine the modified background introduced in Section~\ref{sec:intromodifiedbackground}. 

Let $\alpha \in (0,2)$ and consider a divergence-free $\bar{U} \in C^\infty_0(\R^2)$, $\bar{u} = t^{\frac{1}{\alpha}-1} \bar{U}(\xi)$, and
\begin{equation}
    \label{eq:selfsimilarscalingoff}
\bar{f} := \p_t \bar{u} = t^{\frac{1}{\alpha}-2} \bar{F}(\xi,\tau) \, ,
\end{equation}
as in the notation in Section~\ref{sec:inviscidnonuniqueness}. (The instability of $L_{\rm ss}$ is not relevant to this section.)

The modified background $\tilde{u}$ is the solution to
\begin{equation}
\label{eq:lemma1}
\left\lbrace
\begin{aligned}
\partial_t \tilde{u} &= \Delta \tilde{u} + \bar{f} \quad \text{ in } \RR \times \R^+ \\
\tilde{u}|_{t=0} &= 0 \, .
\end{aligned}
\right.
\end{equation}
Since $\tilde{u}$ is radially symmetric, $\tilde{u} \cdot \nabla \tilde{u} = 0$, and the modified background also solves the forced Navier-Stokes equations.

We will investigate both the short-time smoothing and long-time convergence of $\tilde{u}$.




\subsection{Short-time smoothing}
\begin{remark}
\label{rmk:tildes_Lp}
We observe that 
\begin{equation}
\label{eq:utildebounded2/alpha}
    \tilde{u} \in L^\infty(0,T;W^{1,p}(\R^2)) \, , \quad \forall p \in \Big[1,\frac{2}{\alpha}\Big] \, , \quad \forall T > 0 \, .
\end{equation} 
In particular, $\tilde u$ belongs also to the critical space with $p=2/\alpha$.
Indeed, by an  energy estimate and a direct computation from~\eqref{eq:selfsimilarscalingoff}
\begin{equation}
\label{eq:heat_estimate2}
\norm{\tilde{u}(\cdot,t)}_{L^2_x} \leq \int_0^t \norm{\bar{f}(\cdot,s)}_{L^2_x} \, ds = Ct^{2/\alpha-2} \, , \quad \forall t>0 \, .
\end{equation}

By Duhamel's formula we get
\begin{align}
    \|\tilde \omega(\cdot,t)\|_{L^{2/\alpha}_x} & \leq \int_0^t \|e^{(t-s)\Delta}\|_{L^1\to L^2} \| \bar f(\cdot,s)\|_{W^{1,2/(1+\alpha)}_x} ds \\ & \leq C \int_0^t (t-s)^{-1/2} s^{1/\alpha - 1} ds = Ct^{1/\alpha - 1/2}.
\end{align}
\end{remark}

Next, we demonstrate that $\tilde{u}$ is more regular than Vishik's vortex $\bar{u}$, which belongs only to $L^2_t H^{2-\frac{\alpha}{2}-}_x$. In particular, $\tilde{u}$ enjoys some $L^2_t H^{2+}_x$ regularity which will be useful in the initial layer argument in Section~\ref{sec:back_uniq}.

\begin{lemma}[Short-time smoothing]
\label{lemma:usmooth}
Let $\alpha \in (0,4/3)$. Then
\begin{align}
\label{eqn:higherutilde}
    \tilde{u} \in L^2(0,T; H^{2+\beta}(\R^2)) \, , \quad \forall \beta \in [0, 2-3\alpha/2) \, , \quad \forall T > 0.
\end{align}
\end{lemma}
\begin{proof}
Let $\beta \geq 0$. By the self-similar scaling~\eqref{eq:selfsimilarscalingoff}, $\norm{\bar f}_{\dot H^\beta(\RR)} = t^{\frac{2-\beta}{\alpha} -2} \norm{\bar F}_{\dot H^\beta(\RR)}$ belongs to $L^2_t(0,T)$ under the further condition $\beta < 2-3\alpha/2$. When $\alpha \in (0,4/3)$, we have $2-3\alpha/2 > 0$, which ensures that the condition on $\beta$ is non-empty. Finally, by maximal regularity  
for the heat equation, 
we obtain that $\tilde u \in L^2(0,T; \dot H^{2+\beta}(\RR))$. Since $\tilde u \in L^2(0,T;L^2(\RR))$ by~\Cref{rmk:tildes_Lp}, we have $\tilde u \in L^2(0,T; H^{2+\beta}(\RR))$.
 
\end{proof}

\subsection{Long-time convergence $\tilde{U} \to \bar{U}$}



As discussed in Sections~\ref{sec:intromodifiedbackground} and~\ref{sec:introlongtime}, it is crucial to ensure that the profile of the modified background $\tilde{U}$ represents a good approximation to $\bar{U}$ as $\tau \to +\infty$. The underlying intuition is that the viscous term in, say, the vorticity equation
\begin{equation}
    \label{eq:vorticityequationselfsimheat}
    \p_\tau \tilde{\Omega} - (1 + \frac{1}{\alpha} \cdot \nabla) \tilde{\Omega} = e^{-\tau \gamma} \Delta \tilde{\Omega} + \bar{F}
\end{equation}
becomes progressively less significant as time evolves. However, this question is quite sensitive to the mode of convergence. Convergence of the profiles in high regularity 
would be a strong indication that $\tilde{u}$ lives at length scales $\gtrsim t^{\frac{1}{\alpha}}$.

Consider the equation
\begin{equation}
    \label{eq:heatfordifference}
    \p_t (\tilde{\omega} - \bar{\omega}) = \Delta( \tilde{\omega} - \bar{\omega}) + \Delta \bar{\omega} \, .
\end{equation}
If we consider $\Delta$ to be asymptotically negligible, then the main difficulty is that the ``radial Euler equations" $\p_t \omega = 0$ \emph{do not have a stabilizing mechanism}. Any decay must therefore be due to the choice of topology, heat, or specifics of the source $\Delta \bar{\omega}$.

The zeroth order approach is write the equation~\eqref{eq:vorticityequationselfsimheat} for the difference of the profiles and perform an $L^p$ energy estimate. This yields convergence in the supercritical spaces $L^p$, $p < 2/\alpha$. The action of the term $- (1 + \frac{1}{\alpha} \cdot \nabla)$ is to concentrate $\tilde{\Omega}$ at the origin and amplify it: $r^{-\alpha}$ is a steady state. One way to visualize this process is that at time $t$, the source $\Delta \bar{\omega}$ ``leaves behind" a contribution at length scale $\sim t^{\frac{1}{\alpha}}$ which slowly diffuses, and this looks like concentration in the profile variables. The next order approach is build in the diffusion using estimates for the heat semigroup on~\eqref{eq:heatfordifference}. This yields convergence in $L^p$, $p < +\infty$, not good enough for the nonlinear theory. Ultimately, the way out is to exploit that the source $\Delta \bar{\omega}$ is \emph{cancelling itself in the time integration}~\eqref{eq:Icancelmyselflikeaboss}.

\begin{lemma}
\label{lemma:conv_background}

For all $\alpha \in (0,2)$, $\tau \geq 0$, and $p \in [1,+\infty)$, we have
\begin{equation}
\label{eq:conv_background_velocity}
\| \tilde U - \bar U\|_{L^2} \leq C e^{(1-\frac{2}{\alpha})\tau}
\end{equation}
\begin{equation}
\label{eq:conv_background}
\| \nabla (\tilde \Omega - \bar \Omega)\|_{L^p} + \|r\partial_r \nabla (\tilde \Omega - \bar \Omega)\|_{L^p} \leq C e^{\zeta(\alpha,p) \tau} \, ,
\end{equation}
where
\begin{equation}
    \zeta(\alpha,p) = -\frac{3}{2} + \frac{1}{\alpha} - \frac{2}{\alpha p}  + \frac{1}{p} + \frac{\alpha(5p-2)}{6p+2\alpha p-4} \, .
\end{equation}
\end{lemma}

One could as well prove that the profiles diverge in time in too strong norms due to concentration at the origin imposed by the $- (1 + \frac{1}{\alpha} \cdot \nabla)$ term. 

\begin{remark}
\label{rmk:sigmacond}
In~\Cref{sec:nonlinear}, we require positive constants $\sigma > 0$ and $\zeta \in (0,\gamma)$ depending only on $\alpha$ such that
\begin{align}
\label{eq:ineq_zeta}
\| \tilde U - \bar U \|_{L^2} + \|\nabla (\tilde \Omega - \bar \Omega)\|_{L^2 \cap L^{2+\sigma}} & + \|r\partial_r \nabla (\tilde \Omega - \bar \Omega)\|_{L^2 \cap L^{2+\sigma}} \\ & \qquad  \leq C t^{-\zeta} \, , \quad \forall t \geq 1 \, .
\end{align}
Since~\eqref{eq:conv_background} for $p=2$ decays in time for any $\alpha \in (0,2)$, and since $\zeta(\alpha,p)$ is a continuous function in $(\alpha, p)$, when letting $p=2+\sigma$, we can see that, for any $\alpha > 0$, there exists $\sigma > 0$ such that \eqref{eq:conv_background} decays. 
Moreover, by \eqref{eq:ineq_zeta}, it follows also that
\begin{equation}
    \|\tilde U - \bar U\|_{L^\infty} + \| \nabla(\tilde U - \bar U) \|_{L^\infty} + 
\|\tilde \Omega - \bar \Omega\|_{L^2 \cap L^\infty} \leq  C t^{-\zeta},
\end{equation}
for all $t>0.$ Indeed, via Gagliardo-Nirenberg and Calderón-Zygmund 
 we have:
\begin{align}
    \|\tilde U - \bar U\|_{L^\infty} +
    \|\tilde \Omega - \bar \Omega\|_{L^2} + \|\nabla(\tilde U - \bar U)\|_{L^2}  & \leq C\|\nabla (\tilde \Omega - \bar \Omega) \|^{1/2}_{ L^{2}} \| \tilde U - \bar U\|^{1/2}_{L^2}, 
    \end{align}
\begin{align}
    \|\tilde \Omega - \bar \Omega\|_{L^\infty} + \|\nabla(\tilde U - \bar U)\|_{L^\infty} & \leq C\| \nabla (\tilde \Omega - \bar \Omega) \|^\frac{\sigma}{4+3\sigma}_{ L^{2+\sigma}} \| \nabla (\tilde U - \bar U)\|^\frac{4+2\sigma}{4+3\sigma}_{L^2}.
\end{align}


\end{remark}

\begin{proof}[Proof of Lemma~\ref{lemma:conv_background}]
\textbf{Proof of estimate \eqref{eq:conv_background_velocity}:}
The equation solved by the difference of the profiles $D_u \coloneqq \tilde U - \bar U$ in self-similar variables reads as:
\begin{equation}
\label{eq:diff_vel_ss}
\begin{cases}
\partial_\tau D_u - \left(1-1/\alpha + \xi/\alpha \cdot \nabla \right) D_u - e^{-\tau \gamma}\Delta D_u = e^{-\tau \gamma}\Delta \bar U, \\
\lim_{\tau \to -\infty} D_u(\cdot, \tau) = 0.
\end{cases}
\end{equation}
Performing an energy estimate on \eqref{eq:diff_vel_ss}, we get
\begin{equation}
    \partial_\tau \norm{D_u}_{L^2} + 2\gamma \norm{D_u}_{L^2} \leq C e^{-\tau\gamma},
\end{equation}
which directly gives \eqref{eq:conv_background_velocity} by Gronwall's Lemma.

\textbf{Proof of estimate \eqref{eq:conv_background}, $\| \cdot \|_{L^p}$:} Define $d_\omega \coloneqq \tilde{\omega}- \bar \omega$ and $D_\omega \coloneqq \tilde{\Omega}- \bar \Omega$. By scaling, we have
\begin{equation}
\label{eq:scalingnablaD}
\norm{\nabla D_\omega}_{L^p} + \norm{r\partial_r\nabla D_\omega}_{L^p} = e^{\tau(1+1/\alpha-2/(\alpha p))}\big[ \norm{\nabla d_\omega}_{L^p}+ \norm{r\partial_r\nabla d_\omega}_{L^p} \big].
\end{equation}
and \eqref{eq:conv_background} is equivalent to prove
\begin{equation}
\label{eq:conv_background-new}
\norm{\nabla d_\omega}_{L^p} + \norm{r\partial_r\nabla d_\omega}_{L^p}\leq C  
t^{-5/2 + 1/p + \alpha(5p-2)/(6p+2\alpha p-4)}.
\end{equation}
Observing that $\nabla d_\omega$ solves the heat equation forced by $ \nabla \Delta \bar \omega$
with $0$ initial condition,
we can study the bounds for $\nabla d_\omega$ via Duhamel's formula:
\begin{align}
\label{eq:d_split}
\nabla d_\omega & = \int_0^t e^{(t-s)\Delta} \nabla \Delta \bar \omega ds \\ & = \int_0^1 (e^{(t-s)\Delta} - e^{t\Delta}) \nabla \Delta \bar \omega ds + \int_{t^\beta}^t e^{(t-s)\Delta} \nabla \Delta \bar \omega ds \\ 
& + \int_1^{t^\beta} (e^{(t-s)\Delta} - e^{t\Delta}) \nabla \Delta \bar \omega ds + \int_0^{t^\beta} e^{t\Delta} \nabla \Delta \bar \omega ds.
\end{align}
where $\beta >0$ will be optimized later. It is convenient to split the integral in~\eqref{eq:d_split} into four pieces as above. Indeed the last piece will exploit a cancellation in the integral with respect to $s$; however, to see this cancellation we needed to freeze the time by replacing $e^{(t-s)\Delta}$ with $e^{t\Delta}$, which is a good approximation only if $s$ is up to order $t^\beta$; therefore, we introduce the second and third error term. The first error term is smaller than all the others for large $t$ due to the heat effect.

To proceed with the bounds, we  recall the following estimates obtained by explicit computation on the heat kernel $K(x,t)$, that we need for $i=0,1$, $k =0,1$ and $r=0,1,2,3$: there exists $C>0$ such that
\begin{align}
\label{eq:heatdt}
 \| (\partial_t)^i (r \partial_r)^k ( \nabla )^m e^{(t-\sigma)\Delta}\|_{L^1 \to L^p} = \norm{(\partial_t)^i (r \partial_r)^k (\partial_r)^m K(x,t)}_{L^p} & \leq C t^{-1-i-m/2+1/p}.
\end{align}
Since $\bar \omega$ is self-similar, by a change of variable we have
\begin{align}
    \norm{\nabla \Delta \bar \omega(t)}_{L^p} +\norm{r\partial_r \nabla\Delta \bar \omega(t)}_{L^p}  & = t^{-1-3/\alpha + 2/(\alpha p)} (\norm{\nabla \Delta \bar \Omega}_{L^p}+\norm{r\partial_r \nabla \Delta \bar \Omega}_{L^p}) \\& \leq C t^{-1-3/\alpha + 2/(\alpha p)}.\label{eq:decaybaromega}
\end{align}

The first contribution in~\eqref{eq:d_split} can be dealt with by exploiting the decay of the heat kernel:
\begin{align}
\norm{\int_0^1 (e^{(t-s)\Delta} - e^{t\Delta}) \nabla \Delta \bar \omega ds}_{L^p} & \leq \int_0^1 \big(\norm{e^{(t-s)\Delta} \nabla \Delta}_{L^1 \to L^p}+\norm{ e^{t\Delta} \nabla \Delta}_{L^1 \to L^p} \big) \norm{\bar \omega}_{L^1} ds\\ & \leq Ct^{-1+1/p-3/2} \int_0^1 s^{2/\alpha - 1} ds \\
& \leq Ct^{-5/2 + 1/p}.
\end{align}

For the second contribution, we exploit instead the decay in time of the forcing term in~\eqref{eq:decaybaromega}:
\begin{align}
\norm{\int_{t^\beta}^t e^{(t-s)\Delta} \nabla \Delta \bar \omega ds}_{L^p} & \leq \int_{t^\beta}^t \norm{e^{(t-s)\Delta}}_{L^p \to L^p} \norm{\nabla \Delta \bar \omega}_{L^p} ds \\ & \leq C \int_{t^\beta}^t s^{-3/\alpha - 1 + 2/(\alpha p)} ds \leq C t^{-\beta/\alpha (3-2/p)}. \label{eq:second}
\end{align}

For the third contribution, by~\eqref{eq:heatdt} and  Minkowski's inequality, we first estimate, for $s\in [1,t^\beta]$ {and $t$ large enough to ensure $t^\beta \leq t/2$,}
\begin{align}
\|(e^{(t-s)\Delta} - e^{t\Delta})\nabla \|_{L^1 \to L^p} & = \norm{\nabla(K(\cdot,t-s) - K(\cdot,t))}_{L^p}
 \leq \int_0^s \norm{\partial_{\sigma} \nabla K(x,t-\sigma)}_{L^p} d\sigma \\
& \leq s\sup_{\sigma \in (0,s)} \| \partial_t \nabla e^{(t-\sigma)\Delta}\|_{L^1 \to L^p} \leq Cs t^{- 2 - 1/2 + 1/p} \, .
\label{eq:timederLp}
\end{align}
Hence, we have
\begin{align}
\norm{\int_{1}^{t_\beta} (e^{(t-s)\Delta} - e^{t\Delta}) \nabla \Delta \bar \omega ds}_{L^p} & \leq \int_1^{t^\beta} \norm{(e^{(t-s)\Delta} - e^{t\Delta}) \nabla}_{L^1 \to L^p} \norm{\Delta \bar \omega}_{L^1} ds 
\\
& \leq C t^{- 2 - 1/2 + 1/p} \int_1^{t^\beta} s s^{-1} ds \leq C t^{\beta - 5/2 + 1/p}. \label{eq:third}
\end{align}

Finally, to analyze the fourth contribution, we define
\begin{equation}
\label{def:f}
    f(x) = \int_{0}^{t^\beta} \nabla \Delta \bar \omega \, ds \overset{\; \tilde{s} = s^{-1/\alpha}}{=} \alpha\int_{t^{-\beta/\alpha}}^{\infty} \tilde s^{2} \nabla \Delta \bar \Omega(\tilde sx) \, d\tilde s = - \alpha \int_{0}^{t^{-\beta/\alpha}} \tilde s^2 \nabla \Delta \bar \Omega(\tilde sx) \, d\tilde s \, .
\end{equation}
In the third equality we exploited a cancellation in the integral due to the fact that  $\nabla\Delta\bar{\Omega}$ is a radially symmetric function and hence for every $x\in \R^2$
\begin{align}
    \label{eq:Icancelmyselflikeaboss}
2 \pi \int_{0}^{+\infty} \tilde s^2 \p_r \Delta \bar \Omega(\tilde sx) d\tilde s  = \int_{\mathbb R^2} y \cdot (\nabla \Delta \bar \Omega) (|x|y) \, dy
 = 0.
\end{align}
Hence, since $\norm{\tilde s^2 \nabla \Delta \bar \Omega(\tilde s\cdot)}_{L^p} = \tilde s^{2-2/p} \norm{\nabla \Delta \bar \Omega}_{L^p}$ by scaling, we get that
\begin{equation}
\label{eq:fourth}
\norm{  e^{t\Delta} \int_0^{t^\beta} \nabla \Delta \bar \omega ds}_{L^p} \leq \norm{f}_{L^p}  \leq \alpha \int_{0}^{t^{-\beta/\alpha}} \tilde s^{2-2/p} \norm{\nabla \Delta \bar \Omega}_{L^p} d\tilde s \leq C t^{-\beta/\alpha (3-2/p)}.
\end{equation}
To optimize the choice of $\beta$, we ask the contributions in~\eqref{eq:second} and~\eqref{eq:fourth}, which are of the same magnitude, to be balanced with~\eqref{eq:third}; therefore, we impose
\begin{equation}
-\frac{\beta}{\alpha} \left(3 - \frac{2}{p}\right) = \beta - \frac{5}{2} + \frac{1}{p} \quad \iff \quad \beta = \frac{\alpha(5p-2)}{2(3p+\alpha p -2)}.
\end{equation}
and obtain the decay in time for $\norm{\nabla d_\omega}_{L^p}$  claimed in \eqref{eq:conv_background-new}.

\textbf{Proof of estimate \eqref{eq:conv_background}, $\| r \p_r \cdot \|_{L^p}$:} To bound the second term in the left-hand side of \eqref{eq:conv_background-new}, we apply the $r\partial_r$ operator to the rewriting of $\nabla d_\omega$ in \eqref{eq:d_split}, obtaining
\begin{align}
\label{eqn:rdrdomega}
r\partial_r(\nabla d_\omega) & = \int_0^1 r\partial_r \left((e^{(t-s)\Delta} - e^{t\Delta}) \nabla \Delta \bar \omega\right)   ds \\ & 
+\int_{t^\beta}^t  r\partial_r\left(e^{(t-s)\Delta} \nabla \Delta \bar \omega\right) ds \\ 
& +\int_1^{t^\beta} r\partial_r \left( (e^{(t-s)\Delta} - e^{t\Delta}) \nabla \Delta \bar \omega \right)ds + r\partial_r\left( e^{t\Delta} f\right).
\end{align}
The reason why $r\partial_r(\nabla d_\omega) $ has the same decay in time as $\nabla d_\omega$ relies on the fact that the operator $r \partial_r$ does not affect the decay of the heat kernel and of the self-similar force, as expressed in \eqref{eq:decaybaromega} and \eqref{eq:heatdt} above. For the sake of completeness, we exploit this fact in estimating the four terms in \eqref{eqn:rdrdomega}, which are therefore analogous to the corresponding ones in \eqref{eq:d_split}.

First of all, since the $r\partial_r = x \cdot \nabla$ operator has the following behavior with respect to convolution $r\partial_r (f \ast g) = (r\partial_r f) \ast g + f \ast (r\partial_r g) + 2 f\ast g
$, we have for all $0\leq s<t$:
\begin{equation}
\label{eq:rdrconv}
r\partial_r (e^{(t-s)\Delta} \nabla \Delta \bar \omega) = (\nabla \Delta r\partial_r e^{(t-s)\Delta}) \bar \omega + (e^{(t-s)\Delta} \nabla \Delta) (r\partial_r \bar \omega  + 2 \bar \omega).
\end{equation}
For $s \in[0,1)$, by \eqref{eq:heatdt} all the operators appearing in the previous formula are bounded between $L^1$ and $L^p$ by $ t^{-1+1/p-3/2}$. Hence we use this expression for $s$ and $s=0$,  to bound the first term in the expression for ${r \partial_r \nabla d_\omega}$ in the right-hand side of \eqref{eqn:rdrdomega}
\begin{align} 
 \norm{\int_0^1 r\partial_r \left((e^{(t-s)\Delta} - e^{t\Delta}) \nabla \Delta \bar \omega\right)   ds}_{L^p} 
&\leq 2 \int_0^1 
t^{-1+1/p-3/2} (\norm{r\partial_r \bar \omega}_{L^1} + \norm{\bar \omega}_{L^1}) ds\\
& \leq C t^{-1+1/p-3/2}\int_0^1 s^{-1+2/\alpha} ds \\ & \leq C t^{-5/2+1/p}.
\end{align}

The second term in \eqref{eqn:rdrdomega} can be bounded observing that
\begin{equation}
\label{eq:rdrconv-2}
r\partial_r (e^{(t-s)\Delta} \nabla \Delta \bar \omega) = (r\partial_r e^{(t-s)\Delta}) ( \nabla \Delta\bar \omega) + (e^{(t-s)\Delta} ) (r\partial_r \nabla \Delta\bar \omega+ 2 \nabla \Delta \bar \omega).
\end{equation}
and that all operators, seen from $L^p$ to $L^p$, are bounded by a constant. Hence,
\begin{align}
\norm{\int_{t^\beta}^t  r\partial_r\left(e^{(t-s)\Delta} \nabla \Delta \bar \omega\right) ds}_{L^p}  &\leq \int_{t^\beta}^t  \left(\norm{\nabla \Delta \bar \omega}_{L^p} + \norm{r\partial_r \nabla \Delta \bar \omega}_{L^p}\right) ds \\ & \leq C \int_{t^\beta}^t s^{-1-3/\alpha + 2/(\alpha p)} ds \leq C t^{-\beta/\alpha (3-2/p)}.
\end{align}

For the third term, we have
\begin{align}
    r\partial_r \left( (e^{(t-s)\Delta} - e^{t\Delta}) \nabla \Delta \bar \omega \right) & = (r\partial_r (e^{(t-s)\Delta} - e^{t\Delta}) \nabla ) ( \Delta\bar \omega)  \\ &\quad + ((e^{(t-s)\Delta}  - e^{t\Delta})\nabla ) (r\partial_r \Delta\bar \omega+ 2  \Delta \bar \omega) \, ,
\end{align}
and proceeding as for \eqref{eq:timederLp}, we get a bound for the operators for all $s<t^\beta$:
\begin{align}
\norm{r\partial_r (e^{(t-s)\Delta} - e^{t\Delta})\nabla}_{L^1 \to L^p} \leq s \sup_{\sigma \in (0,s)} \norm{\partial_\sigma r\partial_r  (\nabla e^{(t-\sigma)\Delta})}_{L^1 \to L^p} \leq s t^{-5/2+1/p} \, .
\end{align}
Therefore,
\begin{align}
 \norm{\int_1^{t^\beta} r\partial_r \left( (e^{(t-s)\Delta} - e^{t\Delta}) \nabla \Delta \bar \omega \right)ds}_{L^p} &\leq \int_1^{t^\beta}s t^{-5/2+1/p}  (\norm{\Delta \bar \omega}_{L^1} +\norm{r\partial_r\Delta \bar \omega}_{L^1} )ds \\
&\leq C \int_1^{t^\beta} s t^{-5/2+1/p} s^{-1} ds \\ & \leq C t^\beta t^{-5/2+1/p}.
\end{align}

Finally, we bound the last term; applying the $r\partial_r$ operator to the definition of $f$ in \eqref{def:f}, we have that
\begin{align}
r\partial_r f = \int_0^{t^{-\beta/\alpha}} (-\alpha) \tilde s^2 (r\partial_r \nabla \Delta \bar \Omega) (\tilde s x) d\tilde{s} \, ,
\end{align}
and since by scaling
$\norm{\tilde s^2 r \partial_r \nabla \Delta \bar \Omega(\tilde s x)}_{L^p} = \tilde s^{2-2/p} \norm{r\partial_r \nabla \Delta \bar \Omega(x)}_{L^p}$,  we have
\begin{align}
\norm{r\partial_r (e^{t\Delta} f)}_{L^p} & \leq \norm{r\partial_r e^{t\Delta}}_{L^p \to L^p} \norm{f}_{L^p} + \norm{e^{t\Delta}}_{L^p \to L^p} \norm{r\partial_r f}_{L^p} \\ & \quad + 2 \norm{e^{t\Delta}}_{L^p \to L^p} \norm{f}_{L^p} \\ & 
\leq C t^{-\beta/\alpha(3-2/p)} \, .
\end{align}
Hence, combining the bounds for the four terms, we conclude \eqref{eq:conv_background}.
\end{proof}

%% file: initiallayer.tex
\section{Initial layer: Proof of Theorem~\ref{thm:initiallayer}}
\label{sec:back_uniq}

In this section, $\tilde u$ is the modified background as in Section~\ref{sec:modifiedbackground} with $\alpha \in (0,1)$.
 We consider the linearized Navier-Stokes problem around $\tilde{u}$:
\begin{equation}
    \label{eq:lin_eq}
    \left\lbrace
    \begin{aligned}
    \partial_t u + \PP( u \cdot \nabla \tilde{u} + \tilde{u} \cdot \nabla u) &= \Delta u \\
    \Div u &= 0 
    \end{aligned}
    \right.
\quad \text{ in } \RR \times \R^+
\end{equation}
with initial condition
\begin{equation}
    u|_{t=0} = u_0 \in H^{2+s}_{\rm df}
\end{equation}
and $s \in (0,1)$ as in Section~\ref{sec:maintheoremrevisited}.

Since $\tilde{u} \in L^2_t H^{2+s}_x$ on finite time intervals,~\eqref{eq:lin_eq} is well-posed in $C([0,T];H^{2+s}_{\rm df})$, as implicit in the energy estimates contained in the proof of Theorem~\ref{thm:initiallayer}, \ref{item:initialayer2} below. Define the flow map
\begin{equation}
    \Phi^T : H^{2+s}_{\rm df} \to H^{2+s}_{\rm df} \, ,
\end{equation}
associating to a given initial datum the solution to the linearized problem \eqref{eq:lin_eq} at a given time $T>0$. Then Theorem~\ref{thm:initiallayer}, \ref{item:initialayer1} 
is equivalent to show that \emph{$\im(\Phi^T)$ is dense in $H^{2+s}_{\rm df}$}.

To tackle~\eqref{eq:lin_eq} in $H^{2+s}_{\rm df}$, we study the quantity
\begin{equation}
\label{eq:defq}
q \coloneqq \langle\nabla\rangle^{2+s} u 
\end{equation}
in $L^2_{\rm df}$. $\langle\nabla\rangle^{2+s}$ is the Fourier multiplier by $(1+|\xi|^2)^{1+s/2}$. Applying~$\langle\nabla\rangle^{2+s}$ to~\eqref{eq:lin_eq}, we derive the equation solved by $q$:
\begin{equation}
\label{eq:eq_q}
\partial_t q - \Delta q + \langle\nabla\rangle^{2+s} \mathbb{P} (\tilde{u} \cdot \nabla \langle\nabla\rangle^{-2-s}q) + \langle\nabla\rangle^{2+s} \mathbb{P} (\langle\nabla\rangle^{-2-s}q \cdot \nabla \tilde{u}) = 0\ ,
\end{equation}
with initial condition
\begin{equation}
    q|_{t=0} = \langle\nabla\rangle^{2+s} u_0 \, .
\end{equation}
By definition, $u = \langle\nabla\rangle^{-2-s} q$. 
Define the flow map $\tilde \Phi^T : L^2_{\rm df} \to L^2_{\rm df}$ of~\eqref{eq:lin_eq} at time $T$ and observe that
\begin{equation}
\Phi^T  = \langle\nabla\rangle^{-2-s}  \circ \tilde \Phi^T \circ  \langle\nabla\rangle^{2+s} \, .    
\end{equation}
Since $ \langle\nabla\rangle^{2+s}: H^{2+s}_{\rm df} \to L^2_{\rm df}$ is a bounded, linear, invertible operator with bounded inverse, the density of $\im(\Phi^T)$ in $H^{2+s}_{\rm df}$ is equivalent to the density of $\im(\tilde \Phi^T)$ in $L^2_{\rm df}$.

The operator $\langle\nabla\rangle^{2+s}$ commutes with the Leray projector, but not with the multiplication by $\tilde{u}$; hence, we introduce a commutator:
    \begin{equation}
    \langle\nabla\rangle^{2+s} (\tilde{u} \cdot)=[\langle\nabla\rangle^{2+s}, \tilde{u} \cdot] + \tilde{u} \cdot \langle\nabla\rangle^{2+s}.
    \end{equation}
and analogously we define $[\langle\nabla\rangle^{2+s}, \nabla \tilde{u} \cdot]$. (Recall that $(u \cdot \nabla) \tilde{u} = (\nabla \tilde{u}) u$.)
We rewrite \eqref{eq:eq_q} as
\begin{equation}
\label{eq:eq_qprimo}
\partial_t q - \Delta q + \mathbb P (\tilde{u} \cdot \nabla q)+\mathbb P A q+ \mathbb P Bq = 0,
\end{equation} 
where 
\begin{align}
\label{eq:def_AB}
Aq & \coloneqq [\langle\nabla\rangle^{2+s}, \tilde{u} \cdot] \nabla \langle\nabla\rangle^{-2-s}q, \\ Bq & \coloneqq  \langle\nabla\rangle^{2+s} (\langle\nabla\rangle^{-2-s}q \cdot \nabla \tilde{u})= [\langle\nabla\rangle^{2+s}, \nabla \tilde{u} \cdot] \langle\nabla\rangle^{-2-s} q + \nabla \tilde{u} \cdot q.
\end{align}

The proof of the density of $\im (\tilde \Phi^T)$ in $L^2_{\rm df}(\RR)$ crucially relies on a backward uniqueness result for the adjoint formulation of \eqref{eq:eq_q},
\begin{equation}
\label{eq:eq_q_adj}
-\partial_t q - \Delta q -\mathbb P( \tilde{u} \cdot \nabla q ) +\mathbb P A^* q+\mathbb P B^*q = 0 \, ,
\end{equation}
with prescribed final datum $q(\cdot,T)$. 
We time-reverse the solution via
\begin{equation}
    w(x,t) = q(x,T-t)
\end{equation}
to obtain the initial-value problem
\begin{equation}
\label{eq:eq_w_adj}
\left\lbrace
\begin{aligned}
    &\partial_t w - \Delta w - \mathbb P( \tilde{u}(\cdot,T-t) \cdot \nabla w ) +\mathbb P A^*|_{T-t} w +\mathbb P B^*|_{T-t} w = 0\\ 
    &w|_{t=0} = w_0 \in L^2_{\rm df} \, .
\end{aligned}
\right.
\end{equation}

\begin{lemma}[Backward uniqueness for the adjoint equation]
\label{lemma:adj_inj}
If the solution $w \in C([0,T];L^2) \cap L^2(0,T;H^1)$ to~\eqref{eq:eq_w_adj} satisfies $w(\cdot,T) = 0$, then $w \equiv 0$ on $\R^2 \times (0,T)$.
\end{lemma}


We can prove Lemma~\ref{lemma:adj_inj} 
using the energy methods derived from the Agmon-Nirenberg theory, as in, for example,~\cite{MR0851146}:

\begin{theorem}[\cite{MR0851146}, Theorem 1.1]
\label{thm:ghidaliatheorem}
Let $H,V$ be Hilbert spaces. Suppose that $V$ is separable and densely embedded into $H$. Let $L : V \to V^*$ be bounded, which begets an unbounded operator $L : D(L) \subset H \to H$ with domain $D(L) = \{ \phi \in V : L\phi \in H \}$. Suppose the coercivity property
\begin{equation}
     \exists \lambda \in \R , \eta > 0 \text{ s.t. } \langle A \phi,\phi \rangle + \lambda \| \Phi \|_H^2 \geq \eta \| \phi \|_V^2 \, .
\end{equation}

Let $0<T<+\infty$. Suppose $\phi \in C([0,T]; V) \cap L^2(0,T; D(L))$ satisfies
\begin{align}
\partial_t \phi + L \phi \in H \, , \quad \text{ a.e. in $t$} \\
\| \partial_t \phi + L \phi \|_H \leq n(t) \|\phi\|_V \, , \quad \text{ a.e. in $t$}  \label{eq:bound_ghidaglia}
\end{align}
with $n \in L^2(0,T)$. If $\phi(T)=0$, then $\phi(t)=0$ for all $0\leq t \leq T$.
\end{theorem}


To apply this theorem, we will require bounds on the operators $A$ and $B$ in $L^2(\RR)$:
\begin{lemma}
\label{lemma:Abound}
Let $\tilde u \in H^{2+s}(\RR)$ and $A, B$ defined in \eqref{eq:def_AB}. Then there exists $C = C(s) >0$ such that
\begin{equation}
    \norm{A}_{L^2 \to L^2} + \norm{B}_{L^2 \to L^2} \leq C\norm{\tilde u}_{H^{2+s}(\RR)} \, .
\end{equation}
Consequently, the adjoint operators $A^*, B^* : L^2 \to L^2$ satisfy the same estimate.
\end{lemma}

Since~$\tilde{u} \in L^\infty(0,T;L^\infty) \cap L^2(0,T; H^{2+s})$ by Remark~\ref{rmk:tildes_Lp} and \Cref{lemma:usmooth}, we have that~\eqref{eq:eq_w_adj} is solvable, as claimed in Lemma~\ref{lemma:adj_inj}.

Lemma~\ref{lemma:Abound} follows from the following Kato-Ponce inequality {in \cite{kato1988commutator}}
:
\begin{align}
 \label{eq:li_estimate}
\norm{\jap^{2+s}(fg) - f \jap^{2+s} g }_{L^2} & \leq C (\norm{\nabla f}_{L^{\infty}} \norm{\jap^{s}\nabla g}_{L^{2}} \\ & \quad + \norm{\jap^{1+s}\nabla f}_{L^{2}} \norm{g}_{L^{\infty}}).
\end{align}

\begin{proof}[Proof of Lemma~\ref{lemma:Abound}]
To bound the operator $A$, we first prove that $ [\langle\nabla\rangle^{2+s}, \tilde{u} \cdot]: H^{1+s}(\RR) \rightarrow L^2(\RR)$ is a bounded operator. Indeed, for any $u \in H^{1+s}(\RR)$ we apply the above lemma choosing in \eqref{eq:li_estimate} $f=\tilde u, g = u$, and the  Gagliardo-Nirenberg inequality, to get
\begin{align}
 \norm{[\langle\nabla\rangle^{2+s}, \tilde{u} \cdot]u}_{L^2(\RR)} & \leq C \norm{\nabla \tilde{u}}_{L^\infty(\RR)} \norm{\jap^{s}\nabla u}_{L^2(\RR)}  \\ & \quad + \norm{\jap^{1+s} \nabla \tilde{u}}_{L^2(\RR)} \norm{u}_{L^\infty(\RR)} \nonumber\\
& \leq C \norm{\tilde{u}}_{H^{2+s}(\RR)} \norm{u}_{H^{1+s}(\RR)}.
\end{align}
Moreover, the operator $ \langle\nabla\rangle^{-2-s} \nabla:L^2(\RR) \to H^{1+s}(\RR)$ is bounded. 
Hence, $A =  [\langle\nabla\rangle^{2+s}, \tilde{u} \cdot] \circ \nabla \langle\nabla\rangle^{-2-s}$ is a bounded operator in $L^2(\RR)$, with norm bounded by $C\|\tilde{u} \|_{H^{2+s}(\RR)}$. 
We can similarly prove that the operator $B: L^2(\RR) \rightarrow L^2(\RR)$ is bounded.
We proceed as done for $A$, and choosing in \eqref{eq:li_estimate} {$f=u, g = \nabla \tilde u$}, we have that 
\begin{align}
\|[\langle\nabla\rangle^{2+s}, \nabla \tilde{u} \cdot] \langle\nabla\rangle^{-2-s} u \|_{L^2(\RR)} & \leq C\|\tilde{u}\|_{H^{2+s}(\RR)}\|\langle\nabla\rangle^{-2-s} {u} \|_{H^{2+s}(\RR)} \\ & \leq C\|\tilde{u}\|_{H^{2+s}(\RR)} \|{u}\|_{L^2(\RR)}.
\end{align}
Therefore, using also that $\| \nabla \tilde{u} \cdot \|_{L^2 \to L^2} \leq  \|\nabla \tilde{u}\|_{L^\infty(\RR)} \leq C\|\tilde{u}\|_{H^{2+s}(\RR)}$, we see that $B:L^2(\RR) \to L^2(\RR)$ is a bounded operator.
\end{proof}

\begin{proof}[Proof of \Cref{lemma:adj_inj}]
We will apply Theorem~\ref{thm:ghidaliatheorem} with $\phi = q$, $V= H^1$, $H = L^2$, and $L = -\Delta$. However, we should demonstrate that the solution $w \in C([0,T];L^2) \cap L^2(0,T;H^1)$ additionally satisfies $w \in C([\varepsilon,T];H^1) \cap L^2(\varepsilon,T;H^2)$ for every $\varepsilon \in (0,T)$. 
This follows from smoothing for $e^{t\Delta} w_0$ and maximal regularity, since
\begin{equation}
\begin{aligned}
    \| (\partial_t - \Delta) w\|_{L^2} & =  \|\tilde{u} \cdot \nabla w + A^* w - B^* w\|_{L^2} \\ & \leq C \| \tilde{u} \|_{L^\infty} \| \nabla w \|_{L^2} + C \|\tilde{u} \|_{H^{2+s}} \| w \|_{L^2} \, , \quad \text{ a.e. } t \in (0,T) \, ,
    \end{aligned}
\end{equation}
and the right-hand side belongs to $L^2(0,T)$, thanks to \Cref{lemma:Abound}. This also verifies~\eqref{eq:bound_ghidaglia} with $n(t)= C \| \tilde{u}(\cdot,t) \|_{L^\infty} + C\|\tilde{u}(\cdot,t) \|_{H^{2+s}(\RR)} \in L^2(0,T)$. 



\end{proof}

\begin{proof}[Proof of Theorem~\ref{thm:initiallayer}, \ref{item:initialayer1} (Linear density)]
As observed at the beginning of the section, it is enough to prove that the flow map $\tilde \Phi^T : L^2_{\rm df} \to L^2_{\rm df}$ associated to \eqref{eq:eq_q} has dense image.

In~\Cref{lemma:adj_inj}, we proved that the adjoint flow map associated to~\eqref{eq:eq_q} is injective, that is, $\ker (\tilde \Phi^T)^*  = \{0\}$. By the relationship between kernel and image of a bounded operator on a Hilbert space (see Corollary 2.18 of \cite{MR2759829}), we have 
\begin{equation}
\overline{\im(\tilde \Phi^T)} = (\ker(\tilde \Phi^T)^*)^\perp = L^2_{\rm df} \, .
\end{equation}
\end{proof}

The proof of Theorem~\ref{thm:initiallayer} \ref{item:initialayer2} follows from standard energy estimates and a Gronwall argument.

\begin{proof}[Proof of Theorem~\ref{thm:initiallayer}, \ref{item:initialayer2} (Nonlinear density)]

The short-time existence and uniqueness of a solution to the nonlinear problem \eqref{eq:initialayer2} in $C([0,T];H^{2+s}_{\rm df}(\RR))$ is standard and can be demonstrated by rewriting the equation according to Duhamel's formula and applying a contraction mapping argument. Therefore, we focus only on the quantitative estimates, which can be proved via energy estimates. We define $\psi \coloneqq u - \varepsilon v$; from \eqref{eq:initialayer1} and \eqref{eq:initialayer2}, we get that $\psi$ solves:
\begin{equation}
\label{eq:psi_system}
\begin{cases}
\begin{aligned}
\partial_t \psi &+ \mathbb{P} (\tilde{u} \cdot \nabla \psi + \psi \cdot \nabla \tilde{u} + \varepsilon^2 v \cdot \nabla v \\
&\quad + \psi \cdot \nabla \psi + \varepsilon v \cdot \nabla \psi + \varepsilon \psi \cdot \nabla v) = \Delta \psi,
\end{aligned}\\
\Div \psi = 0,\\
\psi|_{t=0} = {\psi_0}.
\end{cases}
\end{equation}

Applying the operator $\jap^{2+s}$ to the first equation of \eqref{eq:psi_system}, we get:
\begin{align}
    \p_t & \jap^{2+s} \psi -\Delta \jap^{2+s} \psi + \mathbb{P} ((\tilde{u} + \varepsilon v) \cdot \nabla \jap^{2+s} \psi) + \mathbb{P} (\jap^{2+s} \psi \cdot (\nabla \tilde{u} + \varepsilon \nabla v)) \\ & + \mathbb{P}([\jap^{2+s}, (\tilde{u} + \varepsilon v)\cdot] \nabla \psi) +\mathbb{P}([\jap^{2+s}, (\nabla \tilde{u} + \varepsilon \nabla v)\cdot] \psi)  \\ & + \jap^{2+s} \mathbb P (\psi \cdot \nabla \psi) + \varepsilon^2 \jap^{2+s} \mathbb P(v \cdot \nabla v) = 0,
\end{align}
with {$\jap^{2+s} \psi|_{t=0}=\jap^{2+s} \psi_0$}. We proceed by an energy estimate, multiplying the equation by $\jap^{2+s} \psi$ and integrating in space over $\RR$. First of all, since $\mathbb{P} = I - \nabla \Delta^{-1} (\nabla \cdot )$ and $\tilde{u}, v$ and $\psi$ are divergence-free, we have that
\begin{equation}
\int_{\RR} \mathbb{P} ((\tilde{u} + \varepsilon v) \cdot \nabla \jap^{2+s} \psi) \jap^{2+s} \psi = 0.
\end{equation}
Next, we estimate the remaining terms neglecting the Leray Projector, since it is a bounded operator in $L^2(\RR)$, and therefore will preserve the estimates.
By Gagliardo-Nirenberg, we can bound:
\begin{align}
    \| \jap^{2+s} \psi \cdot (\nabla \tilde{u} + \varepsilon \nabla v) \|_{L^2(\RR)} & \leq C \|\psi\|_{H^{2+s}(\RR)} \|\nabla (\tilde{u} + \varepsilon v)\|_{L^\infty(\RR)}, \\
    & \leq C \|\psi\|_{H^{2+s}(\RR)} \|\tilde{u} + \varepsilon v\|_{H^{2+s}(\RR)}. \label{eq:inlayer2_first}
\end{align}
Moreover, from \Cref{lemma:Abound} we have:
\begin{align}
    \|[\jap^{2+s}, (\tilde{u} + \varepsilon v) \cdot] \nabla \psi \|_{L^2(\RR)} & + \| [\jap^{2+s}, \nabla (\tilde{u} + \varepsilon v)\cdot] \psi \|_{L^2(\RR)}\\ & \quad \leq C \|\tilde{u} + \varepsilon v\|_{H^{2+s}(\RR)} \|\psi\|_{H^{2+s}(\RR)} \label{eq:inlayer2_second},
\end{align}
for some $C>0$. 
For the remaining terms, we have that
\begin{align}
         \intR \jap^{2+s} & \Div(\psi \otimes \psi + \varepsilon^2 v \otimes v) \jap^{2+s} \psi \\ & \quad = 
        - \intR \jap^{2+s} (\psi \otimes \psi+ \varepsilon^2 v \otimes v) : \nabla \jap^{2+s} \psi \\ & \quad \leq \| \jap^{2+s} (\psi \otimes \psi + \varepsilon^2 (v \otimes v))\|_{L^2(\RR)} \|\nabla \jap^{2+s} \psi\|_{L^2(\RR)} \\ & \quad \leq C (\|\psi\|^2_{H^{2+s}(\RR)} + \varepsilon^2\|v\|_{H^{2+s}(\RR)}) \|\nabla \jap^{2+s} \psi\|_{L^2(\RR)}\\ & \quad \leq \frac{1}{2} \|\nabla \jap^{2+s} \psi\|^2_{L^2(\RR)} +  \frac{C}{2} (\|\psi\|^4_{H^{2+s}(\RR)} + \varepsilon^4 \|v\|^4_{H^{2+s}(\RR)}). \label{eq:inlayer2_fourth}
\end{align}
We now sum \eqref{eq:inlayer2_first}, \eqref{eq:inlayer2_second} and \eqref{eq:inlayer2_fourth} to conclude the energy estimate; recalling that $\tilde{u} \in L^2(0,T; H^{2+s}(\R^2))$ by \Cref{lemma:usmooth} and $v\in C([0,T]; H^{2+s}(\RR))$ 
and defining $X(t) \coloneqq \|\jap^{2+s} \psi(\cdot, t)\|^2_{L^{2}(\RR)}$, the following inequality holds for any $t \in [0,T]$: 
\begin{equation}
\label{eq:X_inequality}
    \partial_t X(t) \leq C (\|\tilde u + \varepsilon v\|_{H^{2+s}(\RR)} X(t) +  X(t)^2 + \varepsilon^4 \|v\|_{H^{2+s}(\RR)}^4),
    \end{equation}
with ${X(0) = \|\jap^{2+s}\psi_0\|^2_{L^{2}(\RR)}}$. Since {$X(0)$ is small} and $\psi$ is continuous as a function on $[0,T]$ with values in $L^2$,
we call  $T_{\rm max}$ the first time in $t\in [0,T]$ in which $X(t) =1$ and observe that $T_{\rm max}>0$ (or $T_{\rm max}= T$ if such $t$ does not exist). 
Now, since also $\|\tilde u + \varepsilon v\|_{H^{2+s}(\RR)} \leq \|\tilde u\|_{H^{2+s}(\RR)} +\|v\|_{H^{2+s}(\RR)} \leq \|\tilde u\|_{H^{2+s}(\RR)}^2 +C$, where $C$ depends in particular on $v$, 
\eqref{eq:X_inequality} in $(0,T_{\rm max})$ implies that
\begin{equation}
    \partial_t X(t) \leq C (\|\tilde u \|^2_{H^{2+s}(\RR)} + 1) X(t) + C \varepsilon^4 ,
\end{equation}
with ${X(0) = \|\jap^{2+s}\psi_0\|^2_{L^{2}(\RR)}}$ and by Gronwall's lemma we have
\begin{align}
\label{eq:gronwall1}
X(t) & \leq C (\varepsilon^4 +{\|\jap^{2+s}\psi_0\|^2_{L^{2}(\RR)}})  e^{C \int_0^t(1 + \|\tilde u \|^2_{H^{2+s}(\RR)}) }
\\ & \leq 
C (\varepsilon^4 + {\|\jap^{2+s}\psi_0\|^2_{L^{2}(\RR)}})
, \label{eq:gronwall2}
\end{align}
for any $t \in [0,T_{\rm max}]$. 
Finally, choosing any $ \varepsilon, {\|\jap^{2+s}\psi_0\|^2_{L^{2}(\RR)}}\leq C^{-1}
$, where $C$ is the constant appearing in \eqref{eq:gronwall2} we obtain that $T_{\rm max}=T$ and that \eqref{eq:gronwall2} holds in $[0,T]$.

\end{proof}

%% file: nonlinear.tex
\section{Nonlinear argument: Proof of Theorem~\ref{thm:longtimebehavior}}
\label{sec:nonlinear}

We now prove Theorem~\ref{thm:longtimebehavior}, namely, that the solution to the forced Navier-Stokes with 
initial datum $\Phi \in H^{2+s}$ containing a substantial component in the unstable directions will exhibit nonlinear instability.


We consider solutions $\tilde{U} + U$ in similarity variables~\eqref{eq:urescalingnotation}. The equation for $U$ is
\begin{equation}
\label{eq:U}
    \begin{cases*}
    \begin{aligned}
    \partial_\tau U + &\left(-1 + \frac{1}{\alpha}\right) U - \frac{1}{\alpha} \xi \cdot \nabla U \\ &+ \mathbb{P}(\tilde{U} \cdot \nabla U + U \cdot \nabla \tilde{U} + U \cdot \nabla U) = e^{-\tau \gamma} \Delta U 
    \end{aligned}\\
    \Div U = 0 \\
    U|_{\tau = \tau_0} = \Phi \, ,
    \end{cases*}
\end{equation}
with the corresponding vorticity equation
\begin{equation}
\label{eq:Omega}
    \begin{cases*}
    \partial_\tau \Omega -\Omega - \frac{1}{\alpha} \xi \cdot \nabla \Omega +\tilde{U} \cdot \nabla \Omega + U \cdot \nabla \tilde{\Omega} + U \cdot \nabla \Omega = e^{-\tau \gamma} \Delta \Omega \\
    \Omega|_{\tau = \tau_0} = \nabla^\perp \cdot \Phi \, .
    \end{cases*}
\end{equation}

\subsection{Functional setup} 
We seek $U$ in a suitable space equipped with the norm
\begin{equation}
\label{eq:normX}
\norm{U}_X \coloneqq \norm{U}_{L^2} + \norm{\Omega}_{L^2} + \norm{\nabla \Omega}_{L^2 \cap L^{2+\sigma}},
\end{equation}
where $\sigma>0$ is given in \Cref{rmk:sigmacond} only in terms of $\alpha$. The $X$ norm is equivalent to the $W^{2,2+\sigma}$ norm, but~\eqref{eq:normX} indicates which energy estimates we will perform. By the Sobolev embedding $W^{2+s, 2}(\RR) \subseteq  W^{2,2+\sigma}(\RR)$ for the aforementioned choice of $s = 1-2/(2+\sigma)$, the $H^{2+s}$ norm controls the $X$ norm.
%
\begin{remark}
The $L^p$ norms of the velocity for any $p \in [2,+\infty]$ are controlled via the Gagliardo-Nirenberg inequality by the $X$ norm. For instance, for $p=+\infty$, there exists $ C>0$ such that
\begin{equation}
\label{bound:uLinf}
    \| U \|_{L^\infty} \leq C \| U \|_{L^2}^{1/2} \| \nabla \Omega \|_{L^2}^{1/2} \leq C \|U\|_X.
\end{equation}
We can also control $\|\Omega\|_{L^\infty}$ with the $X$ norm:
\begin{equation}
\label{eq:omega_Linf}
    \norm{\Omega}_{L^\infty} \leq C \norm{\nabla \Omega}_{L^{2+\sigma}}^{\sigma /(4+3\sigma)} \norm{\nabla U}_{L^2}^{(4+2\sigma)/(4+3\sigma)}\leq C \|U\|_X.
\end{equation}
\end{remark}
\begin{remark}
For 2D $m_0$-fold symmetric functions, the following Hardy-type inequality holds:
\begin{equation}
\label{eq:hardy}
    |U(x)| \leq C \min(|x|,1) \| U \|_X \, , \quad \forall x \in \R^2.
\end{equation}
Indeed, for any $R>0$ the average $U_{\rm avg}$ of $U$ on the circle $\p B_R$ is zero by symmetry, so 
\begin{equation}
    R^{-1} \| U - U_{\rm avg}
     \|_{L^\infty(\p B_R)} \leq C R^{-1}  \| \p_\theta U \|_{L^\infty(\p B_R)} \leq C \| U \|_{{\rm Lip}}.
\end{equation}
\end{remark}
\subsection{Perturbative argument}
We solve \eqref{eq:U} and \eqref{eq:Omega} for $U\coloneqq \Ulin + \Uper$ and $\Omega \coloneqq \Omegalin + \Omegaper$, where $\Ulin$ and $\Omegalin$ solve the linearized Euler equation around $\bar U$ and $\bar \Omega$ with data $P_{\lambda,\bar{\lambda}} \Phi$ and $\nabla^\perp \cdot P_{\lambda,\bar{\lambda}} \Phi$, namely 
\begin{equation}
\label{eq:Ulin}
\begin{cases}
\partial_\tau \Ulin - L_{\rm ss}^u \Ulin = 0 \\
\Div \Ulin = 0 \\
\Ulin|_{\tau = \tau_0} = P_{\lambda,\bar{\lambda}} \Phi,
\end{cases}
\qquad
    \begin{cases*}
    \partial_\tau \Omegalin - L^\omega_{\rm ss} \Omegalin = 0 \\
    \Omegalin|_{\tau = \tau_0} = \nabla^\perp \cdot P_{\lambda,\bar{\lambda}} \Phi,
    \end{cases*}
\end{equation}
with $L_{\rm ss}^u = L_{\rm ss}$ defined in~\eqref{eq:lssudef} and
\begin{equation}
- L^\omega_{\rm ss} \Omegalin = \left(-1- \frac{\xi}{\alpha} \cdot \nabla\right) \Omegalin + \bar{U} \cdot \nabla \Omegalin + \Ulin \cdot \nabla \bar{\Omega} \, .
\end{equation}
From \eqref{eq:Ulin}, it is clear that $\Ulin = e^{(\tau-\tau_0) L^u_{\rm ss}} (P_{\lambda,\bar{\lambda}} \Phi)$ and $\Omegalin = e^{(\tau-\tau_0) L^\omega_{\rm ss}}  (\nabla^\perp \cdot P_{\lambda,\bar{\lambda}} \Phi)$. 
By \Cref{rmk:petabar}, we have the estimate 
\begin{equation}
\label{eq:bound_Ulin}
\| \Omegalin\|_{W^{3,2+\sigma}(\RR)} \leq C\varepsilon e^{a\tau}, \quad \forall \tau \in \mathbb R.
\end{equation}
The perturbations $\Uper$ and $\Omegaper$ solve, respectively
\begin{align}
\label{eq:Uper}
\begin{split}
\partial_\tau \Uper & -L^u_{\rm ss} \Uper = - \mathbb{P}((\tilde{U} - \bar{U}) \cdot \nabla \Uper) - \mathbb{P}((\tilde{U} - \bar{U}) \cdot \nabla \Ulin) \\ 
& \quad -\mathbb{P} (\Uper \cdot \nabla (\tilde{U} - \bar{U}) + \Uper \cdot \nabla \Uper) - \mathbb{P}(\Uper \cdot \nabla \Ulin + \Ulin \cdot \nabla \Uper) \\ 
& \quad -\mathbb{P}(\Ulin \cdot \nabla (\tilde{U}-\bar U) + \Ulin \cdot \nabla \Ulin) + e^{-\tau \gamma} \Delta \Ulin + e^{-\tau \gamma} \Delta \Uper,
\end{split}
\end{align}
with {$\Uper|_{\tau = \tau_0} = \Phi - P_{\lambda,\bar{\lambda}} \Phi$} and 
\begin{align}
\label{eq:Omegaper}
\begin{split}
\partial_\tau \Omegaper & - L_{\rm ss}^\omega\Omegaper = - (\tilde{U} - \bar{U}) \cdot \nabla \Omegaper - (\tilde{U} - \bar{U}) \cdot \nabla \Omegalin \\ & \quad - \Uper \cdot \nabla (\tilde{\Omega} - \bar{\Omega}) - \Ulin \cdot \nabla (\tilde{\Omega} - \bar{\Omega}) - \Uper \cdot \nabla \Omegaper \\ & \quad - \Uper \cdot \nabla \Omegalin - \Ulin \cdot \nabla \Omegaper - \Ulin \cdot \nabla \Omegalin + e^{-\tau \gamma} \Delta \Omegalin + e^{-\tau \gamma} \Delta \Omegaper,
\end{split}
\end{align}
with {$\Omegaper|_{\tau = \tau_0} = \nabla^\perp \cdot (\Phi - P_{\lambda,\bar{\lambda}}\Phi)$}.

The proof of \Cref{thm:longtimebehavior} consists in showing that the bound \eqref{eq:nonlinear_bound} holds; this is rewritten in terms of $U^{\rm per}$ in the following proposition.

\begin{proposition}[Nonlinear estimate]
\label{prop:nonlinear_estimate}Let $\bar U, \bar F, \lambda, \eta, P_{\lambda}, P_{\bar\lambda}$ be as in the statement of Theorem~\ref{thm:longtimebehavior}.
Then, {for every $0 <c<1 $}, there exists $C:= C(data,c)$ (depending on $\bar U, \bar F, \lambda, \eta$...) and $ \tau_0:= \tau_0(data,c )$ 
such that the following holds.


{For any $0 <c<1 $}, there exist $\tau_0 \gg 1$ and $C\geq 1$ depending on $c,\bar U, \eta, a$ such that the following holds. Let {$\Phi \in W^{2,2+\sigma}$}, {$\varepsilon \coloneqq e^{-a\tau_0}\|P_{\lambda,\bar{\lambda}}\Phi\|_X$, $M \coloneqq e^{-a\tau_0}\|\Phi - P_{\lambda,\bar{\lambda}} \Phi\|_X$}, and let $U^{\rm lin}$ solve the linearized Euler equation \eqref{eq:Ulin} around $\bar U$ with initial datum {$P_{\lambda,\bar{\lambda}} \Phi$} . If
\begin{equation}
\label{eq:tau_max_condition}
\tau_{\rm max} := \frac{1}{a} \log \frac{1}{C(\varepsilon + M)} \geq \tau_0 \, ,
\end{equation}
then the solution $\Uper$ to \eqref{eq:Uper} with initial condition $\Phi - P_{\lambda,\bar{\lambda}} \Phi$ satisfies on $[\tau_0,\tau_{\rm max}]$
\begin{equation}
	\label{eq:main_nonlin_estimate}
\|\Uper\|_X \leq c \varepsilon e^{a\tau} + CM e^{a \tau} \, .
\end{equation}
\end{proposition}
We show \eqref{eq:main_nonlin_estimate} with a bootstrap argument.
\begin{proposition}[Bootstrap estimate]
\label{prop:bootstrap}
In the same notation as \Cref{prop:nonlinear_estimate}, let $K, \bar{C} \geq 1$ and $\tau_{\rm max} \geq \tau_0 \gg 1$ such that
\begin{equation}
\label{eq:ass_K}
e^{-\zeta\tau_0} \leq 1/K, \quad 
(\varepsilon + M) e^{a \tau_{\rm max}} \leq 1/K,
\end{equation} 
where $\zeta > 0$ is the constant depending only on $\alpha$ from \Cref{rmk:sigmacond}.
Assume that the solution $\Uper$ to \eqref{eq:Uper} satisfies \begin{equation}
	\label{eq:bootstrapassumption}
\|\Uper\|_X \leq \bar{C} e^{a\tau} \left(\frac{\varepsilon}{K} + M\right)\qquad \mbox{for every }\tau \in [\tau_0,\bar{\tau}]
\end{equation}
for some $\tau_0 < \bar{\tau} \leq \tau_{\rm max}$. Then, there exists $ C = C(\bar U, \eta, a)$ such that:
\begin{equation}
\label{eq:finalimprovedassumption}
\|\Uper\|_X \leq C \bar{C}^{1/2} (1+\bar{C}^2/K) e^{a\tau} (\varepsilon/K + M) \qquad \mbox{for every }\tau \in [\tau_0,\bar{\tau}].
\end{equation}

\end{proposition}

\begin{proof}[Proof of \Cref{prop:bootstrap}]
We establish a preliminary improved estimate for each term in the norm $\| \Uper \|_X$:
\medskip

\textbf{Step 1: baseline estimate for $\norm{\Uper}_{L^2}$}

\textit{
    Under the bootstrap assumption \eqref{eq:bootstrapassumption}, it holds:
    \begin{equation}
    \label{eq:improved}
    \norm{\Uper}_{L^2} \leq  C(1+\bar{C}^2/K) e^{a\tau} (\varepsilon/K + M),
    \end{equation}
    for all $\tau \in [\tau_0, \tau_{\rm max}]$.}

We denote by $F$ the right-hand side of \eqref{eq:Uper}; let us now analyze separately the terms in $F$ to which we will apply the semigroup $e^{\tau L_{\rm ss}}$. Indeed, we will get the desired estimate for the solution to \eqref{eq:Uper} via Duhamel's formula: 
\begin{equation}
    U^{\rm per}(\cdot, \tau) = e^{(\tau-\tau_0) L^u_{\rm ss}} \Psi + \int_{\tau_0}^\tau e^{(\tau-s)L^u_{\rm ss}} F(s) ds\, .
\end{equation}

The estimate on $L_{\rm ss}: L^2 \rightarrow L^2$ obtained in \eqref{eq:semigroup_estimate} yields
\begin{align}
\label{eq:Uper_gronwall}
    \| \Uper(\cdot, \tau)\|_{L^2(\RR)} & \leq \| e^{\tau L^u_{\rm ss}} \Psi \|_{L^2(\RR)} +\int_{\tau_0}^\tau \|e^{(\tau-s)L^u_{\rm ss}} F(s)\|_{L^2(\RR)} ds \\ & \leq C M e^{a\tau} + C \int_{\tau_0}^\tau e^{(\tau-s)a} \|F(s)\|_{L^2(\RR)} ds\, .
\end{align}

We are now left to bound $\norm{F}_{L^2}$. First of all, since the Leray projector is a bounded operator in $L^2$, it is enough to bound the $L^2$ norm of the object to which it is applied.

We start by bounding the terms related to the fact that $\Ulin$ solves a linear equation, which is linearized around the background $\bar{U}$ and not $\tilde{U}$ and does not contain a viscous term. We use the quantitative results on the convergence to the background from \Cref{lemma:conv_background} and in particular \Cref{rmk:sigmacond} and the bound on the linear part in \eqref{eq:bound_Ulin}: we conclude that there exists $C = C(a,\eta, \tilde{U})$ such that
\begin{align}
    & \norm{\mathbb{P}(\Ulin \cdot \nabla (\bar{U} - \tilde{U}) + \nabla \Ulin \cdot (\bar{U} - \tilde{U}))}_{L^2}  + \norm{e^{-\tau \gamma} \Delta \Ulin}_{L^2} \\ & \quad \leq \norm{\Ulin}_{L^\infty} \norm{\nabla (\bar{U} - \tilde{U})}_{L^2} + \norm{\nabla \Ulin}_{L^2} \norm{\bar{U} - \tilde{U}}_{L^\infty} + \norm{e^{-\tau \gamma} \Delta \Ulin}_{L^2} \\ & \quad \leq  C \varepsilon e^{(a-\zeta)\tau},
\end{align}
\begin{align}
    \norm{e^{-\tau \gamma} \Delta \Ulin}_{L^2} & \leq C \varepsilon e^{(a-\gamma)\tau} \leq C \varepsilon e^{(a-\zeta)\tau},\\
    \norm{\mathbb{P}(\Ulin \cdot \nabla \Ulin)}_{L^2} & \leq \norm{\nabla \Ulin}_{L^2} \norm{\Ulin}_{L^\infty} \leq C \varepsilon^2 e^{2a\tau}.
\end{align}

Moreover, the diffusion term with $\Uper$ gains smallness with respect to the bootstrap assumption \eqref{eq:bootstrapassumption} thanks to the prefactor $e^{-\tau \gamma}$:
\begin{equation}
    \| e^{-\tau \gamma} \Delta \Uper \|_{L^2} \leq e^{-\tau \gamma} \|\nabla \Omegaper\|_{L^2} \leq e^{-\tau \zeta}  \bar{C} e^{a\tau} (\varepsilon/K + M).
\end{equation}

The remaining terms involving $U^{\rm per}$ and its derivatives are lower order terms, related to discrepancy between the backgrounds $\tilde{U}$ and $\bar{U}$, convection terms interacting with $\Ulin$ and nonlinearity. To bound them, we rely on the bootstrap assumption \eqref{eq:bootstrapassumption} and on the estimates in \Cref{lemma:conv_background}, \Cref{rmk:sigmacond} and \eqref{eq:bound_Ulin}, to conclude that:
\begin{align}
    & \norm{\mathbb{P}((\tilde{U} - \bar{U}) \cdot \nabla \Uper) + \Uper \cdot \nabla(\tilde{U} - \bar{U}))}_{L^2} \\ & \quad \leq \norm{\tilde{U} - \bar{U}}_{L^\infty} \norm{\nabla \Uper}_{L^2} + \norm{\Uper}_{L^\infty} \norm{\nabla(\tilde{U} - \bar{U})}_{L^2} \leq C \bar{C} e^{(a-\zeta)\tau} (\varepsilon/K + M),
\end{align}
\begin{align}
    \norm{\mathbb{P}(\Ulin\cdot \nabla \Uper + \Uper \cdot \nabla \Ulin)}_{L^2} & \leq  \norm{\Ulin}_{L^\infty} \norm{\nabla \Uper}_{L^2} + \norm{\Uper}_{L^2} \norm{\nabla \Ulin}_{L^\infty}\\ & \leq C \bar{C} \varepsilon e^{2a\tau} (\varepsilon/K + M), \\
    \norm{\mathbb{P}(\Uper \cdot \nabla \Uper)}_{L^2} & \leq  \norm{\Uper}_{L^\infty} \norm{\nabla \Uper}_{L^2} \leq C \bar{C}^2 e^{2a\tau} (\varepsilon/K + M)^2.
\end{align}
Therefore, integrating in time, we get 
\begin{align}
      \int_{\tau_0}^{\tau} e^{a(\tau-s)}\norm{F(s)}_{L^2} ds & \leq C \int_{\tau_0}^{\tau} e^{a(\tau-s)} ((\varepsilon  +  \bar C (\varepsilon/K + M))  e^{(a-\zeta)s} \\ & \quad\quad\quad + (\varepsilon^2 + \bar C^2 (\varepsilon/K + M)^2 e^{2as}) ds \label{eq:b1}\\ 
     & \leq C \varepsilon e^{a\tau} e^{-\zeta \tau_0} + C \varepsilon^2 e^{2a\tau} \\ & \quad + C \bar{C} e^{-\zeta \tau_0} e^{a\tau} (\varepsilon/K + M) + C \bar{C}^2 e^{2a\tau} (\varepsilon/K + M )^2.
    \end{align}

Using \eqref{eq:Uper_gronwall} and the preliminary assumptions \eqref{eq:ass_K} on $K$ and $\tau_{\rm max}$, we reduce to:
\begin{align}
   \|\Uper\|_{L^2} & \leq Ce^{a\tau}(\varepsilon/K+M) +  C\bar{C} e^{a\tau}(\varepsilon/K + M)/K \\ & \quad + C\bar{C}^2  e^{a\tau}(\varepsilon/K + M)/K \\
   & \leq C(1+\bar{C}^2/K) e^{a\tau}  (\varepsilon/K + M). \label{eq:improved_ass_U} 
\end{align}


\textbf{Step 2: estimate for $\|\Omegaper\|_{L^2}$}

\textit{Under the bootstrap assumption \eqref{eq:bootstrapassumption}, it holds:
    \begin{equation}
            \| \Omega^{\rm per} \|_{L^2} 
            \leq C(1+\bar{C}^2/K) e^{a\tau} (\varepsilon/K + M),     
    \end{equation}
for all $\tau \in [\tau_0,\tau_{\rm max}]$.}

We provide an estimate for $\|\Omegaper\|_{L^2}$ by energy methods, multiplying \eqref{eq:Omegaper} by $\Omegaper$ and integrating by parts on $\RR$. The term representing diffusion for $\Omegaper$ gives us: 
\begin{equation}
    \intR e^{-\tau \gamma} \Delta \Omegaper \Omegaper =  - e^{-\tau \gamma} \norm{\nabla \Omegaper}_{L^2}^2 \leq 0.
\end{equation}
Moreover, since the velocity fields are divergence free, the contribution from the advection terms is zero:
\begin{align}
    \intR ((\Ulin +\Uper+ \tilde{U}) \cdot \nabla) \Omegaper \Omegaper = 0.
\end{align}
Using the improved assumption \eqref{eq:improved_ass_U} on $\|\Uper\|_{L^2}$ from \textbf{Step 1}, we get the bound
\begin{equation}
    \norm{(\Uper \cdot \nabla) \bar{\Omega}}_{L^2} \leq \norm{\Uper}_{L^2} \norm{\nabla\bar{\Omega}}_{L^\infty} \leq C(1+\bar{C}^2/K) e^{a\tau} (\varepsilon/K + M).
\end{equation}

The remaining terms can be bounded thanks to the convergence of the modified background to Vishik's background (\Cref{lemma:conv_background}) and the bound on the linear part \eqref{eq:bound_Ulin}:
\begin{align}
     \norm{(\Uper + \Ulin) \cdot \nabla \Omegalin}_{L^2} +  \norm{e^{-\tau \gamma} \Delta \Omegalin}_{L^2} & \leq (\norm{\Uper}_{L^2} + \norm{\Ulin}_{L^2}) \norm{\nabla \Omegalin}_{L^\infty} \\ & \quad +  \norm{e^{-\tau \gamma} \Delta \Omegalin}_{L^2} \\
     & \leq C \varepsilon e^{a\tau} \bar{C} e^{a\tau} (\varepsilon/K + M) \\ & \quad + C \varepsilon^2 e^{2a\tau} + C e^{-\zeta \tau_0}\varepsilon e^{a\tau}\\ & \leq C \bar{C} (\varepsilon/K + M)/K + C \varepsilon e^{a\tau}/K, \\
    \norm{(\Ulin \cdot \nabla) (\tilde{\Omega}-\bar{\Omega}) + ((\tilde{U}-\bar{U}) \cdot \nabla) \Omegalin }_{L^2} & \leq \norm{\Ulin}_{L^\infty} \norm{\nabla(\tilde{\Omega}-\bar{\Omega})}_{L^2} \\ & \quad + \norm{\tilde{U}-\bar{U}}_{L^\infty} \norm{\nabla \Omegalin}_{L^2} \\
    & \leq C e^{-\zeta \tau_0} \varepsilon e^{a\tau} \leq C \varepsilon e^{a\tau}/K , \\
    \norm{(\Uper \cdot \nabla) (\tilde{\Omega}-\bar{\Omega})}_{L^2} & \leq \norm{\Uper}_{L^\infty} \norm{\nabla(\tilde{\Omega}-\bar{\Omega})}_{L^2} 
    \\ & \leq C e^{-\zeta \tau_0} \bar{C} e^{a\tau} (\varepsilon/K + M). 
\end{align}
In the end, we obtain:
\begin{align}
    \frac{1}{2} \partial_\tau \|\Omegaper\|_{L^2} + \left(\frac{1}{\alpha} - 1 \right) \|\Omegaper\|_{L^2} & \leq C( 1+ \bar{C}^2/K) e^{a\tau} (\varepsilon /K + M).
\end{align}
Therefore, we can conclude by Gronwall's Lemma that
\begin{equation}
\label{eq:impr_omega_L2}
    \|\Omegaper\|_{L^2} \leq  C( 1+ \bar{C}^2/K) e^{a\tau} (\varepsilon /K + M).
\end{equation}
From \eqref{eq:impr_omega_L2} we can obtain an improved estimate in $L^{2+\sigma}$ both for $\Uper$ and $\Omegaper$ via the Gagliardo-Nirenberg inequality for $\theta = (4+\sigma)/(4+2\sigma) \geq 1/2$:

\begin{align}
    \|\Uper\|_{L^{2+\sigma}} & \leq C \| \Uper \|^{\theta}_{L^2} \| \Omegaper\|^{1-\theta}_{L^2} \\ & \leq C( 1+ \bar{C}^2/K) e^{a\tau} (\varepsilon /K + M), \label{eq:improvedU2+} \\
    \|\Omegaper\|_{L^{2+\sigma}} & \leq C \| \Omegaper \|^{\theta}_{L^2} \| \nabla \Omegaper\|^{1-\theta}_{L^2} \\ & \leq  C \| \Omegaper \|^{1/2}_{L^2} \| \nabla \Omegaper\|^{1/2}_{X} \\ 
    & \leq  C \bar{C}^{1/2} (1+\bar{C}^2/K)^{1/2} e^{a\tau} (\varepsilon/K + M). \label{eq:impr_omega_L2+}
\end{align}

\textbf{Step 3: estimates for $wr^{-1} \partial_\theta \Omegaper$ in $L^2$ and $L^{2+\sigma}$}

\textit{
    Under the bootstrap assumption \eqref{eq:bootstrapassumption}, it holds:
    \begin{equation}
            \| wr^{-1}\partial_\theta\Omega^{\rm per} \|_{L^2 \cap L^{2+\sigma}} \leq C \bar{C}^{1/2}(1+\bar{C}^2/K) e^{a\tau} (\varepsilon/K + M),
    \end{equation}
for all $\tau \in [\tau_0,\tau_{\rm max}]$.}

We start by focusing on the angular derivative, as we will build the estimates for $\nabla \Omegaper$ on top of those. In particular, as we need control on the angular derivative $\partial_\theta$ only around the origin, it is convenient to work with a weighted quantity that behaves like $\partial_\theta$ close to the origin and like the angular component of the gradient in polar coordinates, that is, $r^{-1} \partial_\theta$, far from the origin. To this end, let
\begin{equation}
    w(r) = r \mathbf{1}_{[0,1]} + \mathbf{1}_{(1,+\infty)} \, ,
\end{equation}
consider $wr^{-1} \partial_\theta \Omegaper$
and apply $wr^{-1} \partial_\theta$ to \eqref{eq:Omegaper}, recalling that $(\tilde{U},\tilde{\Omega})$ and $(\bar{U}, \bar{\Omega})$ are radial functions:
\begin{align}
& \partial_\tau \left(\frac{w}{r} \partial_\theta \Omegaper\right) - \frac{w}{r} \partial_\theta \Omegaper - \frac{1}{\alpha} \frac{w}{r} \partial_\theta \left(\xi \cdot \nabla \Omegaper\right) = - (\tilde U + \Ulin + \Uper) \cdot \nabla \frac{w}{r} \partial_\theta \Omegaper  \\ & \quad  - \frac{w}{r} \partial_\theta \Uper \cdot \nabla (\bar{\Omega} + (\tilde \Omega - \bar \Omega)) - \frac{w}{r} \partial_\theta \Ulin \cdot \nabla (\tilde{\Omega} - \bar{\Omega}) \\ & \quad - \frac{w}{r} (\partial_\theta \Uper + \partial_\theta \Ulin) \cdot \nabla \Omegaper  - \frac{w}{r} (\partial_\theta \Uper + \partial_\theta \Ulin) \cdot \nabla \Omegalin  \\ & \quad - (\Uper + \Ulin) \cdot \nabla \left(\frac{w}{r} \partial_\theta \Omegalin \right) +\frac{w}{r} \partial_\theta e^{-\tau \gamma} \Delta (\Omegalin + \Omegaper).
\end{align}
We will perform energy estimates by multiplying by $p|wr^{-1} \partial_\theta \Omegaper|^{p-2} wr^{-1} \partial_\theta \Omegaper$ for $p=2,2+\sigma$ and integrating by parts. 

Observing that $\Div(r^{-p}\xi) = (2-p) r^{-p}$ away from $r=0$, for $p=2,2+\sigma$, it holds:
\begin{align}
& - \frac{p}{\alpha}\intR \frac{w}{r} \partial_\theta (\xi \cdot \nabla \Omegaper) \frac{w}{r} \partial_\theta \Omegaper \left|\frac{w}{r} \partial_\theta \Omegaper\right|^{p-2} \\ & \quad = - \frac{p}{\alpha} \int_{B_1} (\xi \cdot \nabla \partial_\theta \Omegaper) \partial_\theta \Omegaper |\partial_\theta \Omegaper|^{p-2} \\ & \qquad - \frac{p}{\alpha}\int_{B_1^C} \left(\frac{\xi}{r^p} \cdot \nabla \partial_\theta \Omegaper\right) \partial_\theta \Omegaper \left|\partial_\theta \Omegaper\right|^{p-2} \\ 
& \quad = \frac{2}{\alpha} \int_{B_1} |\partial_\theta \Omegaper|^{p} + \frac{2-p}{\alpha} \int_{B_1^C} \left| \frac{1}{r} \partial_\theta \Omegaper\right|^{p}.
\end{align}
Moreover, the commutation between $wr^{-1} \partial_\theta$ and $\Delta$ gives
\begin{align}
    \Delta\left(\frac{w}{r} \partial_\theta \right) & = \frac{w}{r} \partial_\theta(\Delta) - \frac{2w}{r^2} \partial_r\partial_\theta + \frac{w}{r^3} \partial_\theta - \frac{\partial_r w}{r^2} \partial_\theta + \frac{2\partial_r w}{r} \partial_r \partial_\theta \\
    & = \frac{w}{r} \partial_\theta(\Delta) + I_{B_1^C} \left(- \frac{2}{r^2} \partial_r\partial_\theta + \frac{1}{r^3} \partial_\theta \right) \, .
\end{align}
Therefore, there exists $C > 0$ such that
\begin{align}
    \intR & \frac{w}{r} \partial_\theta (\Delta \Omegaper) \frac{w}{r} \partial_\theta \Omegaper \left|\frac{w}{r} \partial_\theta \Omegaper\right|^{p-2}  \\ & \quad = \intR \Delta \left(\frac{w}{r} \partial_\theta \Omegaper \right) \frac{w}{r} \partial_\theta \Omegaper \left|\frac{w}{r} \partial_\theta \Omegaper\right|^{p-2}  \\ & \qquad -  \int_{B^C_1} \left(- \frac{2}{r^2} \partial_r\partial_\theta + \frac{1}{r^3} \partial_\theta \right) \Omegaper \frac{1}{r} \partial_\theta \Omegaper \left|\frac{1}{r} \partial_\theta \Omegaper\right|^{p-2}\\
    & \quad \leq - (p-1) \intR \left| \nabla \left(\frac{w}{r} \partial_\theta \Omegaper \right) \right|^2 \left|\frac{w}{r} \partial_\theta \Omegaper \right|^{p-2} \\ & \qquad + C  \int_{B_1^C} \frac{1}{r^2} \left|\frac{1}{r} \partial_\theta \Omegaper\right|^p \\ & \quad \leq C  \norm{\frac{w}{r} \partial_\theta\Omegaper}^p_{L^{p}},\label{eq:laplacian_wrtheta}
\end{align}
where the last inequality holds thanks to the assumption on $\tau_0$ in \eqref{eq:ass_K} and the bootstrap assumption \eqref{eq:bootstrapassumption}.
Then, thanks to the fact that the velocity fields are divergence-free, we have that 
\begin{align}
    \intR (\Ulin + \Uper+ \tilde{U}) \cdot \nabla (wr^{-1} \partial_\theta \Omegaper) p|wr^{-1} \partial_\theta \Omegaper|^{p-2} wr^{-1} \partial_\theta \Omegaper = 0.
\end{align}

We can bound the diffusive terms recalling the assumptions on $\tau_0$ and $\tau_{\rm max}$ in \eqref{eq:ass_K} and using for the perturbative part the bootstrap \eqref{eq:bootstrapassumption} and \eqref{eq:laplacian_wrtheta} and for the linear part \eqref{eq:bound_Ulin}:
\begin{align}
    \norm{ \frac{w}{r} \partial_\theta e^{-\tau \gamma} \Delta \Omegalin}_{L^p} & \leq C e^{-\zeta \tau_0}\varepsilon e^{a\tau} \leq C\varepsilon e^{a\tau}/K , \\
    e^{-\tau \gamma}\intR \frac{w}{r} \partial_\theta (\Delta \Omegaper) \frac{w}{r} \partial_\theta \Omegaper \left| \frac{w}{r} \partial_\theta \Omegaper\right|^{p-2} & \leq C\bar{C} e^{a\tau} (\varepsilon/K + M) \left\|\frac{w}{r} \partial_\theta\Omegaper\right\|^{p-1}_{L^p} /K, \\
\norm{ \frac{w}{r} \partial_\theta (\Ulin + \Uper) \cdot \nabla \Omegaper}_{L^p} & \leq (\norm{\nabla \Ulin}_{L^\infty} + \norm{\nabla \Uper}_{L^\infty} ) \norm{\nabla \Omegaper}_{L^p} \\ & \leq C \bar{C} \varepsilon e^{2a\tau} (\varepsilon/K + M) \\ & \quad + C\bar{C}^2 e^{a\tau} (\varepsilon/K + M)/K, \\
        \norm{ \frac{w}{r} (\partial_\theta \Ulin + \partial_\theta \Uper) \cdot \nabla \Omegalin}_{L^p} & \leq (\norm{\nabla \Ulin}_{L^p} + \norm{\nabla \Uper}_{L^p} ) \norm{\nabla \Omegalin}_{L^\infty} \\ & \leq C \varepsilon^2 e^{2a\tau} 
        + C \varepsilon e^{a\tau} \bar{C} e^{a\tau} (\varepsilon/K + M)\\
        & \leq C \varepsilon e^{a\tau} /K
        + C\bar{C}  e^{a\tau} (\varepsilon/K + M) /K.
\end{align}
Finally, we can perform the following bounds for $p=2,2+\sigma$ for the remaining forcing terms:
\begin{align}
    \norm{(\tilde{U} - \bar{U}) \cdot \nabla     \left(\frac{w}{r} \partial_\theta \Omegalin \right)}_{L^p} & \leq \norm{\bar{U} - \tilde{U}}_{L^\infty} \norm{\nabla \left(\frac{w}{r} \partial_\theta \Omegalin \right)}_{L^p} \leq C \varepsilon e^{a\tau}/K, \\
    \norm{\left( \frac{w}{r} \partial_\theta \Ulin \cdot \nabla\right)(\bar{\Omega} - \tilde{\Omega})}_{L^p} & \leq \norm{ \frac{w}{r} \partial_\theta \Ulin}_{L^\infty} \norm{\nabla (\tilde{\Omega} - \bar{\Omega})}_{L^p} \leq C \varepsilon e^{a\tau}/K, \\
    \norm{\left( \frac{w}{r} \partial_\theta \Uper \cdot \nabla\right)(\bar{\Omega} - \tilde{\Omega})}_{L^p} & \leq \norm{ \frac{w}{r} \partial_\theta \Uper}_{L^\infty} \norm{\nabla (\tilde{\Omega} - \bar{\Omega})}_{L^p} \\ & \leq C \bar C e^{a\tau} (\varepsilon/K + M)/K, \\
    \norm{(\Ulin + \Uper) \cdot \nabla \left(\frac{w}{r} \partial_\theta \Omegalin \right)}_{L^p} & \leq \norm{\Ulin + \Uper}_{L^p} \norm{\Delta \Omegalin}_{L^\infty} \\
    & \leq C\varepsilon e^{a\tau} /K+ C \bar{C}  e^{a\tau} (\varepsilon/K + M) /K, 
    \end{align}
    and using the improved estimate on $\norm{\Omegaper}_{L^2 \cap L^{2+\sigma}}$ from \eqref{eq:impr_omega_L2} and \eqref{eq:impr_omega_L2+}, we get
    \begin{align}
    \norm{\frac{w}{r} \partial_\theta \Uper \cdot \nabla \bar{\Omega}}_{L^p} \leq \norm{\Omegaper}_{L^p} \norm{\nabla \bar{\Omega}}_{L^\infty} \leq C \bar{C}^{1/2} (1+\bar{C}^2/K)^{1/2} e^{a\tau} (\varepsilon/K + M).
\end{align}
In the end, the estimate we obtain for $p=2,2+\sigma$ is:
\begin{align}
    \partial_\tau \norm{\frac{w}{r} \partial_\theta \Omegaper}_{L^p} & + \left(\frac{1}{\alpha} - 1 \right) \norm{\frac{w}{r} \partial_\theta \Omegaper}_{L^p}  \\ & \leq C e^{a\tau} (\varepsilon/K + M) (1  + \bar{C}^2/K + C^{1/2} \bar{C}^{1/2} (1+\bar{C}^2/K)^{1/2}) \\
    & \leq C \bar{C}^{1/2} (1+\bar{C}^2/K) e^{a\tau} (\varepsilon/K + M).
\end{align}
Analogously to \textbf{Step 2}, we can conclude by Gronwall's lemma that
\begin{equation}
\label{eq:impr_angder}
    \norm{\frac{w}{r} \partial_\theta\Omegaper}_{L^p} \leq C \bar{C}^{1/2} (1+\bar{C}^2/K) e^{a\tau} (\varepsilon/K + M).
\end{equation}

\textbf{Step 4: estimates for $\nabla \Omega$ in $L^2$ and $L^{2+\sigma}$}

\textit{
    Under the bootstrap assumption \eqref{eq:bootstrapassumption}, it holds:
    \begin{equation}
                \| \nabla \Omega^{\rm per} \|_{L^2 \cap L^{2+\sigma}} \leq C \bar{C}^{1/2}(1+\bar{C}^2/K) e^{a\tau} (\varepsilon/K + M),
    \end{equation}
for all $\tau \in [\tau_0,\tau_{\rm max}]$.}

This time, we apply the derivatives in Cartesian coordinates $\partial_i$ to the equation:
\begin{align}
    \partial_\tau \partial_i \Omegaper & + (-1-1/\alpha - \xi \cdot \nabla / \alpha )\partial_i \Omegaper + (\tilde{U} + \Ulin + \Uper) \cdot \nabla \partial_i \Omegaper \\ & \quad + \partial_i (\tilde{U} + \Ulin + \Uper) \cdot \nabla  \Omegaper + (\Ulin + \Uper + \tilde{U} - \bar{U}) \cdot \nabla \partial_i \Omegalin \\ & \quad + \partial_i(\Ulin + \Uper + \tilde{U} - \bar{U}) \cdot \nabla \Omegalin + \partial_i \Uper \cdot \nabla \bar{\Omega} + \Uper \cdot \nabla \partial_i \bar{\Omega} \\ & \quad + \partial_i (\Ulin + \Uper) \cdot \nabla (\tilde{\Omega} - \bar{\Omega}) + (\Ulin + \Uper) \cdot \nabla \partial_i (\tilde{\Omega} - \bar{\Omega}) \\ & \quad = e^{-\tau \gamma} \Delta \partial_i \Omegalin + e^{-\tau \gamma} \Delta \partial_i \Omegaper.
\end{align}
We perform an energy estimate by multiplying by $p |\partial_i \Omegaper|^{p-2}\partial_i \Omegaper$ for $p=2,2+\sigma$ and then integrating by parts. As before, we get
\begin{equation}
\intR ((\tilde{U} + \Ulin + \Uper) \cdot \nabla ) \partial_i \Omegaper  |\partial_i \Omegaper|^{p-2} \partial_i \Omegaper = 0.
\end{equation}
Thanks to the improved estimates \eqref{eq:improved_ass_U} and \eqref{eq:improvedU2+} for $\Uper$ and \eqref{eq:impr_omega_L2} and \eqref{eq:impr_omega_L2+} for $\Omegaper$, we can bound:
\begin{align}
    \|\Uper \cdot \nabla \partial_i \bar{\Omega}\|_{L^p} & \leq \|\Uper \|_{L^p} \| \nabla \partial_i \bar{\Omega}\|_{L^\infty} \leq C( 1+ \bar{C}^2/K) e^{a\tau} (\varepsilon /K + M), \\
        \| \partial_i \Uper \cdot \nabla \bar{\Omega}\|_{L^p} &  \leq \|\partial_i \Uper \|_{L^p} \| \nabla \bar{\Omega}\|_{L^\infty} \leq C \bar{C}^{1/2} (1+\bar{C}^2/K)^{1/2} e^{a\tau} (\varepsilon/K + M).
\end{align}
Using the improved assumption \eqref{eq:impr_angder} on $w r^{-1} \partial_\theta \Omegaper$ in $L^p(\RR)$ and the fact that $\partial_i \tilde U$ is bounded in $L^\infty(\RR)$, thanks to \Cref{rmk:sigmacond} and $\partial_i \bar U \in L^\infty(\RR)$, 
\begin{align}
    \| \partial_i \tilde{U} \cdot \nabla \Omegaper \|_{L^p} \leq \| \partial_i \tilde{U}\|_{L^\infty} \norm{\frac{w}{r} \partial_\theta \Omegaper}_{L^p} \leq C \bar{C}^{1/2} (1+\bar{C}^2/K)^{1/2} e^{a\tau} (\varepsilon/K + M).
\end{align}
Thanks to the Hardy-type inequality \eqref{eq:hardy} and the convergence to the background estimates in \Cref{lemma:conv_background} and \Cref{rmk:sigmacond}, we can also bound:
\begin{align}
    \|\Uper \cdot \nabla \partial_i (\tilde \Omega - \bar{\Omega})\|_{L^p} & \leq \left\|\frac{\Uper}{r}\right\|_{L^\infty} \| r\partial_r \nabla (\tilde \Omega - \bar{\Omega})\|_{L^p} \leq C\bar C e^{a\tau} (\varepsilon/K + M) /K, \\
        \| \partial_i \Uper \cdot \nabla (\tilde \Omega - \bar{\Omega})\|_{L^p} &  \leq \|\Omegaper \|_{L^\infty} \| \nabla (\tilde \Omega - \bar{\Omega})\|_{L^p} \leq C\bar C e^{a\tau} (\varepsilon/K + M) /K.
\end{align}
The terms arising from the discrepancy of the profiles are bounded similarly to \textbf{Step 3}:
\begin{align}
    & \| (\tilde{U} - \bar{U}) \cdot \nabla \partial_i \Omegalin\|_{L^p} + \|\partial_i \Ulin \cdot \nabla (\tilde{\Omega} - \bar{\Omega})\|_{L^p} +  \|\Ulin\cdot \nabla \partial_i (\tilde{\Omega} - \bar{\Omega})\|_{L^p} \\ & \quad \leq \| \tilde{U} - \bar{U}\|_{L^\infty} \|\nabla \partial_i \Omegalin\|_{L^p} + \| \partial_i \Ulin \|_{L^\infty} \|\nabla (\tilde{\Omega} - \bar{\Omega}) \|_{L^p} + \left\| \frac{\Ulin} {r}\right\|_{L^\infty} \| r\partial_r \nabla (\tilde{\Omega} - \bar{\Omega})\|_{L^p} \\ & \quad \leq C \varepsilon e^{a\tau}/K.
\end{align}
Finally,
\begin{align}
    \| \partial_i \Ulin \cdot \nabla (\Omegalin + \Omegaper) \|_{L^p}  & \leq \|\partial_i \Ulin \|_{L^\infty} \| \nabla (\Omegalin + \Omegaper) \|_{L^p}  \\ & \leq C \varepsilon e^{a\tau} /K + C \bar{C} e^{a\tau} (\varepsilon/K + M)/K \\
    \| \partial_i \Uper \cdot \nabla (\Omegalin + \Omegaper) \|_{L^p}  & \leq \|\partial_i \Uper \|_{L^\infty} \| \nabla (\Omegalin + \Omegaper) \|_{L^p}  \\ & \leq  C \bar{C} e^{a\tau} (\varepsilon/K + M) /K 
    +C \bar{C}^2 e^{a\tau}  (\varepsilon/K + M) /K, \\ 
        \|\Uper \cdot \nabla \partial_i \Omegalin\|_{L^p} & \leq \| \Uper \|_{L^\infty} \|\nabla \partial_i \Omegalin\|_{L^p} \leq C\bar{C} e^{a\tau} (\varepsilon/K + M)/K,
        \\
        \norm{\Ulin \cdot \nabla \partial_i \Omegalin}_{L^p} + \norm{e^{-\tau \gamma} \Delta \partial_i \Omegalin}_{L^p} & \leq C \varepsilon e^{a\tau}/K + C e^{-\zeta \tau_0}\varepsilon e^{a\tau} \leq C \varepsilon e^{a\tau}/K.
\end{align}
Hence, we can bound
\begin{align}
    \partial_\tau \| \nabla \Omegaper\|_{L^p} & + \left( \frac{2}{\alpha p } - 1 - \frac{1}{\alpha}\right)\| \nabla \Omegaper\|_{L^p} \\ & \leq C e^{a\tau} (\varepsilon/K + M) (1 + C\bar{C}^2/K + C(1+\bar{C}^2/K) + C \bar{C}^{1/2} (1+\bar{C}^2/K)) \\
    & \leq C \bar{C}^{1/2} (1+\bar{C}^2/K) e^{a\tau} (\varepsilon/K + M), 
\end{align}
and we conclude as in the previous steps.

With the bound in \textbf{Step 1-4}, we have proved that under the bootstrap assumption \eqref{eq:bootstrapassumption}, it holds that:
\begin{equation}
\|\Uper \|_X \leq C \bar{C}^{1/2} (1+\bar{C}^2/K) e^{a\tau} (\varepsilon/K + M).
\end{equation}
\end{proof}
Now that \Cref{prop:bootstrap} is proved, we can prove \Cref{prop:nonlinear_estimate}:
\begin{proof}[Proof of \Cref{prop:nonlinear_estimate}]
Let $K, \bar C$ to be fixed later and let $\tau_0$ and $\tau_{\rm max}$ be chosen to satisfy \eqref{eq:ass_K} with equality. Notice that this choice of $\tau_{\rm max}$ corresponds to \eqref{eq:tau_max_condition}. Let $\bar \tau \in [\tau_0, \tau_{\rm max}]$ be the maximum time such that \eqref{eq:bootstrapassumption} holds in $ [\tau_0, \bar \tau]$. By \Cref{prop:bootstrap}, there exists $C:=C(\bar U, \eta, a) \geq 1$ such that 
\begin{equation}
\|\Uper\|_X \leq C \bar{C}^{1/2} (1+\bar{C}^2/K) e^{a\tau} (\varepsilon/K + M), \, \quad \forall \tau \in [\tau_0,\bar{\tau}].
\end{equation}
Now, choose {$\bar{C} = (2C/c)^2$} and {$K = (2C/c)^4$}. We get in this way an improved estimate of the form:
\begin{equation}
\|\Uper\|_X \leq {c\bar{C}} e^{a\tau} (\varepsilon/K + M), \, \quad \forall \tau \in [\tau_0,\bar{\tau}].
\end{equation}
We can then prove \eqref{eq:main_nonlin_estimate} on $[\tau_0,\bar{\tau}]$. The presence of the factor {$c<1$} ensures the fact that the bound is not saturated at $\bar{\tau}$, hence $\bar \tau = \tau_{\rm max}$. 
\end{proof}

\begin{proof}[Proof of \Cref{thm:longtimebehavior}]
We apply \Cref{prop:nonlinear_estimate} 
with {$\Uper = U - \tilde{U} - (e^{\lambda (\tau-\tau_0)} P_{\lambda} \Phi + e^{\bar\lambda (\tau-\tau_0)} P_{\bar\lambda} \Phi)$}, solving \eqref{eq:Uper} with initial datum {$\Phi-P_{\lambda,\bar{\lambda}} \Phi$}. 
By the definition of $\varepsilon$ and $M$ in \Cref{prop:nonlinear_estimate}, the inequality in \eqref{eq:nonlinear_bound} coincides with the one in \eqref{eq:main_nonlin_estimate}. Moreover, we have
\begin{equation}
    \label{eq:phi_norm}
\|\Phi\|_{W^{2,2+s}} \leq (\varepsilon + M) e^{a\tau_0} \leq C \|\Phi\|_{W^{2,2+s}},
\end{equation}
where the first inequality comes from the triangular inequality and the second one by \Cref{rmk:petabar} applied to $\Phi$ at $\tau = \tau_0$, recalling that, by the definition of $\varepsilon$, 
   $ \varepsilon e^{a\tau_0} = \|P_{\lambda,\bar{\lambda}} \Phi\|_{W^{2,2+\sigma}}$.

The upper time bound in \eqref{eq:tau_max_condition} gives control up to 
\begin{equation}
\frac{1}{a} \log\left(\frac{1}{K(\varepsilon + M)}\right) =  \tau_0 + \frac{1}{a} \log\left(\frac{1}{K\|\Phi\|_X}\right),
\end{equation}
which implies the estimate in \eqref{eq:nonlinear_bound}.
\end{proof}

%% file: thm_proof.tex
\section{Proof of \Cref{thm:main}}
\label{sec:final}
\begin{proof}[Proof of \Cref{thm:main}, (\ref{item:parttwotheorem})]
We adopt the notation from Sections~\ref{sec:inviscidnonuniqueness}-\ref{sec:maintheoremrevisited}. We first demonstrate the following auxiliary proposition on the structure of solutions with $\nu=1$:

\textbf{Step 1}: \emph{There exist $T, C > 0, v_0 \in H^{2+s}_{\rm df}(\RR)$ depending on $\bar U, \eta, a$ such that the following holds. For all $\varepsilon \in (0,C^{-1})$ {and for all $\psi_0 \in H^{2+s}_{\rm df}(\RR)$ such that $\|\psi_0\|_{H^{2+s}} \leq C^{-1}\varepsilon$}, there exists a solution $\Ulin$ of the linearized equation in self-similar variables \eqref{eq:Ulin}, suitably close to $\varepsilon \Re( e^{\lambda\tau} \eta)$ in the sense that
\begin{equation}
\label{eq:Ulincloseeps}
\|\Ulin-\varepsilon \Re( e^{\lambda\tau} \eta)\|_{W^{2,2+\sigma}(\RR)} \leq \frac \varepsilon 4 e^{a\tau} 
, \quad \forall \tau \geq \log T =: \tau_0 \, ,
\end{equation}
and the solution $w$ to \eqref{eq:NS_forced} with $\nu = 1$ and initial datum $\varepsilon v_0 + \psi_0$ satisfies
\begin{equation}\label{w-in-ss}
    w(y,s) = s^{\frac{1}{\alpha} - 1} W\left(\frac{y}{s^{1/\alpha}},\log s \right)
\end{equation}
with $W$ suitably close to the profile $\Ulin$ of the linearized solution in the sense that
\begin{equation} 
\label{eq:step1_thesis}
\| W - \tilde{U} - \Ulin \|_{W^{2,2+\sigma}} \leq \frac \varepsilon 8 e^{a\tau}
\, , \quad \forall \tau \in \Big[ \tau_0 , \underbrace{\frac{1}{a} \log \frac{1}{C\varepsilon}}_{=: \tau_{\rm max}} \Big] \, .
\end{equation}  
In particular,  by the triangle inequality on \eqref{eq:Ulincloseeps} and \eqref{eq:step1_thesis},
\begin{equation}
\label{eq:step1_final}
\begin{aligned}
 &\| W - \tilde{U} - \varepsilon \Re( e^{\lambda\tau} \eta)\|_{
    W^{2,2+\sigma}} \leq \frac38 \varepsilon e^{a\tau} 
    \, , \quad \forall \tau \in \left[ \tau_0 , \tau_{\rm max} \right] \, .
\end{aligned} 
\end{equation}
}

\textbf{Proof of Step 1:} Let $\tau_0 > 0$ given by \Cref{thm:longtimebehavior} and set $T=e^{\tau_0}$. Let $\delta>0$ be a small constant to be chosen later in terms the data (including $T$, which is fixed). By the thesis (i) of \Cref{thm:initiallayer}, there exists $v_0 \in H^{2+s}_{\rm df}(\RR)$ such that $v\in C([0,T]; H^{2+s}_{\rm df}(\R^2))$,the  solution to the linearized Navier-Stokes equations around $\tilde u$ with datum $v_0$ \eqref{eq:initialayer1}, satisfies
\begin{equation}
\label{eq:tri1}
 \left\| v(\cdot,T) - T^{\frac{1}{\alpha} - 1} \Re\left( T^{\lambda}\eta \left(\frac{\cdot}{T^{1/\alpha}} \right) \right) \right\|_{H^{2+s}(\RR)} \leq  \delta \, .
\end{equation}
By the thesis (ii) of \Cref{thm:initiallayer}, there exists $C > 0$, depending on the data and $\delta$, such that for all $\varepsilon \in [0, C^{-1}]$ {and for all $\psi_0 \in H^{2+s}_{\rm df}(\RR)$ such that $\|\psi_0\|_{H^{2+s}} \leq C^{-1}$} the solution $w \in C([0,T];H^{2+s}(\RR))$ to the Navier-Stokes equations
\begin{equation}
\left\lbrace
\begin{aligned}
    \partial_t w + \mathbb{P}[\tilde{u} \cdot \nabla w + w \cdot \nabla \tilde{u} + w \cdot \nabla w] &= \Delta w \\
    \div w &= 0 \, ,
\end{aligned}
\right.
\end{equation}
 with datum $\varepsilon v_0${$+\psi_0$} \eqref{eq:initialayer2} satisfies
\begin{equation}
\label{eq:tri2}
    \| w(\cdot,t) - \varepsilon v(\cdot, t) \|_{H^{2+s}(\RR)} \lesssim \varepsilon^2 {+\|\psi_0\|_{H^{2+s}}} \, ,\qquad {\forall t\in [0,T]}.
\end{equation} 
Up to further reducing $\varepsilon$ and imposing $\|\psi_0\|_{H^{2+s}} \ll \varepsilon\delta$, we can assume that the right-hand side is smaller than $\varepsilon\delta$.
Hence, by the triangle inequality on \eqref{eq:tri1} and \eqref{eq:tri2},
\begin{align}
    \left\| w(\cdot,T) - T^{\frac{1}{\alpha} - 1} \Re\left( T^{\lambda}\eta \left(\frac{\cdot}{T^{1/\alpha}} \right) \right) \right\|_{H^{2+s}(\RR)} \leq 2 \varepsilon\delta.
    \label{eq:triang_inlayer}
\end{align}
{Moreover, we have}
\begin{equation}\label{eqn:sizew}
    \| w(\cdot,t)  \|_{H^{2+s}(\RR)} \leq \varepsilon \|v \|_{H^{2+s}(\RR)} +     \| w(\cdot,t) - \varepsilon v(\cdot, t) \|_{H^{2+s}(\RR)} \leq C \varepsilon,
\end{equation}
for all $t\in [0,T]$. We now prepare to apply Theorem~\ref{thm:longtimebehavior}. Define $\Phi \coloneqq T^{1-\frac{1}{\alpha}} w(\cdot T^{\frac{1}{\alpha}}, T)$ and $\Psi \coloneqq \Phi - \Re (\varepsilon T^{\lambda} \eta)$; by \eqref{eq:triang_inlayer}, we have that 
\begin{equation}
\label{eq:bound_psi}
\| \Psi\|_{H^{2+s}} \leq C \varepsilon\delta,
\end{equation}
where the constant $C$ depends only on the data (in fact, on $T$) through the self-similar scaling.
Applying the projector $P_{\lambda, \bar\lambda}$ defined in \Cref{sseq:spectral_projection}, we obtain
\begin{equation}
\label{eq:petaphi}
    P_{\lambda,\bar\lambda} \Phi = \varepsilon \Re(T^\lambda \eta) + P_{\lambda,\bar\lambda} \Psi,
\end{equation}
since by the properties of $P_{\lambda}, P_{\bar\lambda}$ in \eqref{eq:prop_projector}, we have
\begin{align}
(P_{\lambda} + P_{\bar\lambda}) \Re(T^\lambda \eta) =(P_{\lambda} + P_{\bar\lambda}) \frac{T^\lambda \eta + T^{\bar{\lambda}}\bar\eta}{2} = \frac{T^\lambda \eta + T^{\bar{\lambda}}\bar\eta}{2} = \Re(T^\lambda \eta).
\end{align}

%

Let  $\Ulin = e^{(\tau-\tau_0)L_{\rm ss}}(P_{\lambda,\bar\lambda} \Phi)$ be the solution of the linearized Euler equations in self-similar variables  \eqref{eq:Ulin} with datum $P_{\lambda,\bar\lambda} \Phi$ at time $\tau_0$. 
We get the bound \eqref{eq:Ulincloseeps} from \eqref{eq:petaphi}, \eqref{eq:bound_psi}, and \Cref{rmk:petabar} applied to $\Psi$:
\begin{align}
\label{eq:UlinW2+}
\|\Ulin - \varepsilon \Re(e^{\lambda \tau} \eta)
\|_{W^{2,2+\sigma}} & = \| e^{(\tau-\tau_0)L_{\rm ss}} P_{\lambda,\bar\lambda} \Psi\|_{W^{2,2+\sigma}} \\ & \leq Ce^{a(\tau-\tau_0)} \|\Psi\|_{L^2} \leq C e^{a\tau} \varepsilon\delta \, ,
\end{align}
for any $\tau \geq \tau_0$.
Therefore, choosing 
$\delta \leq ({4C})^{-1},$ where $C$ is the constant in the previous formula,
we proved the existence of a solution to \eqref{eq:Ulin} satisfying \eqref{eq:Ulincloseeps}. 
To prove \eqref{eq:step1_thesis}, we apply \Cref{thm:longtimebehavior} for some $c \leq 1$ to be chosen later, to obtain 
\begin{equation}
\label{eq:thm15_thesis}
\begin{aligned}
    &\| W - \tilde{U} - \Ulin \|_{
    W^{2,2+\sigma}} \\ & \quad \leq e^{a(\tau-\tau_0)} \left( c\| P_{\lambda,\bar\lambda} \Phi \|_{W^{2,2+\sigma}} + C \| \Phi -  P_{\lambda,\bar\lambda} \Phi\|_{W^{2,2+\sigma}} \right),
\end{aligned}
\end{equation}
for all $\tau \in [\tau_0, \frac{1}{a} \log \frac{1}{C\varepsilon}]$. We can therefore bound the terms $\norm{P_{\lambda,\bar\lambda} \Phi}_{W^{2,2+\sigma}}$ and $\norm{\Phi - P_{\lambda,\bar\lambda} \Phi}_{W^{2,2+\sigma}}$ on the right-hand side of \eqref{eq:thm15_thesis} as follows: since $P_{\lambda,\bar\lambda} \Phi = \Ulin|_{\tau=\tau_0}$, by \eqref{eq:Ulincloseeps} and a triangle inequality, we have 
\begin{equation}
\label{eq:boundpeta}
\norm{P_{\lambda,\bar\lambda} \Phi}_{W^{2,2+\sigma}} \leq \frac 1 4 \varepsilon e^{a\tau_0} + \varepsilon e^{a\tau_0}
 \|\eta\|_{W^{2,2+\sigma}}
. 
\end{equation}
Moreover, $\Phi - P_{\lambda,\bar\lambda} \Phi = \Psi +  \varepsilon\Re(T^\lambda \eta) - P_{\lambda,\bar\lambda} \Phi$ by the definition of $\Phi$, and by evaluating \eqref{eq:UlinW2+} at time $\tau=\tau_0$, we have
\begin{align}
\norm{\Phi - P_{\lambda,\bar\lambda} \Phi}_{W^{2,2+\sigma}} & \leq \norm{\Psi}_{W^{2,2+\sigma}} + \| P_{\lambda,\bar\lambda}\Phi - \varepsilon\Re(e^{\lambda\tau_0} \eta) \|_{W^{2,2+\sigma}} \\ &\leq C\varepsilon \delta.  \label{eq:boundpeta2}
\end{align}

Therefore, choosing in \eqref{eq:thm15_thesis}
\begin{equation}
c= \frac{1}{16 
 \|\eta\|_{W^{2,2+\sigma}}} \leq 1 
\qquad \mbox{and} \qquad
    \delta \leq \frac {e^{a\tau_0}} {32C} \, ,
\end{equation}
where $C$ is the constant in \eqref{eq:boundpeta2},
from \eqref{eq:thm15_thesis}, \eqref{eq:boundpeta} and \eqref{eq:boundpeta2}, we obtain \eqref{eq:step1_thesis}. 

    \textbf{Step 2:} We prove \Cref{thm:main}, (\ref{item:parttwotheorem}). Given $T, C > 0, v_0 \in H^{2+s}_{\rm df}(\RR)$ from \textbf{Step 1}, we apply \textbf{Step 1} with $\varepsilon_\nu = C^{-1} \nu^{a/\gamma}$, $\psi_0^\nu \in H^{2+s}_{\rm df}(\RR)$ such that $\|\psi_0^\nu\|_{H^{2+s}} \leq C^{-2} \nu^{a/\gamma}$, $c_0 = C^{-1}\|v_0\|_{H^{2+s}}$, to obtain $w^\nu$, the solution of the Navier-Stokes equation with unit viscosity and initial datum $\varepsilon_\nu  v_0 +\psi_0^\nu$ whose profile in self-similar variables $W^\nu$ according to~\eqref{w-in-ss} satisfies~\eqref{eq:step1_final}.
Then we define, in the notation of~\eqref{eq:LnuTnuUnudef},
\begin{equation}
u^{\nu}_{\rm in}(x) = \varepsilon_\nu  U_\nu v_0(y) \, , \text{ where } y = \frac{x}{L_\nu}
\end{equation}
and $u^\nu$ by introducing the rescaling 
\begin{equation}
\label{eq:nuscaling} 
u^\nu(x,t) = U_\nu w^\nu(y,s) \, , \text{ where } y = \frac{x}{L_\nu} \, , \; s=e^\tau= \frac{t}{T_\nu}
\end{equation}
whose initial datum is $u_0^\nu(x) = u^{\nu}_{\rm in}(x) + U_\nu \psi_0^\nu(y)$.
As observed in \Cref{sec:introduction}, $u^\nu$ solves the Navier-Stokes with viscosity $\nu$, and its initial datum $u^\nu_0$ satisfies, by \eqref{eq:norms_relation}, $
\| u_0^\nu\|_{Y_\nu} =  \| w_0^\nu \|_{ H^{2+s}}  =  \| \varepsilon_\nu v_0 {+\psi_0^\nu}\|_{ H^{2+s}} {\leq 2} C^{-1} \nu^{a/\gamma},$ and $\| u_0^\nu - u^{\nu}_{\rm in} \|_{Y_\nu} = \| \psi_0^\nu\|_{H^{2+s}}\leq C^{-2} \nu^{\rm \kappa_{\rm c}}$.
Since 
   $ \frac{y}{s^{1/\alpha}} = \frac{x}{t^{1/\alpha}}$ and $ \log s = \log t - \gamma^{-1}\log \nu,$
we have that the self similar profile associated to $u^\nu$ is a translation in time by $\gamma^{-1}\log \nu $ of $W^\nu$:
\begin{equation}
\label{eq:utildenu}
U^\nu\left(\cdot ,  \tau\right) = W^\nu(\cdot ,\tau - \gamma^{-1}\log \nu ),
\end{equation}
Correspondingly, we define $\tilde{U}^\nu$ to be the translation in time by $\gamma^{-1} \log(\nu)$ of $\tilde U$:
\begin{equation}
\label{eq:rewriting_Utildenu}
\tilde{U}^\nu\left(\cdot, \tau\right) = \tilde{U}\left(\cdot ,\tau - \gamma^{-1}\log \nu \right).
\end{equation}
Its rewriting in standard variables solves the heat equation with initial datum zero, forcing term $\bar{f}$ and viscosity $\nu$.

We now prove equi-boundedness in $\nu$ sufficiently small for the sequence $\{u^\nu\}_{\nu > 0}$ in $L^\infty(0,1; W^{1,p}(\RR))$ for any $p\leq 2/\alpha$, which makes the space we are considering scaling (super)critical.
For the problem with normalized viscosity in the initial layer we can invoke the result from \eqref{eqn:sizew} to conclude that 
\begin{equation}
\| w^\nu \|_{L^\infty(0,T;H^{2+s}(\RR))} \leq C \nu^{a/\gamma},
\end{equation}
and then, by the scaling in \eqref{eq:nuscaling}, we conclude that for all $\nu \leq 1$,
\begin{align}
\| u^\nu \|_{L^\infty(0,T\nu^{1/\gamma}; W^{1,p}(\RR))} & 
\leq C(\nu^{2/(\alpha p) - 1} + \nu^{2/(\alpha p) + 1/\alpha - 1}) \| w^\nu \|_{L^\infty(0,T; H^{2+s}(\RR))} \\ & \leq C. \label{eq:step1inlayer}
\end{align}

We are thus left to prove
 \begin{equation}
 \label{eqn:bohhh}
 \| u^\nu \|_{L^\infty(T\nu^{1/\gamma},1; W^{1,p}(\RR))} \leq C.
 \end{equation}
  To this end, we will bound its self-similar profile in a stronger norm, since
\begin{align}
    \|u^\nu(\cdot,t)\|_{W^{1,p}(\RR)} & \leq (t^{2/(\alpha p) - 1 + 1/\alpha} + t^{2/(\alpha p) - 1}) \|U^\nu (\cdot,\log t) \|_{W^{1,p}(\RR)} \\ & \leq  2 \|U^\nu(\cdot,\log t) \|_{W^{2,2+\sigma}}.
    \label{eq:step1longtime}
\end{align}
We are therefore left to provide a uniform bound on $\|U^\nu(\cdot, \tau) \|_{W^{2,2+\sigma}}$ for $\tau \in (\log T + \gamma^{-1}\log \nu,0)$. We observe that $U^\nu$ in self-similar variables is composed of the sum of $\tilde U^\nu$, the solution of the forced heat equation with viscosity $\nu$ which is suitably close to $\bar U$; the solution of the linearized equation along the unstable eigenfunction $C^{-1}t^a\Re \left((e^{ib( \tau+1/\gamma \log\nu)} \eta\right)$, whose norms grow as $e^{a\tau}$ (notice that $(e^{\tau}{\nu^{1/\gamma}})^{ib}$ is just a phase factor which does not influence relevant norms); and a perturbation which, by \eqref{eq:boundUnu}, has smaller size than the linear part.

More precisely, by \eqref{eq:ineq_zeta} applied at $\tau - \gamma^{-1}\log(\nu)$, we get for every 
$\tau \geq \gamma^{-1}\log \nu$,
\begin{align}
\|\tilde{U}^\nu(\cdot, \tau) - \bar U\|_{W^{2,2+\sigma}} & \leq \|\tilde{U}(\cdot, \tau - \gamma^{-1}\log \nu ) - \bar U\|_{W^{2,2+\sigma}} \\ & \leq Ce^{-\zeta(\tau- \gamma^{-1}\log \nu )}.\label{eq:conv_nu}
\end{align}
In particular, for $\tau \in (\log T + \gamma^{-1}\log \nu,0)$, this quantity is bounded independently of $\nu$.
By the triangle inequality, from \eqref{eq:conv_nu} and the smoothness of $\bar U$, we obtain
\begin{align}
\|\tilde U^\nu(\cdot,\tau)\|_{W^{2,2+\sigma}} & \leq \|\tilde{U}^\nu(\cdot, \tau) - \bar U\|_{W^{2,2+\sigma}} + \|\bar U\|_{W^{2,2+\sigma}}\leq C.\label{eq:boundUtildenu}
\end{align}
For the linear part, we can rewrite $
e^{\lambda\tau}\nu^{-\lambda/\gamma} = e^{a\tau}\nu^{-a/\gamma} z$, with $z \in \mathbb C$, $|z|=1$, and bound
\begin{equation}
\label{eq:bound_lin}
\norm{C^{-1}e^{a\tau}\Re \left(e^{ib( \tau+1/\gamma \log\nu)} \eta\right)}_{W^{2,2+\sigma}} \leq C^{-1} e^{a\tau} \| \Re(z\eta)\|_{W^{2,2+\sigma}}.
\end{equation}

By \textbf{Step 1}, evaluating \eqref{eq:step1_final} at $t = e^{\tau} \nu^{1/\gamma}$, we have that $U^\nu$ satisfies:
\begin{equation} 
\label{eq:boundUnu}
\norm{U^\nu\left(\cdot, \tau\right) - \tilde{U}^\nu\left(\cdot, \tau\right) - 
C^{-1}e^{a\tau}\Re \left(e^{ib( \tau+1/\gamma \log\nu)} \eta\right)
}_{W^{2,2+\sigma}} \leq \frac38 C^{-1} e^{a\tau} ,
\end{equation} 
for all $t \in [T \nu^{1/\gamma}, 1]$. 

By  \eqref{eq:step1longtime}, \eqref{eq:boundUtildenu}, \eqref{eq:bound_lin} and \eqref{eq:boundUnu}, we conclude \eqref{eqn:bohhh}.


Combining \eqref{eqn:bohhh} and \eqref{eq:step1inlayer}, we get equi-boundedness in  $L^\infty(0,1; W^{1,p}(\RR))$, which gives weak convergence to a limit $u^E$ up to a subsequence in this space {for $p \leq 2/\alpha$}. For the $p<2/\alpha$ case, we can upgrade it to strong convergence via \Cref{eq:compactness_lemma}.

To prove that the limit $u^{\rm E}$ is actually not identically $\bar{u}$, let us now analyze the linear part at $t=1$, where $u^\nu(\cdot, 1) = U^\nu(\cdot, 0)$ and $\bar u (\cdot,1) = \bar U$; with the normalization adopted in \Cref{prop:lin_part}, recalling that $\|\Re(z\eta)\|_{L^2(\RR)} = \sqrt{2}/2$ for any $z \in \mathbb C, |z| = 1$ by \Cref{rmk:reeta}, we have
\begin{equation}
\label{eq:size_linear}
\norm{C^{-1}e^{a\tau}\Re \left(e^{ib( \tau+1/\gamma \log\nu)} \eta\right)}_{L^2} = \frac{\sqrt{2}}{2} C^{-1} e^{a\tau}.
\end{equation}
We can conclude that at time $t=1$ the limit of the sequence is different from $\bar{u}$, because, combining \eqref{eq:size_linear}, \eqref{eq:boundUnu}, and~\eqref{eq:conv_nu}, we obtain
\begin{align}
& \| u^\nu(\cdot,1) - \bar{u}(\cdot,1)\|_{L^2(\RR)} = \| U^\nu(\cdot,1) - \bar{U}(\cdot,1)\|_{L^2(\RR)} \\ & \quad \geq  C^{-1}\norm{\Re (e^{ib(1/\gamma \log\nu)} \eta)}_{L^2(\RR)} - \norm{\tilde{U}^\nu - \bar{U}}_{L^2(\RR)} \\ & \qquad - \norm{U^\nu - \tilde{U} - C^{-1} \Re (e^{ib(1/\gamma \log\nu)} \eta)}_{L^2(\RR)} \\
 & \quad \geq \left(\frac{\sqrt 2}{2} - \frac 3 8 \right) \ C^{-1} - C \nu^{\zeta/\gamma}.
\end{align}
For $\nu \to 0$, the limit of $\| u^\nu(\cdot,1) - \bar{u}(\cdot,1)\|_{L^2(\RR)}$ is strictly positive.
\end{proof}

\begin{proof}[Proof of \Cref{thm:main}, (\ref{item:mainthm1})] 
We define $w^\nu(y,s) = U_\nu^{-1} u^\nu(L_\nu y,T_\nu s)$, according to \eqref{eq:nuscaling}, which solves the forced Navier-Stokes equations in $\R^2 \times (0,+\infty)$ with unit viscosity.
We also define $\tilde u^\nu $ to solve the heat equation with initial datum zero, forcing term $\bar{f}$ and viscosity $\nu$, which satisfies $\tilde u^\nu(x,t) = U_\nu \tilde u(x/L_\nu,t/T_\nu)$, whose self-similar profile $\tilde{U}^\nu$ is suitably close to $\bar U$ by~\eqref{eq:conv_nu}.

For every $\nu \leq 1$, by the triangle inequality and a change of variables, we have
\begin{align}
\| (u^\nu - \bar{u})(\cdot,t) \|_{ W^{1,p}(\RR)}& \leq 
\| (u^\nu - \tilde{u}^{\nu})(\cdot,t) \|_{ W^{1,p}(\RR)} +\| (\tilde{u}^\nu - \bar{u})(\cdot,t) \|_{ W^{1,p}
(\RR)} \\ & \leq \nu^{(-1+ 2/(\alpha p))/\gamma}\| (w^\nu-\tilde u)(\cdot, t\nu^{-1/\gamma}) \|_{ W^{1,p}(\RR)} \\ & \qquad +\| (\tilde{u}^\nu - \bar{u})(\cdot,t) \|_{ W^{1,p}(\RR)} , \label{eq:step2_triang}
\end{align}
where the power of $\nu$ is positive since we are working in scaling-supercritical spaces. Let us fix $T_{\rm max}>0$.
We now show that the second term converges to $0$ uniformly in $[0, T_{\rm max}]$ for all $p<2/\alpha$, and it is bounded for $p=2/\alpha$. Indeed,  since for any $p\leq2/\alpha$, $\bar{u} \in L^\infty(0,T; W^{1,p}(\RR))$ by \cite{vishik2018instability1,vishik2018instability2} and $\tilde{u} \in L^\infty(0,T; W^{1,p}(\RR))$ by \Cref{rmk:tildes_Lp}, for every $t \in [0,T\nu^{1/\gamma}]$ we have
\begin{align}
\norm{(\tilde u^\nu - \bar u)(\cdot,t)}_{W^{1,p}(\RR)} 
    & \leq \nu^{(-1+2/(\alpha p))/\gamma} \norm{(\tilde u - \bar u)(\cdot,t\nu^{-1/\gamma})}_{W^{1,p}(\RR)}, \label{eq:backgroundsu} \\
    & \leq C \nu^{(-1+2/(\alpha p))/\gamma}.
\end{align}
Passing to self-similar profiles via a change of variables and by \eqref{eq:conv_nu}, for every $t \in [T\nu^{1/\gamma}, T_{\rm max}]$, we have
\begin{align}
    \| (\tilde u^\nu - \bar u)(\cdot,t)\|_{W^{1,p}(\RR)} \leq t^{2/(\alpha p) -1} \|\tilde U^\nu(\cdot, \log t) - \bar U\|_{W^{1,p}(\RR)} & \leq Ct^{2/(\alpha p) -1-\zeta} \nu^{\zeta/ \gamma}.
\end{align}
If $2/(\alpha p) - 1 - \zeta > 0$, the right-hand side is smaller than
$ CT_{\rm max}^{2/(\alpha p) -1-\zeta} \nu^{\zeta/ \gamma}$ in $[T\nu^{1/\gamma}, T_{\rm max}]$; 
else, it is smaller than
$CT^{2/(\alpha p) -1-\zeta} \nu^{(2/(\alpha p)-1)/\gamma}$. 
In both cases, $\| (\tilde u^\nu - \bar u)(\cdot,t) \|_{W^{1,p}(\RR)}$ goes to zero polynomially in $\nu$, and overall
\begin{equation}
\label{eq:u_nu_zero}
\lim_{\nu \to 0} \| \tilde{u}^\nu - \tilde{u} \|_{L^\infty(0,T_{\rm max}; W^{1,p}(\RR))} =0 \, , \quad \text{ for } p<2/\alpha \, .
\end{equation}
Hence, to bound $\| (u^\nu - \bar{u})(\cdot,t) \|_{ W^{1,p}(\RR)}$ in view of \eqref{eq:step2_triang}, it is enough to prove
\begin{equation}
\label{eq:final1}
 \lim_{\nu \to0 }  \nu^{(-1+ 2/(\alpha p))/\gamma} \| w^\nu - \tilde{u} \|_{L^\infty(0,T_{max}\nu^{-1/\gamma}; W^{1,p}(\RR))} =0
 .
\end{equation}

In the initial layer, for $s \in [0,T]$, this equality follows by Theorem~\eqref{thm:initiallayer}, (ii), applied with initial datum $w^\nu_0 = U_\nu^{-1} u^\nu_0(L_\nu \cdot)$ and $\varepsilon=0$:
\begin{align}
\|(w^\nu -\tilde u)(\cdot,s)\|_{W^{1,p}(\RR)} & \leq \|(w^\nu -\tilde u)(\cdot,s)\|_{H^{2+s}(\RR)} \\ & \leq C \|w_0^\nu\|_{H^{2+s}(\RR)} = C \| u_0^\nu \|_{Y_\nu}. \label{eqn:initiallayervisc1}
\end{align}
The assumption that $\| w_0^\nu\|_{H^{2+s}(\RR)}$ is small is satisfied. 

After the initial layer, we are left to prove that 
\begin{equation}
\label{eqn:lefttodo}
 \lim_{\nu \to0 }  \nu^{(-1+ 2/(\alpha p))/\gamma} \| w^\nu - \tilde{u} \|_{L^\infty(T,T_{\rm max}\nu^{-1/\gamma}; W^{1,p}(\RR))} =0
 .
\end{equation}
To this end, we first pass to self-similar profiles. For every $\tau \in [\log T,\log T_{\rm max}-\gamma^{-1} \log\nu]$,
\begin{align}
& \nu^{(-1+ 2/(\alpha p))/\gamma} \| w^\nu(\cdot, e^\tau) - \tilde{u}(\cdot, e^\tau) \|_{W^{1,p}(\RR)} \\ & \quad \leq C e^{(\tau +\gamma^{-1}\log \nu) (2/(\alpha p) - 1)} \|W^\nu(\cdot, \tau) - \tilde{U}(\cdot, \tau)\|_{W^{1,p}(\RR)} \\ &\quad \leq CT_{\rm max}^{2/(\alpha p) - 1}\|W^\nu(\cdot, \tau) - \tilde{U}(\cdot, \tau) \|_{W^{2, 2+\sigma}(\RR)}.\label{eq:selfsim_profiles}
\end{align}
We apply \Cref{thm:longtimebehavior} to the self-similar profile associated to $w^\nu - \tilde{u}$ {at time $T$}, namely $\Phi := (W^\nu- \tilde U)(\cdot, \log T)=T^{-1/\alpha + 1} (w^\nu - \tilde{u})(\cdot T^{1/\alpha}, T)$, to obtain 
\begin{align}
 & \| W^\nu - \tilde{U} -e^{ (\tau-\tau_0)L_{\rm ss}} P_{\lambda, \bar\lambda} \Phi\|_{W^{2,2+\sigma}} \\ & \quad \leq  e^{a(\tau-\tau_0)} \left( c\| P_{\lambda,\bar\lambda} \Phi \|_{W^{2,2+\sigma}} + C \| \Phi -  P_{\lambda,\bar\lambda} \Phi\|_{W^{2,2+\sigma}} \label{eq:triangular_step1_1}\right),\end{align}
for all $\tau \in [\log(T), -a^{-1}\log(C\|\Phi\|_{W^{2,2+\sigma}})]$.
We observe by \eqref{eqn:initiallayervisc1} and \Cref{rmk:petabar} that
\begin{align}
\label{eq:triangular_step1_2}
\|P_{\lambda,\bar\lambda}\Phi\|_{W^{2,2+\sigma}(\RR)} + \| \Phi -P_{\lambda,\bar\lambda}\Phi\|_{W^{2,2+\sigma}(\RR)} & \leq \| \Phi \|_{W^{2,2+\sigma}(\RR)} \\ & \leq C \| u_0^\nu \|_{Y_\nu},
\end{align}
and that 
\begin{equation}
\label{eq:triangular_step1_3}
    \| e^{ (\tau-\tau_0)L_{\rm ss}} P_{\lambda, \bar\lambda} \Phi\|_{W^{2,2+\sigma}(\RR)} \leq C e^{a\tau}\| u_0^\nu \|_{Y_\nu},
\end{equation}
for a constant $C$ depending on the data (and in particular on $T=e^{\tau_0}$ due to the change to self-similar variables).
By the triangle inequality, \eqref{eq:triangular_step1_1}, \eqref{eq:triangular_step1_2}, and \eqref{eq:triangular_step1_3}, we obtain
\begin{align}
 \| W^\nu - \tilde{U}\|_{W^{2,2+\sigma}(\RR)} &\leq \| W^\nu - \tilde{U} -e^{ (\tau-\tau_0)L_{\rm ss}} P_{\lambda, \bar\lambda} \Phi\|_{W^{2,2+\sigma}(\RR)} \\ & \quad + \| e^{ (\tau-\tau_0)L_{\rm ss}} P_{\lambda, \bar\lambda} \Phi\|_{W^{2,2+\sigma}(\RR)} \\ &  \leq C e^{a\tau}\| u_0^\nu \|_{Y_\nu},
\end{align}
where by the assumption, this estimate is valid in the interval $\tau \in [\log(T), \log T_{max} -\gamma^{-1} \log \nu]$ for $\nu$ sufficiently small.
Hence, in such interval
\begin{align}
 \| W^\nu - \tilde{U}\|_{W^{2,2+\sigma}(\RR)} &\leq  C T_{\rm max}^{a} \nu^{-a/\gamma} \| u_0^\nu \|_{Y_\nu}.
\end{align}
This proves \eqref{eqn:lefttodo}.
By \eqref{eq:step2_triang}, \eqref{eq:final1} and \eqref{eqn:lefttodo}, the claim is proved. 
\end{proof}

\begin{lemma}[Compactness lemma]
\label{eq:compactness_lemma}
Let $p > 2$. Let $T > 0$ and $(u^{(k)})_{k \in \N}$ be a sequence of strong solutions to the Navier-Stokes equations on $\R^2 \times (0,T)$ with viscosity $\nu_k \to 0^+$ and fixed body force $\bar{f} \in L^1_t L^2_x(\R^2 \times (0,T))$ with $\bar{g} := \curl \bar{f} \in L^1_t L^p_x(\R^2 \times (0,T))$. Suppose that $u_0^{(k)} \to 0$ in $L^2(\R^2)$, $\omega_0^{(k)} \to 0$ in $L^p(\R^2)$, and
\begin{equation}
    \label{eq:assumptiononukinlemma6}
\sup_k \| \nabla^2 u^{(k)} \|_{L^\infty_t L^2_x(\R^2 \times (\varepsilon,T))} < +\infty \, , \quad \forall \varepsilon \in (0,T) \, .
\end{equation}
Then there exists a subsequence (not relabeled) such that $u^{(k)}$ converges strongly in $L^\infty_t L^2_x(\R^2 \times (0,T))$ to a weak solution $u$ of the forced 2D Euler equations and $\omega^{(k)}$ converges strongly in $L^\infty_t L^p_x(\R^2 \times (0,T))$.
\end{lemma}


\begin{proof}
    First of all, we may assume that $\bar{f}$ is divergence free by incorporating the gradient in its Helmholtz decomposition into the pressure; the new $\bar{f}$ also belongs to $L^1_t L^2_x(\R^2 \times (0,+\infty))$ (actually, its norm does not increase in this procedure).
    
    Under these assumptions, we have the \emph{a priori} global-in-space energy estimates and vorticity estimates:
    \begin{align}
        & \| u^{(k)} \|_{L^\infty_t L^2_x(\R^2 \times (0,T))} + \nu_k^{1/2} \| \nabla u^{(k)} \|_{L^2_{t,x}(\R^2 \times (0,T))} \\ & \quad \lesssim \| u_0^{(k)} \|_{L^2} + \| \bar{f} \|_{L^1_t L^2_x(\R^2 \times (0,T))} \lesssim 1,
    \end{align}
    \begin{align}
        & \| \omega^{(k)} \|_{L^\infty_t L^p_x(\R^2 \times (0,T))} + \nu_k^{1/p} \| \nabla |\omega^{(k)}|^{\frac{p}{2}} \|_{L^2_{t,x}(\R^2 \times (0,T))}^{\frac{2}{p}} \\ & \quad \lesssim \| \omega_0^{(k)} \|_{L^p} + \| \bar{g} \|_{L^1_t L^p_x(\R^2 \times (0,T))} \lesssim 1 \, .        \label{eq:globalvorticityestinmuhlemma}
    \end{align}
    By Sobolev embedding, we have $u^{(k)} \in L^\infty_t C^\alpha_x(\R^2 \times (0,T))$ with uniform-in-$k$ estimates ($\alpha = 1-2/p$). Following the standard techniques for constructing weak solutions of 2D Euler, we estimate the time derivative in a negative Sobolev space and apply the vorticity estimates to obtain strong convergence of the velocity field. Namely, we pass to a subsequence (without relabeling) such that
    \begin{equation}
        u^{(k)} \to u \text{ in } L^\infty_t (L^2 \cap L^\infty)_x(B_R \times (0,T)) \, , \quad \forall R > 0 \, .
    \end{equation}
    Now we control the solution in the far field. We require the local energy equality
    \begin{equation}
        (\p_t - \nu_k \Delta) \frac{|u^{(k)}|^2}{2} + \nu_k |\nabla u^{(k)}|^2 + \div \left[ u^{(k)} \left( \frac{|u^{(k)}|^2}{2} + p^{(k)} \right) \right] = u^{(k)} \cdot \bar{f} \, .
    \end{equation}
    Let $R > 0$. We multiply by an appropriate cut-off function $\chi_R^2(x) = \chi^2(x/R)$ with $\chi \equiv 1$ on $\R^2 \setminus B_2$ and $\chi \equiv 0$ on $B_1$ and integrate over $\R^2 \times (0,t)$:
    \begin{equation}
    \begin{aligned}
        &\frac{1}{2} \int |u^{(k)}|^2 \chi_R^2 \, dx + \nu_k \int_0^t \int |\nabla u^{(k)}|^2 \chi_R^2 \, dx \, ds = \frac{1}{2} \int |u^{(k)}_0|^2 \chi_R^2 \, dx
        \\
        &\quad  + \frac{\nu_k}{2} \int_0^t \int \Delta \chi^2_R |u^{(k)}|^2 \,dx \,ds + \int_0^t \int \left( \frac{|u^{(k)}|^2}{2} + p^{(k)} \right) u^{(k)} \cdot \nabla \chi_R^2 \, dx \, ds \\ & \quad + \int_0^t \int \bar{f} \cdot u^{(k)} \chi_R^2 \, dx \, ds.
    \end{aligned}
    \end{equation}
    Since $p^{(k)} = (-\Delta)^{-1} \div \div u^{(k)} \otimes u^{(k)}$ ($\bar{f}$ is divergence-free), we have by the Calder{\'o}n-Zygmund estimates that $p^{(k)}$ is uniformly bounded in $L^\infty_t L^q_x(\R^2 \times (0,T))$ for all $q \in (1,+\infty)$. We then estimate the ``boundary terms" above using the global \emph{a priori} bounds. Consequently, we have
    \begin{align}
        \sup_{t \in (0,T)} \int |u^{(k)}|^2(x,t) \chi_R^2 \, dx &\leq o_{k \to +\infty}(1) + o_{R \to +\infty}(1) + \int_0^T \int |\bar{f}|^2 \chi_R^2 \, dx \, dt \\ & = o_{k \to +\infty}(1) + o_{R \to +\infty}(1) \, .
    \end{align}
    This allows us to upgrade the local-in-space strong convergence to the global strong convergence $u^{(k)} \to u$ in $L^\infty_t L^2_x(\R^2 \times (0,T))$.
    
    Next, we interpolate between the assumption~\eqref{eq:assumptiononukinlemma6} on $\nabla^2 u^{(k)}$ and the strong convergence of $u^{(k)}$ to obtain
    \begin{equation}
        \label{eq:almostglobalvorticityconvergence}
        \omega^{(k)} \to \omega \text{ in } L^\infty_t L^p_x(\R^2 \times (\varepsilon,T)) \, , \quad \forall \varepsilon \in (0,T) \, .
    \end{equation}
    We must therefore control the vorticity near the initial time. This is done by the global vorticity estimate~\eqref{eq:globalvorticityestinmuhlemma} on $\R^2 \times (0,\varepsilon)$. That is,
    \begin{equation}
        \| \omega^{(k)} \|_{L^\infty_t L^p_x(\R^2 \times (0,\varepsilon))} \lesssim o_{k \to +\infty}(1) + o_{\varepsilon \to 0^+}(1) \, .
    \end{equation}
    Combining this with~\eqref{eq:almostglobalvorticityconvergence} completes the proof.
\end{proof}

A related tightness argument is contained in~\cite[Lemma 2.2]{Wu2021}.

%% file: spectralappendix.tex
\section{Abstract spectral analysis}
\label{sec:spectralperturbationtheory}

In this section, we formalize a spectral theory applicable to certain singular perturbation problems, including those in Section~\ref{sec:spectral_pb}. This theory abstracts and extends arguments in~\cite{vishik2018instability1,vishik2018instability2,albritton2023instability,MR4429263}.

\begin{proposition}
	\label{pro:mainspectrallemma}
Let $X$ be a Banach space. Suppose that $M_k$ ($k \in \N$) and $M_\infty$ are closed, densely defined operators on $X$, and $K_k$ ($k \in \N$) and $K_\infty : X \to X$ are compact. Suppose that $K_k \to K_\infty$ as $k \to +\infty$. Let $L_k := M_k + K_k$ and $L_\infty := M_\infty + K_\infty$.

(i) Suppose that $\lambda \in \rho(L_\infty)$. Suppose that $R(\lambda,M_k)$ ($k \gg 1$) and $R(\lambda,M_\infty)$ exist,
\begin{equation}
	\label{eq:strongopconv}
R(\lambda,M_k) x \to R(\lambda,M_\infty) x \text{ as } k \to +\infty \, , \quad \forall x \in X
\end{equation}
(convergence in the so-called strong operator topology), and
\begin{equation}
	\label{eq:convholdslocallyuniformy}
R(\lambda,M_k) K_\infty \to R(\lambda,M_\infty) K_\infty \text{ in } \mathcal{B}(X) \text{ as } k \to +\infty \, .
\end{equation}
Then
$R(\lambda,L_k)$ exists for $k \gg 1$, and
\begin{equation}
	\label{eq:convholdslocallyuniformy2}
R(\lambda,L_k)x \to R(\lambda,L_\infty)x \text{ as } k \to +\infty \, , \quad \forall x \in X \, .
\end{equation}

(ii) If $E \subset \rho(L_\infty)$ is compact and (i) holds for every $\lambda \in E$, with the convergences~\eqref{eq:strongopconv} and~\eqref{eq:convholdslocallyuniformy} holding uniformly in $\lambda \in E$, then~\eqref{eq:convholdslocallyuniformy2} holds uniformly in $\lambda \in E$.

(iiia) Suppose that $\lambda_\infty$ is an isolated eigenvalue of $L_\infty$ and~\eqref{eq:strongopconv}-\eqref{eq:convholdslocallyuniformy} hold uniformly in $\lambda \in \bar{B}_\delta(\lambda_\infty)$. Let $\varepsilon > 0$. Then, whenever $k \gg_\varepsilon 1$, we have
\begin{equation}
\sigma(L_k) \cap \bar{B}_\delta(\lambda_\infty) \subset B_\varepsilon(\lambda_\infty) \, .
\end{equation}
Additionally, there exists an isolated eigenvalue $\lambda_k \in \sigma(L_k)$ of finite multiplicity satisfying $|\lambda_k - \lambda_\infty| < \varepsilon$.


(iiib) Suppose furthermore that $\lambda_\infty$ is geometrically (resp. algebraically) simple. Then, for $k \gg 1$, every spectral value in $\sigma(L_k) \cap \bar{B}_\delta(\lambda_\infty)$ is a geometrically (resp. algebraically) simple eigenvalue of finite multiplicity.
\end{proposition}

\eqref{eq:strongopconv} implies~\eqref{eq:convholdslocallyuniformy} when $K_\infty$ can be approximated by finite rank operators (e.g., in a Hilbert space).

\begin{proof}
(i) First, the uniform boundedness principle yields uniform operator bounds on $\| R(\lambda,M_k) \|_{X \to X}$. We can use this,~\eqref{eq:convholdslocallyuniformy}, and $K_k \to K_\infty$ to prove that $R(\lambda,M_k) K_k \to R(\lambda, M_\infty) K_\infty$. Then, since
\begin{equation}
 R(\lambda,L_k) = [I-R(\lambda,M_k)K_k]^{-1} R(\lambda,M_k) \, ,
 \end{equation}
 Lemma~\ref{lem:stabilityofboundedinvertibility} and the assumptions yield that $R(\lambda,L_k)$ exists and converges.

(ii) The proof is the same but with uniform convergence.

(iiia) For $k \gg_\varepsilon 1$, the spectral projector
\begin{equation}
P_k = \frac{1}{2\pi i} \int_\gamma R(\lambda,L_k) \, d\lambda
\end{equation}
where $\gamma$ is the curve $\p B_\varepsilon(\lambda_\infty)$, oriented counterclockwise, is well defined. By (ii), $P_k \to P_\infty$ in the strong operator topology. Since the limiting projector is non-trivial, so is the projector $P_k$. That the eigenvalue is isolated and of finite multiplicity comes from the structural assumptions on the operator $L_k$ (namely, that $\sigma_{\rm ess}(L_k) = \sigma_{\rm ess}(M_k)$, since $K_k$ is a compact perturbation).

(iiib) First, we collect two facts. If the operator $L_k$ has an eigenfunction $\phi_k$ with eigenvalue $\lambda_k$, then
\begin{equation}
(\lambda_k - M_k - K_k) \phi_k = 0 \, ,
\end{equation}
which we rewrite as
\begin{equation}
	\label{eq:eigenfunctionequation}
\phi_k - R(\lambda_k,M_k) K_k \phi_k = 0 \, .
\end{equation}
Heuristically, since $R(\lambda_k,M_k) K_k$ is compact, the equation~\eqref{eq:eigenfunctionequation} substitutes for the property that the eigenfunctions are ``more regular". Additionally, if $\lambda_k \to \lambda_\infty$, then we can exploit that the eigenvalues are stabilizing and\footnote{$\p_\lambda R(\lambda,M_k) = -R(\lambda,M_k)^2$}
\begin{equation}
	\label{eq:ResolventDifferenceEstHi}
\| R(\lambda,M_k) - R(\tilde{\lambda},M_k) \|_{X \to X} \leq \sup_\mu \| R(\mu,M_k)^2 \|_{X \to X} |\lambda - \tilde{\lambda}| \, 
\end{equation}
(where $\mu$ varies over a line segment from $\lambda$ to $\tilde{\lambda}$) to obtain
\begin{equation}
	\label{eq:convergenceofmkkk}
 R(\lambda_k,M_k) K_k \to R(\lambda_\infty,M_\infty) K_\infty \, .
\end{equation}

(Geometric simplicity) Suppose not, i.e., there exists a sequence of unstable eigenvalues $\lambda_k \to \lambda_\infty$, with corresponding eigenspace of dimension $\geq 2$. Then, by Riesz's lemma, there exist eigenfunctions $\phi_k^{(j)}$, $j=1,2$, $k \in \N$, satisfying
\begin{equation}
	\label{eq:keepme1}
\| \phi_k^{(j)} \|_{X} = 1
\end{equation}
\begin{equation}
	\label{eq:keepme2}
{\rm dist}(\phi_k^{(2)}, {\rm span} \, \phi_k^{(1)}) \geq 1/2 \, .
\end{equation}
The second property substitutes for orthogonality in the Banach space setting.
From the equation~\eqref{eq:eigenfunctionequation}, the convergence~\eqref{eq:convergenceofmkkk}, and the compactness of $R(\lambda_\infty,M_\infty)K_\infty$, there exists a subsequence in $k$ (not relabeled) such that $\phi^{(j)}_k \to \phi_\infty^{(j)}$ in $X$ ($j=1,2$) and $\phi_\infty^{(j)}$ is a limiting eigenfunction. By strong convergence,~\eqref{eq:keepme1}-\eqref{eq:keepme2} are kept with $k=+\infty$. This contradicts the geometric simplicity of $\lambda_\infty$.


(Algebraic simplicity) We have already established geometric simplicity, so we must eliminate possible Jordan blocks. Suppose not, i.e., there exists a sequence of unstable eigenvalues $\lambda_k \to \lambda_\infty$ with normalized eigenfunction $\phi_k$ and (normalized) generalized eigenfunction $\psi_k$, which we may take to satisfy
\begin{equation}
{\rm dist}(\psi_k, {\rm span} \, \phi_k) \geq 1/2 \, .
\end{equation}
 We have already established that $\phi_k \to \phi_\infty$ as $k \to +\infty$ and $\phi_\infty$ is a limiting eigenfunction. The generalized eigenfunction satisfies
\begin{equation}
(\lambda - M_k - K_k) \psi_k = \alpha_k \phi_k \, ,
\end{equation}
for a parameter $\alpha_k \neq 0$. We rewrite the equation as
\begin{equation}
\psi_k - R(\lambda_k,M_k) K_k \psi_k = \alpha_k R(\lambda_k,M_k) K_k \phi_k = \alpha_k \phi_k.
\end{equation}
$\alpha_k$ must remain bounded, since $\phi_k \to \phi_\infty$ with unit norm. Therefore, there exists a subsequence in $k$ (not relabeled) such that $\alpha_k \to \alpha_\infty$ (potentially zero) and  $R(\lambda_k,M_k) K_k \psi_k$ converges by compactness, so $\psi_k \to \psi_\infty$ in $X$ as $k \to +\infty$, and we deduce that $\psi_\infty$ is a (generalized) eigenfunction. This contradicts the algebraic simplicity of $\lambda_\infty$.
\end{proof}

\begin{lemma}[Stability of bounded invertibility]
	\label{lem:stabilityofboundedinvertibility}
Let $X$ be a Banach space, $M$ and $\tilde{M}$ be closed, densely defined operators on $X$, and $K, \tilde{K} : X \to X$ be compact. 

Suppose $\lambda \in \rho(L) \cap \rho(M) \cap \rho(\tilde{M})$. Then
\begin{equation}
[I - R(\lambda,M) K]^{-1} = R(\lambda,L) (\lambda - M) : D(M) \to D(M)
 \end{equation} 
extends to a bounded operator on $X$. Suppose
\begin{equation}
	\label{eq:assumptiononresolv}
	\begin{aligned}
&\| [R(\lambda,\tilde{M}) - R(\lambda,M)] K \|_{X \to X} + \| R(\lambda,\tilde{M}) (\tilde{K} - K) \|_{X \to X} \\
&\quad \leq \frac{1}{2 \| R(\lambda,L) (\lambda - M) \|_{X \to X} } \, .	
	\end{aligned}
\end{equation}
Then $\lambda \in \rho(\tilde{L})$, and
\begin{equation}
	\label{eq:resolventconclusion}
\| R(\lambda,\tilde{L}) \|_{X \to X} \leq 2 \| R(\lambda,L) (\lambda - M) \|_{X \to X} \| R(\lambda,\tilde{M}) \|_{X \to X} \, .
\end{equation}
\end{lemma}
\begin{proof}
We utilize the decomposition
\begin{equation}
\lambda - L = \lambda - M - K =  (\lambda - M) (I - R(\lambda,M) K) \, .
\end{equation}
By the assumption, $I - R(\lambda,M) K : D(M) \to D(M)$ is invertible (hence, invertible $X \to X$).\footnote{Since $K$ is compact, $I - R(\lambda,M) K$ is Fredholm of index zero, and any element of $\phi$ of the kernel would necessarily satisfy $\phi = R(\lambda,M) K \in D(M)$.} The analogous decomposition with tildes is
\begin{equation}
	\label{eq:equationforinverse}
\lambda - \tilde{L} = \lambda - \tilde{M} - \tilde{K} = (\lambda - \tilde{M}) (I - R(\lambda,\tilde{M}) \tilde{K}) \, .
\end{equation}
Since
\begin{equation}
\begin{aligned}
&\| (I - R(\lambda,\tilde{M}) \tilde{K}) - (I - R(\lambda,M) K) \|_{X \to X} \leq\\
&\quad \| [R(\lambda,\tilde{M}) - R(\lambda,M)] K \|_{X \to X} + \| R(\lambda,\tilde{M}) (\tilde{K} - K) \|_{X \to X}
\end{aligned}
\end{equation}
our assumption~\eqref{eq:assumptiononresolv} yields that $I - R(\lambda,\tilde{M}) \tilde{K} : X \to X$ is invertible (hence, invertible $D(\tilde{M}) \to D(\tilde{M})$), and
\begin{equation}
	\label{eq:neededforcontrol}
\| (I - R(\lambda,\tilde{M}) \tilde{K})^{-1} \|_{X \to X} \leq C \| R(\lambda,L) (\lambda - M) \|_{X \to X} \, .
\end{equation}
 We therefore control $R(\lambda,\tilde{L})$ using~\eqref{eq:equationforinverse} and \eqref{eq:neededforcontrol} to obtain~\eqref{eq:resolventconclusion}.
\end{proof}



We will now prove a perturbation result for an (algebraically simple) eigenvalue-eigenfunction pair $(\lambda_0,\phi_0)$. For this, it is somewhat more convenient to work with a continuous family of operators.

For the remainder of the section, we make the following standing assumptions: Let $H$ be a Hilbert space. Let $M : H \to H$ be bounded and $K : H \to H$ be compact. Let $A : D(A) \subset H \to H$ be closed and densely defined. Let
\begin{equation}
L_\varepsilon := M + K + \varepsilon A \, .
\end{equation}
That is, $L_0 : H \to H$ is bounded, whereas $L_\varepsilon : D(A) \subset H \to H$ is unbounded. We assume that there exist $\varepsilon_0 > 0$ and $\Lambda \in \R$ satisfying
\begin{equation}
R(\lambda,M+\varepsilon A) \text{ exists and } \| R(\lambda,M+\varepsilon A) \|_{H \to H} \lesssim 1 \, , \quad \forall \Re \lambda \geq \Lambda \, , \varepsilon \in (0,\varepsilon_0] \, .
\end{equation}

Under the above assumptions, the essential spectrum of $L_\varepsilon$ belongs to $\{ \Re \lambda < \Lambda \}$. The remainder of the spectrum is discrete, consisting of isolated eigenvalues of finite algebraic multiplicity, which may accumulate only on the essential spectrum. For any $\varepsilon'$ away from zero, $(\varepsilon'/\varepsilon) L_\varepsilon$, $|\varepsilon - \varepsilon'| \ll 1$, may be considered an analytic perturbation of $L_{\varepsilon'}$. Therefore, away from $\varepsilon = 0$, the discrete spectrum varies continuously and group eigenspaces vary analytically, as in Chapter 3 of \cite{kato2013perturbation}.

Finally, we assume
\begin{equation}
\begin{aligned}
&\forall x \in H , \;  \varepsilon \mapsto R(\lambda,M+\varepsilon A) x \text{ is continuous in } \varepsilon , \\
&\quad \text{ locally uniformly in } \lambda \in \{ \Re \lambda \geq \Lambda \} \, .
\end{aligned}
\end{equation}
(This is only a new assumption at $\varepsilon = 0$.)

Our primary interest is in the behavior of the spectrum at $\varepsilon = 0$. We suppose that $\lambda_0$ with $\Re \lambda_0 > \Lambda$ is an algebraically simple eigenvalue of $L_0$ and a normalized eigenfunction $\phi_0$. Let $P_{\lambda_0}$ be the associated spectral projection.

Below, we consider the perturbation theory of $(\lambda_0,\phi_0)$.

\begin{proposition}[Eigenvalue perturbation]
	\label{pro:eigenvalueperturbation}
Under the above assumptions, there exist $\varepsilon_0 > 0$ and a $\mathbb{C} \times H$-neighborhood $U$ of $(\lambda_0,\phi_0)$ such that, whenever $\varepsilon \in [0,\varepsilon_1]$, there exists a unique solution $(\lambda_\varepsilon,\phi_\varepsilon) \in U$ to the equations
\begin{equation}
	\label{eq:eigenfunctionequationphi}
(I - R(\lambda,M+\varepsilon A) K) \phi = 0
\end{equation}
\begin{equation}
	\label{eq:normalizationequation}
\langle P_{\lambda_0} \phi, \phi_0 \rangle = 1 \, .
\end{equation}
The map
\begin{equation}
\varepsilon \mapsto (\lambda_\varepsilon,\phi_\varepsilon) : [0,\varepsilon_1] \to \mathbb{C} \times H
\end{equation}
is continuous.
\end{proposition}
\begin{proof}
Solutions to the equations~\eqref{eq:eigenfunctionequationphi}-\eqref{eq:normalizationequation} are precisely zeros of the map
\begin{equation}
	\label{eq:maptolinearize}
(\phi,\lambda,\varepsilon) \mapsto \begin{bmatrix}
(I - R(\lambda,M+\varepsilon A) K) \phi \\
\langle P_{\lambda_0} \phi, \phi_0 \rangle - 1
\end{bmatrix} \in H \times \mathbb{C} \, .
\end{equation}
Then the conclusion will follow from the assumptions and Proposition~\ref{pro:IFT} (Implicit function theorem) once we verify that the linearization of~\eqref{eq:maptolinearize} at $(\phi_0,\lambda_0,0)$,
\begin{equation}
\begin{bmatrix}
I - R(\lambda_0,M) K & - D_\lambda|_{\lambda=\lambda_0} R(\lambda,M) K \phi_0 \\
\langle P_{\lambda_0} \cdot, \phi_0 \rangle & 0
\end{bmatrix} \, ,
\end{equation}
is invertible on $H \times \mathbb{C}$. Observe
\begin{equation}
- D_\lambda|_{\lambda=\lambda_0} (\lambda - M)^{-1} K \phi_0 = (\lambda_0 - M)^{-2} K \phi_0 = (\lambda_0 - M)^{-1} \phi_0 \, ,
\end{equation}
since $\phi_0 = (\lambda_0 - M)^{-1} K \phi_0$. To prove invertibility, it is more convenient to post-compose with the invertible operator\begin{equation}
\begin{bmatrix}
\lambda_0 - M & \\
& 1
\end{bmatrix} \, .
\end{equation}
The result is
\begin{equation}
\begin{bmatrix}
\lambda_0 - M - K & \phi_0 \\
\langle P_{\lambda_0} \cdot, \phi_0 \rangle & 0
\end{bmatrix} \, ,
\end{equation}
which is invertible on $H \times \mathbb{C}$, since $\lambda_0 - M - K : \ker P_{\lambda_0} \to \ker P_{\lambda_0}$ is invertible.
\end{proof}









\begin{corollary}[No two eigenvalues]
	\label{cor:notwoevals}
	Under the above assumptions, fix $0 < \delta \ll 1$ such that $\overline{B_\delta(\lambda_0)} \setminus \{ \lambda_0 \} \subset \varrho(L)$. Then, for $0 < \varepsilon \ll 1$, $L_\varepsilon$ has only the single eigenvalue $\lambda_\varepsilon$ in $B_\delta(\lambda_0)$.
\end{corollary}
\begin{proof}
Let $\lambda_\varepsilon \to \lambda_0$ be any sequence (we avoid writing $\varepsilon_k$) of eigenvalues in $B_\delta(\lambda_0)$. We normalize the corresponding eigenfunctions $\psi_\varepsilon$ to have unit norm. By the analysis in Proposition~\ref{pro:mainspectrallemma} (iiib), we must have that, along a subsequence, $\psi_\varepsilon \to c \phi_0$, $|c|=1$, as $\varepsilon \to 0^+$. We may further normalize $\psi_\varepsilon$ such that $c=1$. Then, for $\varepsilon$ sufficiently small, Proposition~\ref{pro:eigenvalueperturbation} ensures that $\psi_\varepsilon = \phi_\varepsilon$ and $\lambda_\varepsilon$ is the (unique) eigenvalue. Therefore, there cannot be two distinct eigenvalues in $B_\delta(\lambda_0)$.
\end{proof}

\begin{remark}[Variants]
One can shift $\Lambda$ (for example, to $\Lambda = 0$) by replacing $L$ by $L + \kappa I$. One can consider eigenvalues in a left half-plane $\{ \Re \lambda \leq \Lambda \}$ by replacing $L$ by~$-L$. 
One can consider unbounded $M$, as is done in~\cite{vishik2018instability1,albritton2023instability}

 
\end{remark}

Finally, we recall the implicit function theorem with only continuity required in the parameter (often, $C^1$ is required in all variables).\footnote{\cite{IFTMathoverflow} informed the preparation of this statement.}

\begin{proposition}[Implicit function theorem]
	\label{pro:IFT}
Let $X$, $Y$ be Banach spaces. Let $D$ be an open neighborhood of the origin in $X$. Let $U \subset \R^N$, $N \geq 1$, containing the origin. Suppose that $F(x,\varepsilon) : D \times U \to Y$ is continuous and
\begin{equation}
D_x F : D \times U \to \mathcal{B}(X,Y) \text{ exists and is continuous} \, ,
\end{equation}
\begin{equation}
F(0,0) = 0 \, , \quad D_x F(0,0) \text{ is invertible} \, .
\end{equation}
Then there exists an open neighborhood $D'$ of the origin in $X$ and $U' = U \cap B_\delta$ such that for every $\varepsilon \in U'$, there exists a unique solution $x(\varepsilon) \in D'$ to the equation
\begin{equation}
F(x,\varepsilon) = 0 \, .
\end{equation}
Furthermore,
\begin{equation}
\varepsilon \mapsto x(\varepsilon) : U' \to D' \text{ is continuous} \, .
\end{equation}
\end{proposition}



\begin{proof}
The equation to solve is
\begin{equation}
F(x,\varepsilon) =: D_x F(0,0) x + N(x,\varepsilon) = 0 \, .
\end{equation}
Notably, $N(0,0) = 0$ and $D_x N(0,0) = 0$. We rewrite the equation as
\begin{equation}
x + D_x F(0,0)^{-1} N(x,\varepsilon) = 0 \, .
\end{equation}
That is, our goal is to identify a fixed point for the mapping
\begin{equation}
\Phi_\varepsilon(x) = - D_x F(0,0)^{-1} N(x,\varepsilon) \, .
\end{equation}
Let $C_0 = \| D_x F(0,0)^{-1} \|_{Y \to X}$. First, we restrict $x$ and $\varepsilon$ to ensure that
\begin{equation}
\| D_x N(x,\varepsilon) \|_{X \to Y} \leq \frac{1}{2C_0} \, ,
\end{equation}
yielding, in particular, the contractivity $\| \Phi_\varepsilon(x) - \Phi_\varepsilon(x') \|_X \leq \| x-x'\|_X/2$. Then we observe
\begin{align}
\| N(x,\varepsilon) \|_X & \leq \| N(x,\varepsilon) - N(0,\varepsilon) \|_X + \| N(0,\varepsilon) \|_X \\ & \leq \frac{1}{2C_0 }\| x \|_X + o_{\varepsilon \to 0}(1)	\label{eq:Nthing}
\end{align}
yielding that $\Phi_\varepsilon$ stabilizes the ball $\bar{B}_\delta$ provided that $0 < \varepsilon,\delta \ll 1$ and $2 C_0 \delta \geq {o_\varepsilon \to 0}(1)$ in~\eqref{eq:Nthing}. The Banach fixed point theorem guarantees the existence and uniqueness of a fixed point $x(\varepsilon)$ in $\bar{B}_\delta$. Since we may choose $\delta \to 0^+$ as $\varepsilon \to 0$, the construction additionally guarantees the continuity of $x(\varepsilon)$ at $x=0$. For continuity at $\varepsilon_0$, we may apply the theorem after shifting $x(\varepsilon_0)$ and $\varepsilon_0$ to zero.
\end{proof}

%% file: forceremarks.tex
\section{Force perturbations}
\label{sec:forceperturbations}

Consider the $\varepsilon,\nu \to 0^+$ selection problem with $\nu$-dependent perturbations to the force. The topology in which the forces converge matters. Examples with quite low regularity forces can be exhibited by $2\frac{1}{2}$-dimensional flows~\cite{MR4595604,MR4799447} or convex integration.

\emph{There exists two families of forces $(f^\nu_{\pm})_{\nu \in (0,\nu_0)}$, converging strongly in supercritical spaces (weakly in critical spaces, with respect to the $L^\alpha = T$ scaling) to a single force $f^{\rm E}$, and whose solutions converge to non-unique Euler solutions $u^{\rm E}_{\pm}$, respectively.}

Let
\begin{equation}
    \tilde{U}(\xi,\tau) := \phi(\tau) \bar{U}(\xi) \, ,
\end{equation}
where $\phi$ is a smooth cut-off function satisfying $\phi \equiv 1$ on $[1,+\infty)$ and $\phi \equiv 0$ on $(-\infty,-1]$. We plug $\tilde{u}$ into the Navier-Stokes equations with unit viscosity:
\begin{equation}
    \p_t \tilde{u} - \Delta \tilde{u} =: \tilde{f} \, ,
\end{equation}
differing from $\p_t \bar{u} =: \bar{f}$ by $O(1)$ in the norm $\| f \|_{\rm crit} := \sup_{t \in \R_+} t \| \curl f \|_{L^{2/\alpha}}$, which is critical. However, the rescaled forces $\tilde{f}^\nu = U_\nu T_\nu^{-1} \tilde{f}(x/L_\nu,t/T_\nu)$ converge to $\bar{f}$ in supercritical spaces $L^1_t W^{1,p}_x(\R^2 \times (0,1))$, $p < 2/\alpha$, as $\nu \to 0^+$.




Define
\begin{equation}
    u^\nu_{\pm} = \tilde{u}^\nu \pm u^{\rm lin} \, ,
\end{equation}
according to the so-called Golovkin trick \cite{golovkin1964nonuniqueness, golovkin1965examples}. Then $u^\nu_\pm$ solves the Navier-Stokes equations with the force
\begin{equation}
    f^\nu_\pm := \tilde{f}^\nu - u^{\rm lin} \cdot \nabla u^{\rm lin} \pm (\tilde{u}^\nu - \bar{u}) \cdot  u^{\rm lin} \pm  u^{\rm lin} \cdot \nabla (\tilde{u}^\nu - \bar{u}) \pm \nu \Delta u^{\rm lin} \, .
\end{equation}
We have the flexibility to choose the growth rate $a \gg 1/\alpha$, so that the terms containing $u^{\rm lin}$ are quite smooth and converge to zero in high regularity spaces (except for the term $-u^{\rm lin} \cdot \nabla u^{\rm lin}$, which we build into the inviscid force). Then
\begin{equation}
    u^{\nu}_\pm \to u_\pm^{\rm E} := \bar{u} \pm  u^{\rm lin}
\end{equation}
\begin{equation}
    f^{\nu}_\pm \to f^{\rm E} := \bar{f} - u^{\rm lin} \cdot \nabla u^{\rm lin} \, ,
\end{equation}
and $u_\pm^{\rm E}$ solves the Euler equations with the force $f^{\rm E}$. The convergence is strong in supercritical spaces and weak in critical spaces (with respect to the $L^\alpha = T$ scaling).

%


%% file: main_new.bbl
\newcommand{\etalchar}[1]{$^{#1}$}
\begin{thebibliography}{CLFNLV19}

\bibitem[ABC22]{MR4429263}
Dallas Albritton, Elia Bru\'{e}, and Maria Colombo.
\newblock Non-uniqueness of {L}eray solutions of the forced {N}avier-{S}tokes equations.
\newblock {\em Ann. of Math. (2)}, 196(1):415--455, 2022.

\bibitem[ABC{\etalchar{+}}24]{albritton2023instability}
Dallas Albritton, Elia Bru\'e, Maria Colombo, Camillo De~Lellis, Vikram Giri, Maximilian Janisch, and Hyunju Kwon.
\newblock {\em Instability and non-uniqueness for the 2{D} {E}uler equations, after {M}. {V}ishik}, volume 219 of {\em Annals of Mathematics Studies}.
\newblock Princeton University Press, Princeton, NJ, 2024.

\bibitem[AC23]{MR4616679}
Dallas Albritton and Maria Colombo.
\newblock Non-uniqueness of {L}eray solutions to the hypodissipative {N}avier-{S}tokes equations in two dimensions.
\newblock {\em Comm. Math. Phys.}, 402(1):429--446, 2023.

\bibitem[AO23]{dallaswojtekvortexcolumns}
Dallas Albritton and Wojciech O{\.z}a{\'n}ski.
\newblock Linear and nonlinear instability of vortex columns.
\newblock {\em arXiv preprint arXiv:2310.20674}, 2023.

\bibitem[AV23]{armstrong2023anomalous}
Scott Armstrong and Vlad Vicol.
\newblock Anomalous diffusion by fractal homogenization.
\newblock {\em arXiv preprint arXiv:2305.05048}, 2023.

\bibitem[BC23]{MR4645737}
Elia Bru\'e and Maria Colombo.
\newblock Nonuniqueness of solutions to the {E}uler equations with vorticity in a {L}orentz space.
\newblock {\em Comm. Math. Phys.}, 403(2):1171--1192, 2023.

\bibitem[BCC22]{MR4496901}
Paolo Bonicatto, Gennaro Ciampa, and Gianluca Crippa.
\newblock On the advection-diffusion equation with rough coefficients: weak solutions and vanishing viscosity.
\newblock {\em J. Math. Pures Appl. (9)}, 167:204--224, 2022.

\bibitem[BCC{\etalchar{+}}24]{MR4799447}
Elia Bru\'e, Maria Colombo, Gianluca Crippa, Camillo De~Lellis, and Massimo Sorella.
\newblock Onsager critical solutions of the forced {N}avier-{S}tokes equations.
\newblock {\em Commun. Pure Appl. Anal.}, 23(10):1350--1366, 2024.

\bibitem[BCK24a]{BruColKum24b}
Elia Bru{\`e}, Maria Colombo, and Anuj Kumar.
\newblock Flexibility of two-dimensional {E}uler flows with integrable vorticity, 2024.
\newblock cvgmt preprint.

\bibitem[BCK24b]{BruColKum24a}
Elia Bru{\`e}, Maria Colombo, and Anuj Kumar.
\newblock Sharp nonuniqueness in the transport equation with {S}obolev velocity field, 2024.
\newblock cvgmt preprint.

\bibitem[BDL23]{MR4595604}
Elia Bru\`e and Camillo De~Lellis.
\newblock Anomalous dissipation for the forced 3{D} {N}avier-{S}tokes equations.
\newblock {\em Comm. Math. Phys.}, 400(3):1507--1533, 2023.

\bibitem[BDLSV19]{OnsagerAdmissible}
Tristan Buckmaster, Camillo De~Lellis, L\'{a}szl\'{o} Sz\'{e}kelyhidi, Jr., and Vlad Vicol.
\newblock Onsager's conjecture for admissible weak solutions.
\newblock {\em Comm. Pure Appl. Math.}, 72(2):229--274, 2019.

\bibitem[BGK98]{Bernard1998}
Denis Bernard, Krzysztof Gawedzki, and Antti Kupiainen.
\newblock Slow modes in passive advection.
\newblock {\em Journal of Statistical Physics}, 90(3–4):519–569, February 1998.

\bibitem[BGS02]{bardosguostrauss}
C.~Bardos, Y.~Guo, and W.~Strauss.
\newblock Stable and unstable ideal plane flows.
\newblock {\em Chinese Ann. Math. Ser. B}, 23(2):149--164, 2002.
\newblock Dedicated to the memory of Jacques-Louis Lions.

\bibitem[BM24]{MR4725248}
Miriam Buck and Stefano Modena.
\newblock Non-uniqueness and energy dissipation for 2{D} {E}uler equations with vorticity in {H}ardy spaces.
\newblock {\em J. Math. Fluid Mech.}, 26(2):Paper No. 26, 39, 2024.

\bibitem[Bre11]{MR2759829}
Haim Brezis.
\newblock {\em Functional analysis, {S}obolev spaces and partial differential equations}.
\newblock Universitext. Springer, New York, 2011.

\bibitem[BS21]{MR4182316}
Alberto Bressan and Wen Shen.
\newblock A posteriori error estimates for self-similar solutions to the {E}uler equations.
\newblock {\em Discrete Contin. Dyn. Syst.}, 41(1):113--130, 2021.

\bibitem[BSJW23]{burczak2023anomalous}
Jan Burczak, L{\'a}szl{\'o} Sz{\'e}kelyhidi~Jr, and Bian Wu.
\newblock Anomalous dissipation and {E}uler flows.
\newblock {\em arXiv preprint arXiv:2310.02934}, 2023.

\bibitem[BV19]{BuckmasterVicolAnnals}
Tristan Buckmaster and Vlad Vicol.
\newblock Nonuniqueness of weak solutions to the {N}avier-{S}tokes equation.
\newblock {\em Ann. of Math. (2)}, 189(1):101--144, 2019.

\bibitem[BV20]{BuckmasterPhenomenologies2020}
Tristan Buckmaster and Vlad Vicol.
\newblock Convex integration constructions in hydrodynamics.
\newblock {\em Bulletin of the American Mathematical Society}, 58(1):1–44, November 2020.

\bibitem[CCS20]{MR4037474}
Gennaro Ciampa, Gianluca Crippa, and Stefano Spirito.
\newblock Smooth approximation is not a selection principle for the transport equation with rough vector field.
\newblock {\em Calc. Var. Partial Differential Equations}, 59(1):Paper No. 13, 21, 2020.

\bibitem[CCS23]{MR4662772}
Maria Colombo, Gianluca Crippa, and Massimo Sorella.
\newblock Anomalous dissipation and lack of selection in the {O}bukhov-{C}orrsin theory of scalar turbulence.
\newblock {\em Ann. PDE}, 9(2):Paper No. 21, 48, 2023.

\bibitem[CFLS16]{MR3551263}
A.~Cheskidov, M.~C.~Lopes Filho, H.~J.~Nussenzveig Lopes, and R.~Shvydkoy.
\newblock Energy conservation in two-dimensional incompressible ideal fluids.
\newblock {\em Comm. Math. Phys.}, 348(1):129--143, 2016.

\bibitem[CFMS24]{castro2024proof}
{\'A}ngel Castro, Daniel Faraco, Francisco Mengual, and Marcos Solera.
\newblock {A proof of {V}ishik's nonuniqueness {T}heorem for the forced 2{D} {E}uler equation}.
\newblock {\em arXiv preprint arXiv:2404.15995}, 2024.

\bibitem[CFMS25]{castro2025unstable}
{\'A}ngel Castro, Daniel Faraco, Francisco Mengual, and Marcos Solera.
\newblock Unstable vortices, sharp non-uniqueness with forcing, and global smooth solutions for the {S}{Q}{G} equation.
\newblock {\em arXiv preprint arXiv:2502.10274}, 2025.

\bibitem[CG22]{Chaturvedi2022}
Sanchit Chaturvedi and Cole Graham.
\newblock The inviscid limit of viscous {B}urgers at nondegenerate shock formation.
\newblock {\em Annals of PDE}, 9(1), December 2022.

\bibitem[CLFNLV19]{constantin2019vorticity}
Peter Constantin, Milton~C Lopes~Filho, Helena~J Nussenzveig~Lopes, and Vlad Vicol.
\newblock Vorticity measures and the inviscid limit.
\newblock {\em Archive for Rational Mechanics and Analysis}, 234:575--593, 2019.

\bibitem[CS15]{MR3366055}
Gianluca Crippa and Stefano Spirito.
\newblock Renormalized solutions of the 2{D} {E}uler equations.
\newblock {\em Comm. Math. Phys.}, 339(1):191--198, 2015.

\bibitem[CS24]{CriSte21}
Gianluca Crippa and Giorgio Stefani.
\newblock An elementary proof of existence and uniqueness for the {E}uler flow in localized {Y}udovich spaces.
\newblock {\em Calc. Var. Partial Differential Equations}, 63(7):168, 31 pp., 2024.
\newblock cvgmt preprint.

\bibitem[DEIJ22]{MR4381138}
Theodore~D. Drivas, Tarek~M. Elgindi, Gautam Iyer, and In-Jee Jeong.
\newblock Anomalous dissipation in passive scalar transport.
\newblock {\em Arch. Ration. Mech. Anal.}, 243(3):1151--1180, 2022.

\bibitem[Del91]{delort1991existence}
Jean-Marc Delort.
\newblock Existence de nappes de tourbillon en dimension deux.
\newblock {\em Journal of the American Mathematical Society}, 4(3):553--586, 1991.

\bibitem[DL89]{DiPerna1989}
R.~J. DiPerna and P.~L. Lions.
\newblock Ordinary differential equations, transport theory and {S}obolev spaces.
\newblock {\em Inventiones Mathematicae}, 98(3):511–547, October 1989.

\bibitem[DLG22]{MR4396067}
Camillo De~Lellis and Vikram Giri.
\newblock Smoothing does not give a selection principle for transport equations with bounded autonomous fields.
\newblock {\em Ann. Math. Qu\'e.}, 46(1):27--39, 2022.

\bibitem[DLSJ22]{de2022weak}
Camillo De~Lellis and L{\'a}szl{\'o} Sz{\'e}kelyhidi~Jr.
\newblock Weak stability and closure in turbulence.
\newblock {\em Philosophical Transactions of the Royal Society A}, 380(2218):20210091, 2022.

\bibitem[DM87]{diperna1987concentrations}
Ronald~J DiPerna and Andrew~J Majda.
\newblock Concentrations in regularizations for 2-d incompressible flow.
\newblock {\em Communications on Pure and Applied Mathematics}, 40(3):301--345, 1987.

\bibitem[DM24]{dolce2024selfsimilarinstabilityforcednonuniqueness}
Michele Dolce and Giulia Mescolini.
\newblock Self-similar instability and forced nonuniqueness: an application to the 2{D} {E}uler equations.
\newblock {\em arXiv preprint arXiv:2411.18452}, 2024.

\bibitem[DRP24]{de2024no}
Luigi De~Rosa and Jaemin Park.
\newblock No anomalous dissipation in two-dimensional incompressible fluids.
\newblock {\em arXiv preprint arXiv:2403.04668}, 2024.

\bibitem[DS09]{delellis2009euler}
C.~{De Lellis} and L.~{Sz{\'e}kelyhidi Jr.}
\newblock The {E}uler equations as a differential inclusion.
\newblock {\em Annals of Mathematics (2)}, 170(3):1417--1436, 2009.

\bibitem[DS13]{delellis2013dissipative}
C.~{De Lellis} and L.~{Sz{\'e}kelyhidi Jr.}
\newblock Dissipative continuous {E}uler flows.
\newblock {\em Inventiones Mathematicae}, 193(2):377--407, 2013.

\bibitem[ED14]{Eyink2014}
Gregory~L. Eyink and Theodore~D. Drivas.
\newblock Spontaneous stochasticity and anomalous dissipation for {B}urgers equation.
\newblock {\em Journal of Statistical Physics}, 158(2):386–432, October 2014.

\bibitem[EL24]{elgindi2024norm}
Tarek~M Elgindi and Kyle Liss.
\newblock Norm growth, non-uniqueness, and anomalous dissipation in passive scalars.
\newblock {\em Archive for Rational Mechanics and Analysis}, 248(6):120, 2024.

\bibitem[EN00]{engel2000one}
Klaus-Jochen Engel and Rainer Nagel.
\newblock {\em One-Parameter Semigroups for linear evolution equations}.
\newblock Springer-Verlag, 2000.

\bibitem[Eyi24]{EyinkSurvey2024}
Gregory Eyink.
\newblock Onsager’s `ideal turbulence' theory.
\newblock {\em Journal of Fluid Mechanics}, 988, May 2024.

\bibitem[FSV97]{friedlanderstraussvishikearly}
S.~Friedlander, W.~Strauss, and M.~Vishik.
\newblock Nonlinear instability in an ideal fluid.
\newblock {\em Ann. Inst. H. Poincar\'{e} Anal. Non Lin\'{e}aire}, 14(2):187--209, 1997.

\bibitem[Ghi86]{MR0851146}
Jean-Michel Ghidaglia.
\newblock Some backward uniqueness results.
\newblock {\em Nonlinear Anal.}, 10(8):777--790, 1986.

\bibitem[GKN23]{giri20233}
Vikram Giri, Hyunju Kwon, and Matthew Novack.
\newblock The ${L}^3$-based strong {O}nsager theorem.
\newblock {\em arXiv preprint arXiv:2305.18509}, 2023.

\bibitem[Gol64]{golovkin1964nonuniqueness}
Kirill~Kapitonovich Golovkin.
\newblock Nonuniqueness of the solutions of certain boundary problems for the equations of hydromechanics.
\newblock {\em USSR Computational Mathematics and Mathematical Physics}, 4(4):212--215, 1964.

\bibitem[Gol65]{golovkin1965examples}
Kirill~Kapitonovich Golovkin.
\newblock Examples of the non-uniqueness and low stability of solutions of the equations of hydrodynamics.
\newblock {\em USSR Computational Mathematics and Mathematical Physics}, 5(4):115--132, 1965.

\bibitem[GR24]{giriradu}
Vikram Giri and R{\u a}zvan-Octavian Radu.
\newblock The {O}nsager conjecture in 2{D}: a {N}ewton-{N}ash iteration.
\newblock {\em Inventiones mathematicae}, 238(2):691--768, 2024.

\bibitem[GS19]{gallay2019linear}
Thierry Gallay and Didier Smets.
\newblock On the linear stability of vortex columns in the energy space.
\newblock {\em Journal of Mathematical Fluid Mechanics}, 21(4):48, 2019.

\bibitem[GS20]{Gallay2020}
Thierry Gallay and Didier Smets.
\newblock Spectral stability of inviscid columnar vortices.
\newblock {\em Analysis \& PDE}, 13(6):1777–1832, September 2020.

\bibitem[GV00]{Gawdzki2000}
Krzysztof Gawedzki and Massimo Vergassola.
\newblock Phase transition in the passive scalar advection.
\newblock {\em Physica D: Nonlinear Phenomena}, 138(1–2):63–90, April 2000.

\bibitem[Gv23]{guillodsverak}
Julien Guillod and Vladim{\'i}r \v{S}ver{\'a}k.
\newblock Numerical investigations of non-uniqueness for the {N}avier–{S}tokes initial value problem in borderline spaces.
\newblock {\em Journal of Mathematical Fluid Mechanics}, 25(3), May 2023.

\bibitem[Ise18]{isett2018proof}
P.~Isett.
\newblock A {P}roof of {O}nsager's {C}onjecture.
\newblock {\em Annals of Mathematics (2)}, 188(3):871--963, 2018.

\bibitem[J{\v S}15]{jiasverakillposed}
Hao Jia and Vladim{\'i}r {\v S}ver{\'a}k.
\newblock Are the incompressible 3d {N}avier-{S}tokes equations locally ill-posed in the natural energy space?
\newblock {\em J. Funct. Anal.}, 268(12):3734--3766, 2015.

\bibitem[Kat13]{kato2013perturbation}
Tosio Kato.
\newblock {\em Perturbation theory for linear operators}, volume 132.
\newblock Springer Science \& Business Media, 2013.

\bibitem[KMW24]{kalinin2024scale}
Konstantin Kalinin, Govind Menon, and Bian Wu.
\newblock Scale invariant bounds for the {K}elvin-{H}elmholtz instability.
\newblock {\em arXiv preprint arXiv:2410.14557}, 2024.

\bibitem[KP88]{kato1988commutator}
Tosio Kato and Gustavo Ponce.
\newblock Commutator estimates and the {E}uler and {N}avier-{S}tokes equations.
\newblock {\em Communications on Pure and Applied Mathematics}, 41(7):891--907, 1988.

\bibitem[Kru70]{kruvzkov1970first}
Stanislav~N Kru{\v{z}}kov.
\newblock First order quasilinear equations in several independent variables.
\newblock {\em Mathematics of the USSR-Sbornik}, 10(2):217, 1970.

\bibitem[LFMNL06]{lopes2006weak}
Milton~C Lopes~Filho, Anna~L Mazzucato, and Helena~J Nussenzveig~Lopes.
\newblock Weak solutions, renormalized solutions and enstrophy defects in 2{D} turbulence.
\newblock {\em Archive for rational mechanics and analysis}, 179:353--387, 2006.

\bibitem[LMPP21]{lanthaler2021conservation}
Samuel Lanthaler, Siddhartha Mishra, and Carlos Par{\'e}s-Pulido.
\newblock On the conservation of energy in two-dimensional incompressible flows.
\newblock {\em Nonlinearity}, 34(2):1084, 2021.

\bibitem[Mai16]{Mailybaev2016}
Alexei~A. Mailybaev.
\newblock Spontaneous stochasticity of velocity in turbulence models.
\newblock {\em Multiscale Modeling \& Simulation}, 14(1):96–112, January 2016.

\bibitem[Maj]{IFTMathoverflow}
Pietro Majer.
\newblock Implicit function theorem with continuous dependence on parameter.
\newblock MathOverflow.
\newblock URL:https://mathoverflow.net/q/409166 (version: 2021-11-23).

\bibitem[Men24]{MR4788287}
Francisco Mengual.
\newblock Non-uniqueness of admissible solutions for the 2{D} {E}uler equation with {$L^p$} vortex data.
\newblock {\em Comm. Math. Phys.}, 405(9):Paper No. 207, 28, 2024.

\bibitem[MPS25]{mescolini2025vanishing}
Giulia Mescolini, Jules Pitcho, and Massimo Sorella.
\newblock On vanishing diffusivity selection for the advection equation.
\newblock {\em Annali di Matematica Pura ed Applicata (1923-)}, pages 1--21, 2025.

\bibitem[Ons49]{OnsagerStatisticalHydrodynamics1949}
L.~Onsager.
\newblock Statistical hydrodynamics.
\newblock {\em Il Nuovo Cimento}, 6(S2):279–287, March 1949.

\bibitem[PDS14]{palmer2014real}
TN~Palmer, A~D{\"o}ring, and G~Seregin.
\newblock The real butterfly effect.
\newblock {\em Nonlinearity}, 27(9):R123, 2014.

\bibitem[Sch93]{Scheffer1993}
Vladimir Scheffer.
\newblock An inviscid flow with compact support in space-time.
\newblock {\em Journal of Geometric Analysis}, 3(4):343–401, July 1993.

\bibitem[Shn97]{shnirelman1997nonuniqueness}
Alexander Shnirelman.
\newblock On the nonuniqueness of weak solution of the {E}uler equation.
\newblock {\em Communications on Pure and Applied Mathematics: A Journal Issued by the Courant Institute of Mathematical Sciences}, 50(12):1261--1286, 1997.

\bibitem[Sz{\'e}11]{Szkelyhidi2011}
Laszlo Sz{\'e}kelyhidi, Jr.
\newblock Weak solutions to the incompressible {E}uler equations with vortex sheet initial data.
\newblock {\em Comptes Rendus. Math{\'e}matique}, 349(19–20):1063–1066, September 2011.

\bibitem[TBM20]{thalabard2020butterfly}
Simon Thalabard, J{\'e}r{\'e}mie Bec, and Alexei~A Mailybaev.
\newblock From the butterfly effect to spontaneous stochasticity in singular shear flows.
\newblock {\em Communications Physics}, 3(1):1--8, 2020.

\bibitem[VF03]{friedlandervishik}
M.~Vishik and S.~Friedlander.
\newblock Nonlinear instability in two dimensional ideal fluids: {T}he case of a dominant eigenvalue.
\newblock {\em Comm. Math. Phys.}, 243(2):261--273, 2003.

\bibitem[Vis18a]{vishik2018instability1}
Misha Vishik.
\newblock Instability and non-uniqueness in the {C}auchy problem for the {E}uler equations of an ideal incompressible fluid. {P}art {I}.
\newblock {\em arXiv preprint arXiv:1805.09426}, 2018.

\bibitem[Vis18b]{vishik2018instability2}
Misha Vishik.
\newblock Instability and non-uniqueness in the {C}auchy problem for the {E}uler equations of an ideal incompressible fluid. {P}art {I}{I}.
\newblock {\em arXiv preprint arXiv:1805.09440}, 2018.

\bibitem[VY23]{Vasseur2023}
Alexis~F. Vasseur and Jincheng Yang.
\newblock Boundary vorticity estimates for {N}avier-{S}tokes and application to the inviscid limit.
\newblock {\em SIAM Journal on Mathematical Analysis}, 55(4):3081–3107, July 2023.

\bibitem[Wu21]{Wu2021}
Bian Wu.
\newblock Partially regular weak solutions of the {N}avier-{S}tokes equations in $\mathbb{R}^4 \times [0, \infty [$.
\newblock {\em Archive for Rational Mechanics and Analysis}, 239(3):1771–1808, January 2021.

\bibitem[Yud63]{yudovich1963non}
Victor~I Yudovich.
\newblock Non-stationary flow of an ideal incompressible liquid.
\newblock {\em USSR Computational Mathematics and Mathematical Physics}, 3(6):1407--1456, 1963.

\end{thebibliography}
